\documentclass[reqno,dvipsnames]{amsart}
\pdfoutput=1

\usepackage{amssymb, amsmath}
\usepackage{mathrsfs}
\usepackage{amscd} 
\usepackage{amsmath}
\usepackage[english]{babel}
\usepackage{microtype}               
\usepackage[T1]{fontenc}
\usepackage{lmodern}      

\usepackage{hyperref}
\hypersetup{
	colorlinks   = true,          
	urlcolor     = blue,          
	linkcolor    = blue,          
	citecolor   = red             
}

\usepackage{enumitem}
\setenumerate[1]{label=(\roman*)}    

\usepackage[normalem]{ulem}
\usepackage{graphicx}
\usepackage[curve]{xypic}
\usepackage{tikz}
\usetikzlibrary{arrows,decorations,shapes,shadows}
\usetikzlibrary{calc,intersections}
\usepackage{verbatim}

\newbox\mybox
\def\overtag#1#2#3{\setbox\mybox\hbox{$#1$}\hbox to
  0pt{\vbox to 0pt{\vglue-#3\vglue-\ht\mybox\hbox to \wd\mybox
      {\hss$\ss#2$\hss}\vss}\hss}\box\mybox}
\def\undertag#1#2#3{\setbox\mybox\hbox{$#1$}\hbox to 0pt{\vbox to
    0pt{\vglue#3\vglue\ht\mybox\hbox to \wd\mybox
      {\hss$\ss#2$\hss}\vss}\hss}\box\mybox}
\def\lefttag#1#2#3{\hbox to 0pt{\vbox to 0pt{\vglue -6pt\hbox to
      0pt{\hss$\ss#2$\hskip#3}\vss}}#1}
\def\righttag#1#2#3{\hbox to 0pt{\vbox to 0pt{\vglue -6pt\hbox to
      0pt{\hskip#3$\ss#2$\hss}\vss}}#1}
\let\ss\scriptstyle
\def\splicediag#1#2{\xymatrix@R=#1pt@C=#2pt@M=0pt@W=0pt@H=0pt}
\def\Dot{\lower.2pc\hbox to 2pt{\hss$\bullet$\hss}}
\def\Circ{\lower.2pc\hbox to 2pt{\hss$\circ$\hss}}
\def\Vdots{\raise5pt\hbox{$\vdots$}}
\newcommand\lineto{\ar@{-}}
\newcommand\dashto{\ar@{--}}
\newcommand\dotto{\ar@{.}}

\let\cal\mathcal
\renewcommand{\setminus}{\smallsetminus}
\newcommand\Q{{\mathbb Q}}
\newcommand\R{{\mathbb R}}
\newcommand\C{{\mathbb C}}
\newcommand\Z{{\mathbb Z}}
\newcommand\N{{\mathbb N}}

\newcommand\calI{\mathcal{I}}
\newcommand\calO{\mathcal{O}}

\DeclareMathOperator{\NL}{NL}
\DeclareMathOperator{\Div}{Div}
\DeclareMathOperator{\val}{val}
\DeclareMathOperator{\ord}{ord}
\DeclareMathOperator{\length}{length}
\DeclareMathOperator{\St}{St}
\DeclareMathOperator{\mult}{mult}

\newtheorem{theorem}{Theorem}[section]
\newtheorem{proposition}[theorem]{Proposition}
\newtheorem*{theorem*}{Theorem}
\newtheorem{corollary}[theorem]{Corollary}
\newtheorem{lemma}[theorem]{Lemma}

\theoremstyle{definition}

\newtheorem*{amalgamation*}{Amalgamation}
\newtheorem{example}[theorem]{Example}
\newtheorem{remark}[theorem]{Remark}
\newtheorem*{remark*}{Remark}
\newtheorem{definition}[theorem]{Definition}

\renewcommand{\int}{\operatorname{int}}
\newcommand{\lcm}{\operatorname{lcm}}

\usepackage{xcolor}

\begin{document}

\title{Inner geometry of complex surfaces: a valuative approach}

\author{Andr\'e Belotto da Silva}
\address{Aix Marseille Univ, CNRS, Centrale Marseille, I2M, Marseille, France}
\email{\href{mailto:andre-ricardo.BELOTTO-DA-SILVA@univ-amu.fr}{andre-ricardo.belotto-da-silva@univ-amu.fr}}
\urladdr{\url{https://andrebelotto.com}}

\author{Lorenzo Fantini}
\address{Goethe-Universit\"at Frankfurt, Institut f\"ur Mathematik, Frankfurt am Main, Germany}
\email{\href{mailto:fantini@math.uni-frankfurt.de}{fantini@math.uni-frankfurt.de}}
\urladdr{\url{https://lorenzofantini.eu/}}

\author{Anne Pichon}
\address{Aix Marseille Univ, CNRS, Centrale Marseille, I2M, Marseille, France}
\email{\href{mailto:anne.pichon@univ-amu.fr}{anne.pichon@univ-amu.fr}}
\urladdr{\url{http://iml.univ-mrs.fr/~pichon/}}

\begin{abstract}
Given a complex analytic germ $(X, 0)$ in $(\mathbb C^n, 0)$, the standard Hermitian metric of $\mathbb C^n$ induces a natural arc-length metric on $(X, 0)$, called the inner metric. 
We study the inner metric structure of the germ of an isolated complex surface singularity $(X,0)$ by means of an infinite family of   numerical analytic invariants, called inner rates.  
Our main result is a formula for the Laplacian of the inner rate function on a space of valuations, the non-archimedean link of $(X,0)$.  
We deduce in particular that the global data consisting of the topology of $(X,0)$, together with the configuration of a generic hyperplane section and of the polar curve of a generic plane projection of $(X,0)$, completely determine all the inner rates on $(X,0)$, and hence the local metric structure of the germ.
Several other applications of our formula are discussed in the paper.
\end{abstract}


 \maketitle

 
 \section{Introduction}
 
{\renewcommand*{\thetheorem}{\Alph{theorem}}
 
Given a complex analytic germ  $(X,0) \subset (\C^n,0)$ at the origin of $\C^n$, the standard Hermitian metric of $\C^n$ induces a natural metric on $(X,0)$ by measuring the lengths of arcs on $(X,0)$.
This metric, called the \emph{inner metric}, has been widely studied in several different contexts.
For example, Bernig and Lytchak \cite{BernigLytchak2007} use it to introduce a notion of metric tangent cone, while Hsiang and Pati \cite{HsiangPati1985} and Nagase \cite{Nagase1989} use it to study the the $L^2$-cohomology of singular algebraic surfaces, following original ideas of Cheeger \cite{Cheeger1979}.

The study of the inner metric is motivated by the following fact: while it is well known that, for $\epsilon >0$ sufficiently small, $X$ is locally homeomorphic to the cone over its link $X^{(\epsilon)} = X \cap S_{\epsilon}$, where $S_{\epsilon}$ denotes the sphere centered at $0$ with radius $\epsilon$ in $\C^n$, in general the metric germ $(X,0)$ is not metrically conical.
Indeed, there are parts of its link $X^{(\epsilon)}$ whose diameters with respect to the inner metric shrink faster than linearly when approaching the singular point. 
It is then natural to study how $X^{(\epsilon)}$ behaves metrically when approaching the origin.

In this paper, we study the metric structure of the germ of an isolated complex surface singularity $(X,0)$ by means of an infinite family of numerical analytic invariants, called \emph{inner rates}.
A resolution of $(X,0)$ is \emph{good} if its exceptional locus is a normal crossing divisor.
Given a good resolution $\pi \colon X_{\pi} \to X$ of $(X,0)$ that factors through the blowup of the maximal ideal and through the Nash modification of $(X,0)$, an irreducible component $E$ of the exceptional divisor $\pi^{-1}(0)$ of $\pi$, and a small neighborhood $\cal N(E)$ of $E$ in $X_\pi$ with a neighborhood of each double point of $\pi^{-1}(0)$ removed, the inner rate $q_E$ of $E$ is a rational number that measures how fast the subset $\pi\big(\cal N(E)\big)$ of $(X,0)$ shrinks when approaching the origin (see Figure~\ref{figure:introduction} for a pictorial explanation and Definition~\ref{definition:inner_rate} for a precise definition). 
The inner rate $q_E$
is independent of the choice of an embedding of $(X,0)$ into a smooth germ, only depending on the analytic type of $(X,0)$ 
and on the divisor $E$.

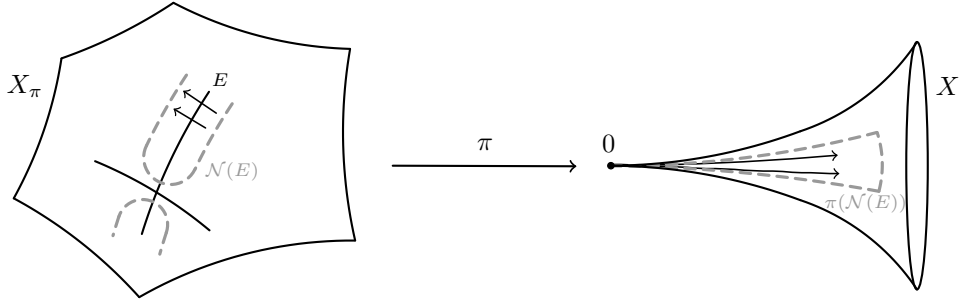
\begin{figure}[h] 
	\definecolor{wwwwww}{rgb}{0.6,0.6,0.6}  
	\begin{tikzpicture}[line cap=round,line join=round,x=1cm,y=1cm,scale=.82] 
	\clip(-7.8,-6.15) rectangle (8,-1.3); 
	\draw [->,shift={(20.533269590563883,435.45528635933516)},line width=0.6pt] 
	plot[domain=4.670325079481537:4.67879,variable=\t]({1*439.8443530756736*cos(\t r)+0*439.8443530756736*sin(\t r)},{0*439.8443530756736*cos(\t r)+1*439.8443530756736*sin(\t r)}); 
	\draw [->,shift={(1.7903967258481908,41.514388674119466)},line width=0.6pt]  plot[domain=4.717810490300544:4.799,variable=\t]({1*45.515057579285234*cos(\t r)+0*45.515057579285234*sin(\t r)},{0*45.515057579285234*cos(\t r)+1*45.515057579285234*sin(\t r)}); 
	%
	\draw [shift={(1.7474102853789522,-30.03867770068142)},line width=1.2pt,dash pattern=on 4pt off 3pt,color=wwwwww]  plot[domain=1.391856295407007:1.559669279064335,variable=\t]({1*26.040289723694034*cos(\t r)+0*26.040289723694034*sin(\t r)},{0*26.040289723694034*cos(\t r)+1*26.040289723694034*sin(\t r)});
	\draw [shift={(2.1453073608375393,12.661596119393295)},line width=1.2pt,dash pattern=on 4pt off 3pt,color=wwwwww]  plot[domain=4.705898006282137:4.967721759664982,variable=\t]({1*16.661947124943925*cos(\t r)+0*16.661947124943925*sin(\t r)},{0*16.661947124943925*cos(\t r)+1*16.661947124943925*sin(\t r)});
	\draw [shift={(4.740679075007138,-3.986053931469311)},line width=1.2pt,dash pattern=on 4pt off 3pt,color=wwwwww]  plot[domain=-0.25511974469180654:0.3151875140213137,variable=\t]({1*1.6964639181215802*cos(\t r)+0*1.6964639181215802*sin(\t r)},{0*1.6964639181215802*cos(\t r)+1*1.6964639181215802*sin(\t r)});
	\draw [shift={(3.934017779439453,0.05045630189504486)},line width=0.8pt]  plot[domain=5.037107471322216:5.695435151144759,variable=\t]({1*3.684232102547742*cos(\t r)+0*3.684232102547742*sin(\t r)},{0*3.684232102547742*cos(\t r)+1*3.684232102547742*sin(\t r)});
	\draw [rotate around={90:(7,-4)},line width=0.8pt] (7,-4) ellipse (2.0075859661541795cm and 0.17436000544584215cm);
	\draw [shift={(1.9956512084169284,4.953921162247546)},line width=0.8pt]  plot[domain=4.717024306805295:5.067563026195491,variable=\t]({1*8.954017356207599*cos(\t r)+0*8.954017356207599*sin(\t r)},{0*8.954017356207599*cos(\t r)+1*8.954017356207599*sin(\t r)});
	\draw [shift={(1.9956512084169333,-12.953921162247497)},line width=0.8pt]  plot[domain=1.2156222809840949:1.5661610003742918,variable=\t]({1*8.95401735620757*cos(\t r)+0*8.95401735620757*sin(\t r)},{0*8.95401735620757*cos(\t r)+1*8.95401735620757*sin(\t r)});
	\draw [->,line width=0.8pt] (-1.5,-4) -- (1.4784685399818294,-4.009178139065214);  
	\draw [shift={(3.9340177794394577,-8.050456301895018)},line width=0.8pt]  plot[domain=0.5877501560348264:1.246077835857371,variable=\t]({1*3.6842321025477345*cos(\t r)+0*3.6842321025477345*sin(\t r)},{0*3.6842321025477345*cos(\t r)+1*3.6842321025477345*sin(\t r)});
	\draw [shift={(5.6626065823743845,-3.453433886162705)},line width=0.8pt]  plot[domain=2.971587878520733:3.364364854915331,variable=\t]({1*7.969056769555354*cos(\t r)+0*7.969056769555354*sin(\t r)},{0*7.969056769555354*cos(\t r)+1*7.969056769555354*sin(\t r)});
	\draw [shift={(-2.162400475654458,-12.14080740026628)},line width=0.8pt]  plot[domain=1.5631628872774035:2.0916355616130495,variable=\t]({1*6.926938422096639*cos(\t r)+0*6.926938422096639*sin(\t r)},{0*6.926938422096639*cos(\t r)+1*6.926938422096639*sin(\t r)});
	\draw [shift={(-1.993120370036868,4.511268593475872)},line width=0.8pt]  plot[domain=4.231494967510567:4.682404871611727,variable=\t]({1*6.61941659093236*cos(\t r)+0*6.61941659093236*sin(\t r)},{0*6.61941659093236*cos(\t r)+1*6.61941659093236*sin(\t r)});
	\draw [shift={(-9.580186094175373,5.47252578929432)},line width=0.8pt]  plot[domain=5.049525559069649:5.297531238270618,variable=\t]({1*8.192941903720733*cos(\t r)+0*8.192941903720733*sin(\t r)},{0*8.192941903720733*cos(\t r)+1*8.192941903720733*sin(\t r)});
	\draw [shift={(-13.634190404676545,-1.2219439672984476)},line width=0.8pt]  plot[domain=5.778790061782033:6.131023724265962,variable=\t]({1*6.8431845304883625*cos(\t r)+0*6.8431845304883625*sin(\t r)},{0*6.8431845304883625*cos(\t r)+1*6.8431845304883625*sin(\t r)});
	\draw [shift={(-11.262386186857666,-11.212152680983046)},line width=0.8pt]  plot[domain=0.73203359295954:1.0744669145658485,variable=\t]({1*7.600103262564913*cos(\t r)+0*7.600103262564913*sin(\t r)},{0*7.600103262564913*cos(\t r)+1*7.600103262564913*sin(\t r)});
	\draw [shift={(-9.07120502581474,-10.558796647760097)},line width=0.8pt]  plot[domain=0.8750235682085904:1.1782686708620953,variable=\t]({1*7.178542301205925*cos(\t r)+0*7.178542301205925*sin(\t r)},{0*7.178542301205925*cos(\t r)+1*7.178542301205925*sin(\t r)});
	\draw [shift={(3.15,-7.81)},line width=0.8pt]  plot[domain=2.56053098439544:2.841226682431574,variable=\t]({1*9.12678098439172*cos(\t r)+0*9.12678098439172*sin(\t r)},{0*9.12678098439172*cos(\t r)+1*9.12678098439172*sin(\t r)});
	\draw [shift={(3.15,-7.81)},line width=1.2pt,dash pattern=on 4pt off 3pt,color=wwwwww]  plot[domain=2.5578177142562817:2.695143002337464,variable=\t]({1*9.580858789221626*cos(\t r)+0*9.580858789221626*sin(\t r)},{0*9.580858789221626*cos(\t r)+1*9.580858789221626*sin(\t r)});
	\draw [shift={(3.15,-7.81)},line width=1.2pt,dash pattern=on 4pt off 3pt,color=wwwwww]  plot[domain=2.5553979720905495:2.696651395427956,variable=\t]({1*8.705929397569424*cos(\t r)+0*8.705929397569424*sin(\t r)},{0*8.705929397569424*cos(\t r)+1*8.705929397569424*sin(\t r)});
	\draw [shift={(-5.100325862662596,-3.8679829582447454)},line width=1.2pt,dash pattern=on 4pt off 3pt,color=wwwwww]  plot[domain=2.6801286047569395:5.821721258346733,variable=\t]({1*0.4373528273756912*cos(\t r)+0*0.4373528273756912*sin(\t r)},{0*0.4373528273756912*cos(\t r)+1*0.4373528273756912*sin(\t r)});
	\draw [shift={(-5.554912282985272,-4.9543775747415335)},line width=1.2pt,dash pattern=on 4pt off 3pt,color=wwwwww]  plot[domain=-0.32814195250013256:2.8134507010896606,variable=\t]({1*0.4060137779027065*cos(\t r)+0*0.4060137779027065*sin(\t r)},{0*0.4060137779027065*cos(\t r)+1*0.4060137779027065*sin(\t r)});
	\draw [shift={(3.15,-7.81)},line width=1.2pt,dash pattern=on 4pt off 3pt,color=wwwwww]  plot[domain=2.8251244539159845:2.866957747242475,variable=\t]({1*8.755348660035269*cos(\t r)+0*8.755348660035269*sin(\t r)},{0*8.755348660035269*cos(\t r)+1*8.755348660035269*sin(\t r)});
	\draw [shift={(3.15,-7.81)},line width=1.2pt,dash pattern=on 4pt off 3pt,color=wwwwww]  plot[domain=2.824133665514144:2.8582363204509242,variable=\t]({1*9.567325582268333*cos(\t r)+0*9.567325582268333*sin(\t r)},{0*9.567325582268333*cos(\t r)+1*9.567325582268333*sin(\t r)});
	%
	\draw [->,line width=0.6pt] (-4.360687359697925,-3.1473715785552754) -- (-4.896657276796503,-2.784561173134744); 
	\draw [->,line width=0.6pt] (-4.534259155927542,-3.3877334721463765) -- (-5.0537376909615475,-3.076046351126011); 
	\begin{scriptsize} 
	\draw[color=wwwwww] (6.17,-4.61) node {$\pi(\mathcal N(E))$};
	\draw[color=black] (-4.3,-2.6) node {$E$};
	\draw[color=wwwwww] (-4.1,-4.1) node {$\mathcal N(E)$};
	\end{scriptsize}
	\draw[color=black] (7.5,-2.7) node {$X$};
	\draw [fill=black] (2.037,-4) circle (1.5pt);
	\draw[color=black] (2,-3.65) node {$0$};
	\draw[color=black] (-0,-3.65) node {$\pi$};
	\draw[color=black] (-7.46,-2.7) node {$X_\pi$};
	\end{tikzpicture}
	
	\caption{The inner rate $q_E$, which measures the size of the subset $\pi\big(\cal N(E)\big)$ of $(X,0)$ delimited by the dashed gray line on the right, can also bee defined as the inner contact of the pushforwards to $X$ of two generic curvettes of $E$ (see the discussion of Section~\ref{sec:inner rates function}).}\label{figure:introduction}
\end{figure}

As the sets of the form $\pi\big(\cal N(E)\big)$ cover the germ $(X,0)$ and can be taken to be arbitrarily small by refining the resolution $\pi$, the knowledge of all the inner rates of the exceptional divisors of all good resolutions of $(X,0)$ gives a very fine understanding of the metric structure of the germ.
For example, one can use the inner rates to compute the contact order for the inner metric of any pair of complex curve germs in $(X,0)$, using a minimax procedure (see Remark~\ref{remark:maxmin}).

The inner rates form an infinite family of analytic invariants of the germ $(X,0)$. 
Some specific inner rates have been studied in the context of Lipschitz geometry, but the bilipschitz class of $(X,0)$ can only determine finitely many of them.
This is a consequence of the tameness of the bilipschitz classification of germs, which was proved by Mostovski in the complex analytic setting \cite{Mostowski1985} and by Parusi\'{n}ski in the real semi-analytic setting \cite{Parusinski1988}.
The finite set of inner rates which are bilipschitz invariants for the inner metric has been described explicitly by Birbrair, Neumann, and Pichon\cite{BirbrairNeumannPichon2014}.

While the inner rates have raised significant interest in the last decade, it was still an open question to determine how the global geometry of the singularity $(X,0)$ influences their behavior.
Our first result states that the global data of the topology of $(X,0)$, together with the configuration of a generic hyperplane section and of the polar curve of a generic plane projection $  (X,0) \to(\C^2,0)$, not only influences the behavior of the inner rates, but in fact completely determines all of them.
 
\begin{theorem}\label{thm:A}
Let $(X,0)$ be an isolated complex surface singularity, let $E$ be an irreducible component of the exceptional  divisor of some resolution of $(X,0)$, and let $\pi\colon X_\pi\to X$ be the minimal good resolution of $(X,0)$ that factors through the blowup of the maximal ideal and through the Nash transform of $(X,0)$. 
Then the inner rate $q_E$ of $E$ is determined by, and can be computed in terms of, the following data:
\begin{enumerate}
	\item 
	the topological data consisting of the dual graph $\Gamma_\pi$ of $\pi^{-1}(0)$ decorated with the Euler classes and the genera of its vertices;
	
	\item 
	on each vertex $v$ of $\Gamma_\pi$, an arrow for each irreducible component of a generic hyperplane section of $(X,0)$ whose strict transform on $X_\pi$ passes through the irreducible component  $E_v$ of $\pi^{-1}(0)$ corresponding to $v$;
	
	\item 
	on each vertex $v$ of $\Gamma_\pi$, an arrow for each irreducible component of  the  polar curve of  a generic projection $ (X,0) \to(\C^2,0)$ whose strict transform on $X_\pi$ passes through  $E_v$.
\end{enumerate}
\end{theorem}
 
In earlier papers on the subject, the inner rates were always computed by considering a generic projection $\ell \colon (X,0) \to (\C^2,0)$ and lifting the inner rates of the components of the exceptional divisor of a suitable resolution of the discriminant curve of $\ell$. 
Outside of the simplest examples, this approach turned out to be very impractical and computationally expensive, since it is generally very hard to compute discriminant curves and to decide whether a projection is generic.
On the other hand, given the data of the theorem, which is generally much simpler to obtain, our result also provides a very easy way to compute the inner rate $q_E$ by means of an elementary linear algebra computation.
To showcase this fact, in Example~\ref{example:very_big} we illustrate, for a fairly complicated singularity, how simple it is to obtain the inner rates, which were computed in \cite[Example 15.2]{BirbrairNeumannPichon2014} using Maple.

We obtain Theorem~\ref{thm:A} as a consequence (see Corollary~\ref{corollary_global_to_inner}) of a stronger result about the Laplacian on the inner rate function on a space of valuations, a non-archimedean avatar of the link $X^{(\epsilon)}$ of $(X,0)$.
Indeed, it is very natural to study the inner rates from the point of view of valuation theory, since $q_E$ only depends on the \emph{divisorial valuation} associated with an exceptional divisor $E$.

The \emph{non-archimedean link} $\NL(X,0)$ of $(X,0)$ is defined as the set of (suitably normalized) semi-valuations on the completed local ring $\cal O_{X,0}$ of $X$ at $0$ that are trivial on $\C$ (see Definition~\ref{def:NL}). 
It is a Hausdorff topological space that contains the set of divisorial valuations of $(X,0)$ as a dense subset.
For example, if $X$ is smooth at $0$ then the associated non-archimedean link is a well known object, the valuative tree of Favre and Jonsson \cite{FavreJonsson2004}; the singular case was first studied by Favre in \cite{Favre2010}.

If $\pi\colon X_\pi\to X$ is a good resolution of $(X,0)$, then the dual graph $\Gamma_\pi$ of the exceptional divisor $\pi^{-1}(0)$ embeds naturally in $\NL(X,0)$, and the latter deforms continuously onto the former.
Moreover, these retractions identify $\NL(X,0)$ with the projective limit of all the dual graphs of the good resolutions of $(X,0)$, allowing one to see $\NL(X,0)$ as a universal dual graph.
This makes it a very convenient object in the study of the inner rates.

We endow each dual resolution graph $\Gamma_\pi$ with a metric defined as follows.
Let $e=[v,v']$ be an edge of $\Gamma_\pi$ and let $E_v$ be the component of the exceptional divisor $\pi^{-1}(0)$ corresponding to $v$.
Then the \emph{multiplicity} $m_v$ of $v$ is defined as the order of vanishing on $E$ of the pullback on $X_\pi$ of a generic hyperplane section of $(X,0)$ (that is, equivalently, $m_v$ is the multiplicity of $E_v$ in the exceptional divisor $\pi^{-1}(0)$).
Similarly, denote by $m_{v'}$ the multiplicity of $v'$.
Then we declare the length of the edge $e$ to be $1/(m_vm_{v'})$.
This length has a natural geometrical interpretation as the opposite of the screw number of the representative of the monodromy automorphism on the piece defined by the edge $e$ of the Milnor fiber of a generic linear form on $(X,0)$ (see \cite[Theorem 7.3.(iv)]{MatsumotoMontesinos2011} and the discussion of Remark~\ref{remark_length_dehn_screw}). 
These lengths give rise to a metric on each $\Gamma_\pi$ and, by refining the resolutions $\pi$, to a natural metric on $\NL(X,0)$.

The starting point of our exploration lies in the fact that the map sending a divisorial valuation $v$ to the associated inner rate $q_{E_v}$ extends canonically to a continuous function $\cal I_X\colon\NL(X,0)\to \R_+\cup\{+\infty\}$ that has the remarkable property that its restriction $\calI_X|_{\Gamma_\pi}$ to any dual resolution graph $\Gamma_\pi$ is piecewise linear with integral slopes with respect to the metric defined above.
This allows to study the inner rate function $\calI_X$ using classical tools of potential theory on metric graphs, as is done for example in an arithmetic setting by Baker and Nicaise \cite{BakerNicaise2016}.
Namely, our main result is a formula computing the \emph{Laplacian} $\Delta_{\Gamma_\pi}\big(\calI_X\big)$ of the restriction of the inner rate function to a dual resolution graph $\Gamma_\pi$, that is the divisor on $\Gamma_\pi$ whose coefficient $\Delta_{\Gamma_\pi}\big(\calI_X\big)(v)$ at a vertex $v$ of $\Gamma_\pi$ is the sum of the slopes of $\calI_X$ on the edges of $\Gamma_\pi$ emanating from $v$ (see Section~\ref{subsection_preliminaries_laplacians} for a detailed explanation of all the relevant notions).

In order to measure quantitatively the geometric data appearing in the statement of Theorem~\ref{thm:A}, 
for any vertex $v$ of $\Gamma_{\pi}$ denote by $l_v$ (respectively by $p_v$) the number of arrows on $v$ associated with a generic hyperplane section (respectively with a generic polar curve) of $(X,0)$, by $\check E_{v}$ the reduced curve obtained by removing from $E_v$ the double points of $\pi^{-1}(0)$, and by $\chi(\check E_{v})$ the Euler characteristic of $\check E_{v}$.
We can now state a slightly weaker version of our main result, Theorem \ref{theorem_main}.

\begin{theorem}[Laplacian of the inner rate function]
\label{thm:B} 
	Let $(X,0)$ be an isolated complex surface singularity and let $\pi\colon X_\pi\to X$ be a good resolution of $(X,0)$ that factors through the blowup of the maximal ideal and through the Nash transform of $(X,0)$.
	For every vertex $v$ of $\Gamma_{\pi}$, we have:
	\[
	\Delta_{\Gamma_\pi}\big(\calI_X\big)(v) =  m_v\big(2l_v-p_v-\chi(\check E_{v})\big).
	\]
\end{theorem}

We obtain this theorem by first proving the corresponding result in the smooth case and then showing how to carefully lift each term via a generic projection of $(X,0)$ onto a smooth germ $(\C^2,0)$ by studying the topology of the Milnor fiber of a generic linear form on $(X,0)$. 
%
%
The proof uses in an essential way the fact that the link of a complex surface is a $3$-dimensional real manifold, as it relies on a computation of screw numbers of Dehn twists which appear naturally as part of the monodromy of the fibration over the circle of a graph-manifold.
For this reason, such a precise result is very specific to the dimension two.

Besides the complete description of inner rates obtained in Theorem~\ref{thm:A}, Theorem~\ref{thm:B} has several interesting consequences.
First, as an immediate by-product of our result we deduce that the inner rate function $\cal I_X$ is actually linear, and not just piecewise linear, along every string of the metric graph $\Gamma_{\pi}$ (Corollary~\ref{cor:linearity on strings_general}). 
Furthermore, the fact that the Laplacian is a divisor of degree zero of $\Gamma_\pi$ (that is, the sum over all the vertices $v$ of $\Gamma_\pi$ of the integers $\Delta_{\Gamma_\pi}\big(\calI_X\big)(v)$ vanishes) allows us to retrieve the L\^e--Greuel--Teissier formula \cite{LeTeissier1981}, which is a singular version of the classical L\^e--Greuel formula, in the case of a generic linear form $h \colon (X,0) \to (\C,0)$ on a complex surface singularity (Proposition~\ref{proposition:Le-Greuel}). 

The Laplacian formula of Theorem~\ref{thm:B} also imposes strong restrictions on the possible configurations of the values of $l_v$ and $p_v$ on the vertices of the graph $\Gamma_{\pi}$.
To perform a detailed study of this phenomenon, we prove an alternative, more numerical version of  Theorem~\ref{theorem_main}, Proposition~\ref{proposition:application_linear_algebra}, which not only permits us to easily derive an extended version of Theorem~\ref{thm:A} (Corollary \ref{corollary_global_to_inner}), but also has multiple other applications.
For example, considering now the minimal resolution of $(X,0)$ which factors through the blowup of the maximal ideal (and not necessarily also through the Nash transform) of $(X,0)$, then the data of the inner rates and the integers $l_v$ on $\Gamma_\pi$ determine the localization of the  strict transform of the polar curve $\Pi$ of a generic projection of $(X,0)$, and vice versa.
More interestingly, prescribing the values of the inner rates is not required, as even without doing so Proposition~\ref{proposition:application_linear_algebra} still gives us strong restrictions on the relative positions of the components of the polar curve and of the generic hyperplane section, and on the values of the inner rates as well, as illustrated in Example~\ref{example:BS}.
This application has recently been further explored by the authors \cite{BelottodaSilvaFantiniPichon2020}.
	
We also remark that the hypotheses on the factorization of $\pi$ in Theorems~\ref{thm:A} and \ref{thm:B}, although useful to simplify the statement, are not strictly necessary.
A milder hypothesis appears in Theorem~\ref{theorem_main}, while no requirement on $\pi$ is present in Proposition~\ref{proposition:application_linear_algebra}.

These restrictions on the relative behavior of the polar curve and of the generic hyperplane section are to be interpreted as evidence describing the expected duality between the two main algorithms of resolution of a complex surface, via normalized blowups of points (after Zariski \cite{Zariski1939}) or via normalized Nash transforms (after Spivakovsky \cite{Spivakovsky1990}), whose existence was discussed by D. T. L\^e (see \cite[Section 4.3]{Le2000}). 
Further restrictions can be derived from the fact that not every degree zero divisor on a metric graph can be realized as the Laplacian of a piecewise linear function with integral slopes; this is part of a rapidly developing research area and we refer the interested reader to the recent book \cite{corry2018divisors} for an overview of divisor theory on metric graphs that focuses on similar questions.

Finally, as explained at the end of the paper, Proposition  \ref{proposition:application_linear_algebra} can also be considered from the point of Michel \cite[Theorem 4.9]{Michel2008}, which gives information on the localization of the polar curve of the germ of a finite morphism $(f,g) \colon (X,0) \to (\C^2,0)$ in terms of the topology of the pair $\big(X,(fg)^{-1}(0)\big)$. 
In the case where $f$ and $g$ are generic linear forms on $(X,0)$, that is when the morphism $ \ell = (f,g)$ is a generic projection of $(X,0)$, \cite[Theorem 4.9]{Michel2008} does not give more information than the polar multiplicity described by L\^e--Greuel--Teissier formula.
However, our Proposition~\ref{proposition:application_linear_algebra} refines this result, imposing strong restrictions on the possible localization of the polar curve. 

}

\smallskip
\begin{center}
	$\diamond$
\end{center}
\medskip

Let us conclude this introduction by discussing several future directions whose exploration is made possible by our valuative point of view on inner rates, and which will be the subjects of further study by the authors.

First of all, as observed in Remark~\ref{remark:ultrametric}, the inner rates can be used to define an ultrametric (that is, non-archimedean) inner distance on $\NL(X,0)$ via a standard minimax procedure which in this case has a very natural geometric interpretation.
Ultrametric distances on valuations spaces of surface singularities have been recently studied by Barroso, Perez, Popescu-Pampu, and Ruggiero \cite{BarrosoPerezPopescu-PampuRuggiero2018}; however only a special class of surface singularities, those having contractible non-archimedean link (called \emph{arborescent} in \emph{loc.cit.}) were endowed there with an ultrametric distance carrying a geometric interpretation.
Our non-archimedean inner distance on $\NL(X,0)$ will be used to define canonical metric decompositions of non-archimedean links, parallel to those of \cite{BirbrairNeumannPichon2014} but much more intrinsic.

We expect that the inner rate of a divisorial valuation $v$ of $(X,0)$ can be studied via birational techniques, namely using logarithmic Fitting ideals (see Remark~\ref{remark:discrepancies}).
This approach would be well suited to be generalized to the higher dimensional case, for example by using Hsiang--Pati local coordinates, which were introduced for surfaces in \cite{HsiangPati1985} and recently generalized to the three dimensional case by Belotto da Silva, Bierstone, Grandjean, and Milman 
\cite{BelottodaSilvaBierstoneGrandjeanMilman2017}.
In fact, the metric properties of singular varieties of higher dimension are poorly understood and this point of view could help shed some light on them, although we do not expect a result as sharp as our Theorem~\ref{thm:B}.

It is also worth noticing that the Laplacian formula could be fairly easily extended to the case of a complex surface germ with non isolated singularity by adding a correcting term involving the multiplicities of generic hyperplane sections of $X$ at points of $\mathrm{Sing} X \setminus \{0\}$. 
This correcting term  would appear in a generalized statement of Proposition \ref{proposition_lifting_K}, that is in the computation of the Euler Characteristic of the  part $F_v$ of the Milnor fiber of a generic linear form on $(X,0)$.  
Moreover, the fact that our formula enables us to recover a L\^e--Greuel type formula for a particular function also suggests the possibility of extending the Laplacian formula to more general settings such as the metric behavior of a pair of holomorphic germs $(f,g) \colon (X,0) \to (\C^2,0)$, which as of now is still poorly understood.

Finally, let us remark that the inner rates were used in \cite{BirbrairNeumannPichon2014} to study the Lipschitz classification of the inner metric surface germs $(X,0)$.
While in this paper we chose to focus on a metric germ $(X,0)$, and not on its Lipschitz class, our methods seem very well adapted to study questions of bilipschitz geometry.
For example, the valuation-theoretic point of view permits to give an elegant valuative version of the complete invariant on the inner Lipschitz geometry of $(X,0)$ constructed in \cite{BirbrairNeumannPichon2014}.

\smallskip
\begin{center}
$\diamond$
\end{center}
\medskip

Let us give a short outline of the structure of the paper.
All the material about non-archimedean links, dual graphs, and potential theory on metric graphs that we need in the paper is recalled in section \ref{sec:preliminary}.
Section \ref{sec:inner rates function} is devoted to the construction of the inner rate function  $\cal I_X$ and to the proof of its basic properties, such as its piecewise linearity.
In Section~\ref{section_main_theorem} we state and prove our main result, Theorem \ref{theorem_main}, which is a stronger version of Theorem \ref{thm:B}.
Section~\ref{section:applications} is devoted to proving a slightly more precise version of Theorem~\ref{thm:A}, Corollary~\ref{corollary_global_to_inner}, and discussing other applications of Theorem~\ref{theorem_main}.

We aimed to make the paper entirely self-contained. 
The main definitions and results are illustrated with the help of a recurring example, that of the $E_8$ surface singularity, which is treated in detail throughout the paper (see Examples \ref{example:E8_graph}, \ref{example:E8_metric_graph}, \ref{example:E8_Laplacien}, \ref{example:E8_inner_rate}, \ref{example:E8_LPK}, and \ref{example:E8_Theorem}).

\subsection*{Acknowledgments} 
We warmly thank Nicolas Dutertre, Charles Favre, Fran\c cois Loeser, Walter Neumann, Patrick Popescu-Pampu, and Matteo Ruggiero for interesting conversations, and the anonymous referee for their thorough comments and for pointing out a mistake in Example~\ref{example:BS}.
We are deeply indebted to Delphine et Marinette, whose method from \cite{Ayme1946} we applied with much profit while searching for the correct statement of Theorem~\ref{theorem_main}, and to the master Maurits C. Escher, whose celebrated lithography \emph{Ascending and Descending} \cite{Escher1960} inspired the argument given in Remark~\ref{remark_laplacian_determines_function}.
This work has been partially supported by the project \emph{Lipschitz geometry of singularities (LISA)} of the \emph{Agence Nationale de la Recherche} (project ANR-17-CE40-0023) and by the 
\emph{PEPS--JCJC M\'etriques singuli\`eres, valuations et g\'eom\'etrie Lipschitz des vari\'et\'es} of the \emph{Institut National des Sciences Math\'ematiques et de leurs Interactions}.

 
\section{Preliminaries}
\label{sec:preliminary}
 
\subsection{Non-archimedean links}
\label{subsection_preliminaries_NL}
Throughout the paper, $(X,0)$ will always be an isolated complex surface singularity.
We will begin by introducing the valuation space we will work with, the {non-archimedean link} of $(X,0)$.

Denote by $\calO=\widehat{\calO_{X,0}}$ the completion of the local ring of $X$ at $0$ with respect to its maximal ideal.
A (rank 1) \emph{semivaluation} on $\calO$ is a map $v\colon \calO\to \R_+\cup\{+\infty\}$ such that, for every $f$ and $g$ in $\cal O$ and every $\lambda$ in $\C$, we have
\begin{enumerate}
	\item $v(fg)=v(f)+v(g)$;
	\item $v(f+g)\geq \min \{v(f),v(g)\}$;
	\item $v(\lambda)=\begin{cases}
	+\infty & \mbox{if } \lambda=0\\
	0 & \text{otherwise.}
	\end{cases}$
\end{enumerate}
Note that we do not require $0$ to be the only element sent to $+\infty$, which is why our maps are only semivaluations  rather than valuations.
If $v$ is a semivaluation on $\calO$, the valuation of $v$ on the maximal ideal $\mathfrak M$ of $\calO$ is defined as $v(\mathfrak M)=\inf\big\{v(f)\,\big|\,f\in\mathfrak M\big\}$.
Observe that by definition $v(\mathfrak M)$ can be computed as the valuation of the equation of a generic hyperplane section of $(X,0)$.

\begin{example}\label{example_divisorial_valuation}
The main example of (semi)valuation that we will consider in this paper is the following.
Let $\pi\colon X_\pi\to X$ be a good resolution of $(X,0)$ and let $E$ be an irreducible component of the exceptional divisor $  \pi^{-1}(0)$. Then the map
\begin{align*}
v_E \colon & \calO \longrightarrow \R_+\cup\{+\infty\}\\ 
&f\longmapsto \frac{\mathrm{ord}_E(f)}{\mathrm{ord}_E(\mathfrak M)}
\end{align*}
 is a valuation on $\calO$.
We call it the \emph{divisorial valuation} associated with $E$.
Note that this valuation does not depend on the choice of $\pi$, in the following sense: if $\pi'=\pi\circ\rho$ is another good resolution of $(X,0)$ that dominates $\pi$, then $v_E=v_{E'}$, where $E'$ denotes the strict transform of $E$ in $X_{\pi'}$.
If $v$ is the divisorial valuation associated with a component of the exceptional divisor $\pi^{-1}(0)$ of a good resolution $\pi\colon X_\pi\to X$ of $(X,0)$, we will generally denote by $E_v$ this prime divisor, and by $m_v$ the positive integer $\ord_{E_v}(\mathfrak M)$, which we call the \emph{multiplicity} of $E_v$.
This terminology is justified by the fact that $m_v$ is also the multiplicity of $E_v$ in $\pi^{-1}(0)$, when the latter is considered with its natural scheme structure.
In practice, $m_v$ is usually computed as the order of vanishing along $E_v$ of the total transform by $\pi$ of a generic hyperplane section of $(X,0)$.
For a plethora of examples of semivaluations we refer the reader to \cite[Chapter 1]{FavreJonsson2004}.
\end{example}

We are interested in a projectivisation of the space semivaluations on $\cal O$.
This is be achieved by only considering those semivaluations that are normalized by requiring that the valuation of the maximal ideal is 1.

\begin{definition}\label{def:NL}
The \emph{non-archimedean link} $\NL(X,0)$ of $(X,0)$ is the topological space whose underlying set is
\[
\NL(X,0)=\big\{v\colon \cal O\to\R_+\cup\{+\infty\}\text{ semivaluation}\;\big|\; v(\mathfrak{M})=1 \hbox{ and } v|_{\C}=0\big\}
\]
and whose topology is induced from the product topology of $\big(\R_+\cup\{+\infty\}\big)^\mathfrak{M}$ (that is, it is the topology of the point-wise convergence).
\end{definition}

\begin{remark}
	The non-archimedean link can be endowed with an additional analytic structure, by considering the space of semi-valuations on $\mathcal O$ as a non-archimedean analytic space, in the sense of Berkovich \cite{Berkovich1990}, over the field $\C$ endowed with the trivial absolute value.
	This point of view was developed in \cite{Fantini2018}, where it was used to obtain a non-archimedean characterization of the essential valuations of a surface singularities in arbitrary characteristic, and later in \cite{FantiniFavreRuggiero2018} to give a characterization of sandwiched surface singularities.
	Moreover, this can be done independently from classical resolution of singularities, as it is possible to study the analytic structure of non-archimedean links to deduce the existence of resolutions of singularities of complex surfaces (see \cite[Theorem 8.6]{FantiniTurchetti2018}).
	However, as the analytic structure of $\NL(X,0)$ plays no role in this paper, we will not discuss it further.
\end{remark}

If $(X,0)=(\C^2,0)$ is a smooth surface germ, its non-archimedean link $\NL(\C^2,0)$ is the so-called \emph{valuative tree}, the main object of study of the monograph \cite{FavreJonsson2004}.
It is an infinite real tree which has several applications in commutative algebra and to the study of the dynamics of complex polynomials in two variables.

The non-archimedean link $\NL(X,0)$ is a useful object if one wants to study the resolutions of $(X,0)$, as we will now explain.
We call \emph{good resolution} of $(X,0)$ a proper morphism $\pi\colon X_\pi \to X$ such that $X_\pi$ is regular, $\pi$ is an isomorphism outside of its exceptional locus $\pi^{-1}(0)$, and the latter is a strict normal crossing divisor on $X_\pi$.
The fact that the exceptional divisor has normal crossing allows us to associate with it its \emph{dual graph} $\Gamma_\pi$, which is the graph whose vertices are in bijection with the irreducible components of $\pi^{-1}(0)$, and where two vertices of $\Gamma_\pi$ are joined by an edge for each point of the intersection of the corresponding components.
We will generally write $v$ for a vertex of $\Gamma_\pi$ and $E_v$ for the associated component of $\pi^{-1}(0)$ (note that this is consistent with the notation introduced in Example~\ref{example_divisorial_valuation}).
We also write $g(v)$ for the genus of the component $E_v$; it will sometimes be useful to think of $g(v)$ as a function with finite support on the topological space underlying $\Gamma_\pi$.

While in general not all good resolutions of $(X,0)$ factor through the blowup of its maximal ideal (see Example~\ref{example:BS}), throughout the paper this will be the case most of the time.
We call the subset of $\NL(X,0)$ consisting of the divisorial valuations associated with the exceptional components of the blowup of the maximal ideal of $\NL(X,0)$ the set of \emph{$\cal L$-nodes} of $(X,0)$.
In other words, this is the set of the Rees valuations of the maximal ideal of $(X,0)$.

\begin{example}\label{example:E8_graph}
	Consider the standard singularity $X=E_8$, which is the hypersurface in $\C^3$ defined by the equation $x^2+y^3+z^5=0$.   
	A good resolution $\pi\colon X_\pi\to X$ of $(X,0)$ can be obtained by the method described in \cite[Chapter II]{Laufer1972}; this method is particularly useful for us because it is based on projections to $\C^2$ and thus fits well with our point of view.
	It considers the projection $\ell = (y,z) \colon (X,0) \to (\C^2,0)$ and, given a suitable embedded resolution $\sigma \colon Y \to \C^2$ of the associated discriminant curve $\Delta \colon y^3+z^5=0$, gives a simple algorithm to compute a resolution of $(X,0)$ as a cover of $Y$.
	In this example, the dual graph $\Gamma_\sigma$ of the minimal embedded resolution $\sigma$ of $\Delta$ in $\C^2$ is depicted on the left of Figure~\ref{figure:curve}.
	Its vertices are labeled as $\nu_0, \nu_1, \nu_2$, and $\nu_3$ in their order of appearance as exceptional divisors of blowups in the resolution process.
	The negative number attached to each vertex denotes the self-intersection of the corresponding exceptional curve, while the arrow denotes the strict transform of $\Delta$. 
	In this case Laufer's algorithm requires us to refine the resolution $\sigma$ to another resolution $\sigma'$ obtained by blowing up once each double point of the exceptional divisor $\sigma^{-1}(0)$.
	This yields the dual graph $\Gamma_{\sigma'}$ depicted on the right of Figure~\ref{figure:curve}.
	Again, each vertex is decorated by the self-intersection of the corresponding exceptional curve and the arrow denotes the strict transform of $\Delta$. 
	
	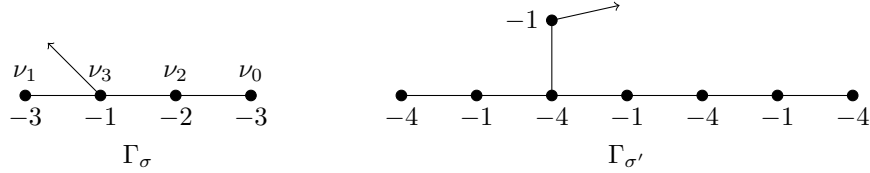
\begin{figure}[h] 
	\centering
	\begin{tikzpicture}
	\node(a)at(-.5,-.8){$\Gamma_\sigma$};
	\draw[thin ](-2,0)--(1,0);
	\draw[fill ] (-2,0)circle(2pt);
	\draw[fill ] (-1,0)circle(2pt);
	\draw[fill ] (0,0)circle(2pt);
	\draw[fill ] (1,0)circle(2pt);
	
	\draw[thin,>-stealth,->](-1,0)--+(-0.7,0.7);

	\node(a)at(-2,-0.3){$-3$};
	\node(a)at(-1,-0.3){$-1$};
	\node(a)at(0,-0.3){$-2$};
	\node(a)at(1,-0.3){$-3$};
	
	\node(a)at(-2,0.3){$\nu_1$};
	\node(a)at(-1,0.3){$\nu_3$};
	\node(a)at(0,0.3){$\nu_2$};
	\node(a)at(1,0.3){$\nu_0$};

	\begin{scope}[xshift=5cm]

\draw[thin ](-2,0)--(4,0);
\draw[thin ](0,0)--(0,1);
\draw[fill ] (-2,0)circle(2pt);
\draw[fill ] (-1,0)circle(2pt);
\draw[fill ] (0,0)circle(2pt);
\draw[fill ] (1,0)circle(2pt);
\draw[fill ] (2,0)circle(2pt);

\draw[fill ] (3,0)circle(2pt);

\draw[fill ] (4,0)circle(2pt);

\draw[fill ] (0,1)circle(2pt);

	\draw[thin,>-stealth,->](0,1)--+(+0.9,0.2);

\node(a)at(-2,-0.3){$-4$};
\node(a)at(-1,-0.3){$-1$};
\node(a)at(0,-0.3){$-4$};
\node(a)at(1,-0.3){$-1$};
\node(a)at(2,-0.3){$-4$};
\node(a)at(3,-0.3){$-1$};
\node(a)at(4,-0.3){$-4$};
\node(a)at(-0.4,1){$-1$};

\node(a)at(1,-0.8){$\Gamma_{\sigma'}$};
	
	\end{scope}
	\end{tikzpicture}
	\caption{Dual resolution graphs for the plane curve $\{y^3+z^5=0\}$.}\label{figure:curve}
\end{figure}

	\noindent Then it follows from Laufer's algorithm that the exceptional divisor $\pi^{-1}(0)$ of $\pi$ is a tree of eight rational curves whose dual graph $\Gamma_\pi$ is depicted in Figure~\ref{fig:E8}.   
	We label the vertices of $\Gamma_\pi$ as $v_0, \ldots, v_7$, since this will make it simple to refer to this example in the rest of the paper. 
	Each vertex of $\Gamma_\pi$ is mapped by $\ell$ to the corresponding vertex of $\Gamma_\sigma$.
	As before, the negative numbers attached to the $v_i$ are the self-intersections of the corresponding exceptional components; as no self-intersection is equal to $-1$ this resolution is minimal. 
	Observe that $\pi$ factors through the blowup of the maximal ideal of $(X,0)$ and the only $\cal L$-node of $(X,0)$ is $v_0$.
	This example will be detailed throughout the whole paper, see Examples~\ref{example:E8_metric_graph}, \ref{example:E8_Laplacien}, \ref{example:E8_inner_rate}, \ref{example:E8_LPK}, and \ref{example:E8_Theorem}.
	\begin{figure}[h] 
		\centering
		\begin{tikzpicture}
		\draw[thin ](-2,0)--(4,0);
		\draw[thin ](0,0)--(0,1);
		\draw[fill ] (-2,0)circle(2pt);
		\draw[fill ] (-1,0)circle(2pt);
		\draw[fill ] (0,0)circle(2pt);
		\draw[fill ] (1,0)circle(2pt);
		\draw[fill ] (2,0)circle(2pt);
		
		\draw[fill ] (3,0)circle(2pt);
		
		\draw[fill ] (4,0)circle(2pt);
		
		\draw[fill ] (0,1)circle(2pt);

		\node(a)at(4,0.3){ $v_0$};
		\node(a)at(3,0.3){ $v_1$};
		\node(a)at(2,0.3){ $v_2$};
		\node(a)at(1,0.3){ $v_3$};
		\node(a)at(0.3,0.3){$v_4$};
		\node(a)at(-1,0.3){$v_5$};
		\node(a)at(-2,0.3){ $v_6$};
		\node(a)at(0.3,1){ $v_7$};
		
		\node(a)at(-2,-0.4){$-2$};
		\node(a)at(-1,-0.4){$-2$};
		\node(a)at(0,-0.4){$-2$};
		\node(a)at(1,-0.4){$-2$};
		\node(a)at(2,-0.4){$-2$};
		\node(a)at(3,-0.4){$-2$};
		\node(a)at(4,-0.4){$-2$};
		\node(a)at(-0.4,1){$-2$};
		\end{tikzpicture} 
		\caption{The dual resolution graph $\Gamma_{\pi }$ for the singularity $E_8$.}\label{fig:E8}
	\end{figure}
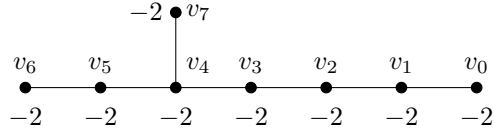
\end{example}


We will now explain how to describe the structure of a non-archimedean link $\NL(X,0)$ in terms of the dual graphs of the good resolutions of $(X,0)$.
The results below can be deduced fairly easily from the analogous result for the valuative tree $\NL(\C^2,0)$, originally proven in \cite[Chapter 6]{FavreJonsson2004} and elegantly summarized in \cite[Section 7]{Jonsson2015}, as was discussed for example in \cite[Section 2.5]{GignacRuggiero2017}.

Given any  good resolution $\pi \colon X_\pi \to X$ with dual graph $\Gamma_\pi$, there exist a natural embedding
\[
i_\pi \colon \Gamma_{\pi} {\lhook\joinrel\longrightarrow} \NL(X,0)
\]
and a canonical continuous retraction 
\[
r_\pi\colon \NL(X,0)\longrightarrow\Gamma_\pi
\]
such that $r_\pi\circ i_\pi=\mathrm{Id}_{\Gamma_\pi}$.
The embedding $i_\pi$ maps each vertex $v$ of $\Gamma_\pi$ to the divisorial valuation associated with the component $E_v$, and each edge $e=[v,v']$ that corresponds to a point $p$ of the intersection $E_v\cap E_v'$ to the set of monomial valuations on $X_\pi$ at $p$. 
More precisely, given $0\leq\omega\leq1$ and $f\in\cal O$, we choose an isomorphism $\widehat{\mathcal O_{X_\pi,P}}\cong \C[[x,y]]$ such that $E_v$ and $E_{v'}$ are defined locally at $p$ by $x=0$ and $y=0$, and define a valuation $v_{\pi,p,\omega}$ on $\cal O$ by setting 
\[
v_{\pi,p,\omega}(f)=\inf\big\{(i,j)\cdot((1-\omega)/m_{v'},\omega/m_{v})\,\big|\,a_{i,j}\neq0\big\},
\]
where $f$ is an element $\cal O$ and we denote by $\pi^*f=\sum a_{i,j}x^iy^j$ the image in $\widehat{\mathcal O_{X_\pi,P}}$ of the pullback of $f$ through $\pi$.
Observe that $v_{\pi,p,0}=v$ and $v_{\pi,p,1}=v'$, while the points $v_{\pi,p,\omega}$ for $0<\omega<1$ form a segment joining the two divisorial valuations $v$ and $v'$ in $\NL(X,0)$.

The retraction $r_\pi$ is defined as follows. 
Given a point $w$ of $\NL(X,0)$, it follows from the valuative criterion of properness that $w$ has a non-empty center $\mathrm{c}_\pi(w)$ on $X_\pi$.
The center $\mathrm{c}_\pi(w)$ is the biggest irreducible subvariety $Y$ of $X_\pi$ such that for every regular function $f$ on $X$ we have $w(f)\geq0$ and $w(f)>0$ if and only if $f$ vanishes on $Y$.
If $\mathrm{c}_\pi(w)$ is a whole component $E_v$ of $\pi^{-1}(0)$ or a point of $E_v$ that is smooth in $\pi^{-1}(0)$, set $r_\pi(w)=v$.
If $\mathrm{c}_\pi(w)=p$ lies on the intersection of two components $E_v$ and $E_{v'}$ of $\pi^{-1}(0)$, set $r_\pi(w)=v_{\pi,p,\omega}$, with $\omega=w(y)/\big(w(x)+w(y)\big)$, where as above $x$ and $y$ are local equations for $E_v$ and $E_{v'}$ at $p$.

The continuous retractions $r_\pi$ induce then a natural homeomorphism
\[
\NL(X,0)\stackrel{\cong}{\longrightarrow}\varprojlim_\pi \Gamma_\pi
\] 
from $\NL(X,0)$ to the inverse limit of the dual graphs $\Gamma_\pi$, where $\pi$ ranges over the filtered family of good resolutions of $(X,0)$ (see \cite[Theorem 2.25]{GignacRuggiero2017} and \cite[Theorem 7.9]{Jonsson2015}).
This means that the non-archimedean link $\NL(X,0)$ can be thought of as a universal dual graph, making it a very convenient object for studying the combinatorics of the resolutions of $(X,0)$.
 
Dual resolution graphs can be endowed with a natural metric as follows.
Let $\pi \colon X_\pi \to X$ be a good resolution of $(X,0)$ which factors through the blowup of the maximal ideal of $(X,0)$.
	We define a metric on $\Gamma_\pi$ by declaring the length of an edge $e=[v,v']$ to be 
 \[
\length(e) = \frac{1}{ m_v m_{v'}}.
 \]

\begin{remark}\label{remark_length_dehn_screw}
This metric has a natural geometric interpretation following \cite[Theorem 7.3.(iv)]{MatsumotoMontesinos2011}.
Indeed, if the edge $e=[v,v']$ corresponds to a point $P$ of $E_v\cap E_{v'}$, which in turn corresponds to a union of annuli in the Milnor fiber of a generic linear form on $(X,0)$, then the opposite $-\length(e) = - \frac{1}{ m_v m_{v'}}$ of the length of $e$ coincides with the screw number of the representative of the monodromy automorphism (which is a product of Dehn twists) on these annuli. 
This observation, which will be explained in detail in Section~\ref{subsection_monodromy_lengths}, will play an important role in the proof of our main result. 
\end{remark}

\begin{example}\label{example:E8_metric_graph}
	Consider again the singularity $(X,0)=(E_8,0)$ of Example~\ref{example:E8_graph} and its minimal good resolution $\pi \colon X_{\pi} \to X$. 
	In Figure~\ref{fig:E8 multiplicities lengths}, the vertices of $\Gamma_\pi$ are decorated with the multiplicities of the corresponding exceptional components (in parenthesis), which can be computed by choosing a generic linear form $h \colon (X,0) \to (\C,0)$ on $(X,0)$, for example $h=z$.
	The edges of $\Gamma_\pi$ are decorated with the corresponding lenghts. 
	For example, we can observe that the length of the path from $v_0$ to $v_7$ is $1/6+1/12+1/20+1/30+1/18=7/18$.
	The $\cal L$-node $v_0$ is decorated with one arrow, representing the fact that the strict transform $h^*$ of $h$ on $X_\pi$ is an irreducible curve passing through the divisor $E_{v_0}$. 
	\begin{figure}[h] 
		\centering
		\begin{tikzpicture}
		\draw[thin,>-stealth,->](4,0)--+(0.7,0.7);

		\draw[thin ](-2,0)--(4,0);
		\draw[thin ](0,0)--(0,1);
		\draw[fill ] (-2,0)circle(2pt);
		\draw[fill ] (-1,0)circle(2pt);
		\draw[fill ] (0,0)circle(2pt);
		\draw[fill ] (1,0)circle(2pt);
		\draw[fill ] (2,0)circle(2pt);
		
		\draw[fill ] (3,0)circle(2pt);
		
		\draw[fill ] (4,0)circle(2pt);
		
		\draw[fill ] (0,1)circle(2pt);

		\node(a)at(5,0.7){ $h^*$};
		\node(a)at(3.5,0.3){ $\frac{1}{6}$};
		\node(a)at(2.5,0.3){ $\frac{1}{12}$};
		\node(a)at(1.5,0.3){ $\frac{1}{20}$};
		\node(a)at(0.6,0.3){$\frac{1}{30}$};
		\node(a)at(-0.6,0.3){$\frac{1}{24}$};
		\node(a)at(-1.5,0.3){ $\frac18$};
		\node(a)at(0.2,0.6){ $\frac{1}{18}$};
		
		\node(a)at(-2,-0.4){$(2)$};
		\node(a)at(-1,-0.4){$(4)$};
		\node(a)at(0,-0.4){$(6)$};
		\node(a)at(1,-0.4){$(5)$};
		\node(a)at(2,-0.4){$(4)$};
		\node(a)at(3,-0.4){$(3)$};
		\node(a)at(4,-0.4){$(2)$};
		\node(a)at(-0.4,1){$(3)$};
		
		\end{tikzpicture} 
		\caption{ }\label{fig:E8 multiplicities lengths}
	\end{figure}
\end{example}
 
Let $\pi\colon X_\pi\to X$ be a good resolution of $(X,0)$ that factors through the blowup of the maximal ideal of $(X,0)$, let $p$ be a point of the exceptional divisor $\pi^{-1}(0)$ of $\pi$ that lies on the intersection of two irreducible components $E_v$ and $E_{v'}$ of $\pi^{-1}(0)$, and let $\pi'\colon X_{\pi'}\to X$ be the good resolution of $(X,0)$ obtained from $\pi$ by blowing up $X_\pi$ at the point $p$.
Then the exceptional component $E_w$ of $(\pi')^{-1}(0)$ arising from the blowup of $p$ has multiplicity $m_w=m_v+m_{v'}$. 
Since ${1}/{m_v(m_v+m_{v'})}+{1}/{m_{v'}(m_v+m_{v'})}={1}/{m_vm_{v'}}$, this means that the canonical inclusion $\Gamma_{\pi}\hookrightarrow\Gamma_{\pi'}$, which is bijective at the level of the underlying topological spaces, is also an isometry.
Therefore, by passing to the limit, the metrics on the dual graphs $\Gamma_\pi$ define a metric on the non-archimedean link $\NL(X,0)$.
We call this metric the {\it skeletal metric} on $\NL(X,0)$.


\subsection{Generic projections} \label{subsec:projection}
We will study the surface germ $(X,0)$ using suitable projections to $(\C^2,0)$, making use of a classical notion of generic projection due to Teissier.

Fix an embedding of $(X,0)$ in some smooth germ $(\C^n,0)$, and consider the morphism $\ell_{\cal D}\colon(X,0)\to(\C^2,0)$ obtained by restricting to $X$ the projection along a $(n-2)$-dimensional linear subspace $\cal D$ of $\C^n$. 
Whenever $\ell_{\cal D}$ is finite, the associated \emph{polar curve} $\Pi_{\cal D}$ is the closure in $(X,0)$ of the ramification locus of $\ell_{\cal D}$ in $X\setminus\{0\}$. 
The Grassmannian variety of $(n-2)$-planes in $\C^n$ contains a dense open set $\Omega$ such that $\ell_{\cal D}$ is finite and the family $\{\Pi_{\cal D}\}_{\cal D\in\Omega}$ is well behaved (for example, it is equisingular in a strong sense).
The precise definition, which can be found in \cite[Section 2]{NeumannPedersenPichon2018}, builds on work of Teissier (see in particular \cite[Lemme-cl\'e V 1.2.2]{Teissier1982}).
We say that a morphisms $\ell\colon (X,0)\to(\C^2,0)$ is a \emph{generic projection} of $(X,0)$ if $\ell=\ell_{\cal D}$ for some $\cal D$ in $\Omega$.

The smallest modification 
of $(X,0)$ that resolves the basepoints of the family of polar curves $\{\Pi_{\cal D}\}_{\cal D\in\Omega}$ is the \emph{Nash transform} of $(X,0)$ (see {\cite[Part III, Theorem 1.2]{Spivakovsky1990}} and {\cite[Section 2]{Gonzalez-Sprinberg1982}}), which is the blowup of $(X,0)$ along its Jacobian ideal.
We call the subset of $\NL(X,0)$ consisting of the divisorial valuations associated with the exceptional components of the Nash transform of $(X,0)$ the set of \emph{$\cal P$-nodes} of $(X,0)$.
In another terminology, this is the set of the Rees valuations of the Jacobian ideal of $X$.

If $\pi \colon X_{\pi} \to X$ is a good resolution of $(X,0)$ which factors through its Nash transform, we say that  a projection $\ell \colon (X,0) \to (\C^2,0)$  is {\it generic with respect to} $\pi$ if it is a generic projection in the sense above and the strict transform of its polar curve via $\pi$ only intersects the components of $\pi^{-1}(0)$ corresponding to the $\cal P$-nodes.

A generic projection   $\ell \colon (X,0) \to (\C^2,0)$ induces a natural morphism
\[
\widetilde{\ell} \colon \NL(X,0) \to \NL(\C^2,0).
\]
Indeed,  $\ell$ induces a map $\ell^\# \colon \widehat{\mathcal O_{X,0}} \to \widehat{\mathcal O_{\C^2,0}}$, hence a point of $\NL(X,0)$ (which is a semivaluation on $\widehat{\mathcal O_{X,0}}$) gives rise to a point of $\NL(\C^2,0)$ simply by pre-composing the semivaluation with $\ell^\#$. 
A concrete way to compute $\widetilde\ell(v)$ for a divisorial valuation $v\in\NL(X,0)$ goes as follows. 
Take a good resolution $\pi$ of $(X,0)$ such that $v$ is associated with an irreducible  component $E_v$ of $\pi^{-1}(0)$ and consider a generic pair of curvettes $\gamma$ and $\gamma'$ of $E_v$. 
Let $\sigma$ be the minimal sequence of blowups of $(\C^2,0)$ such that the strict transforms of $\ell(\gamma)$ and $\ell(\gamma')$ by $\sigma$ meet the exceptional divisor $\sigma^{-1}(0)$ at distinct points. 
Then, being generic, they meet $\sigma^{-1}(0)$ at smooth points along the exceptional curve $C_{\nu}$ created at the last blowup. 
It can be seen via a standard Hirzebruch--Jung argument that neither the morphism $ \sigma$ nor the divisorial valuation $\nu \in \NL(\C^2,0)$ depend on the choice of the generic pair $\gamma, \gamma'$, and indeed we have $\widetilde{\ell}(v) =\nu $.

Since $\ell$ is finite, it is not hard to show that $\widetilde{\ell}$ is a branched covering, that is a finite map that is a finite topological covering out of a nowhere dense \emph{ramification locus}.
On the level of dual graphs, if $\pi\colon X_\pi\to X$ is a good resolution of $(X,0)$ that factors through its Nash transform and $\sigma\colon Y_\sigma\to \C^2$ is any sequence of blowups of $\C^2$ above the origin such that $\widetilde\ell\big(V(\Gamma_\pi)\big)\subset V(\Gamma_\sigma)$, then $\widetilde\ell(\Gamma_\pi)$ is the subgraph of $\Gamma_\sigma$ spanned by the vertices in $\widetilde\ell\big(V(\Gamma_\pi)\big)$, and the restriction of $\widetilde\ell$ to $\Gamma_\pi$ is a (ramified) covering of graphs onto its image.


\subsection{Laplacians on metric graphs}
\label{subsection_preliminaries_laplacians}

We will briefly recall some basic notions of divisor theory on metric graphs.

For us a \emph{graph} is a   finite and connected metric graph 
\[
\Gamma = \big( V(\Gamma), E(\Gamma), l\colon E(\Gamma) \to \Q_{>0} \big),
\]
where $V(G)$ is the set of vertices of $\Gamma$, $E(G)$ is the set of its edges, and $l$ attaches a length to each edge of $\Gamma$.
We allow $\Gamma$ to have loops and multiple edges.
We will freely identify $\Gamma$ with its geometric realization, which is the metric space whose metric is induced by the lengths of its edges. 

A \emph{divisor} $D=\sum a_v[v]$ of $\Gamma$ is a finite sum of points of $\Gamma$ with integral coefficients $a_v \in \Z$.
We also denote by  $D(v)$  the coefficient  $a_v$  of a divisor $D$ at a point $v$ of $\Gamma$, and by $\Div(\Gamma)=\bigoplus_{v\in\Gamma}\Z[v]$ the abelian group of divisors of $\Gamma$.
The \emph{degree} $\deg(D)$ of $D$ is the integer $\deg(D)=\sum_{v\in\Gamma}D(v)$.

A function $F\colon \Gamma\to \R$ is said to be \emph{piecewise linear} if $F$ is a continuous piecewise affine map with integral slopes  (with respect to the metric induced by $l$ on $\Gamma$) and $F$ has only finitely many points of non-linearity on each edge of $\Gamma$.

\begin{definition} \label{definition:laplacian}
	If $F\colon \Gamma\to\R$ is a piecewise linear map, its \emph{Laplacian} $\Delta_\Gamma(F)$ is the divisor of $\Gamma$ whose coefficient $\Delta_\Gamma(F)(v)$ at a point $v$ of $\Gamma$ is the sum of the outgoing slopes of $F$ at $v$.
\end{definition}

\begin{example}\label{example:E8_Laplacien}
	Consider the metric graph $\Gamma_\pi$ associated with the dual graph of the minimal resolution of the singularity $E_8$, as described in Examples~\ref{example:E8_graph} and \ref{example:E8_metric_graph}. 
	Let $F$ be the function on $\Gamma_\pi$ which is linear on its edges and such that $\big(F(v_i)\big)_{i=0}^7=\big(1,\frac{4}{3},\frac{3}{2},\frac{8}{5},\frac{5}{3},\frac{7}{4},2,2\big)$; we will see in Example~\ref{example:E8_inner_rate} that this function $F$ is the inner rate function of $E_8$.
	Then it is easy to see that $F$ grows linearly with slope $2$ on the path from $v_0$ to $v_6$, while it grows linearly with slope $6$ on the edge $[v_{4},v_{7}]$.
	Therefore, its Laplacian is the divisor $\Delta_{\Gamma_\pi}(F)=2[v_0]+6[v_4]-2[v_6]-6[v_{7}]$.
\end{example}

Observe that, as we do not require $F$ to be linear inside the edges of $\Gamma$, its Laplacian $\Delta_\Gamma(F)$ is not necessarily supported on $V(\Gamma)$.
Moreover, since every segment $[v,v']$ in $\Gamma$ such that $F|_{[v,v']}$ is linear contributes with the same slope but opposite signs to the Laplacian $\Delta_\Gamma(F)(v)$ and $\Delta_\Gamma(F)(v')$ at $v$ and $v'$ respectively, the Laplacian $\Delta_\Gamma(F)$ of $F$ is a divisor of degree 0.

A function $f\colon \Gamma'\to\Gamma$ between two metric graphs induces a natural map $f_*\colon \Div(\Gamma')\to\Div(\Gamma)$ defined by sending a divisor $D=\sum_{v'\in\Gamma'} a_{v'}[v']$ to the divisor $f_*D=\sum_{v\in\Gamma}b_v[v]$, where $b_v=\sum_{v'\in f^{-1}(v)}a_v$.

Similarly, we call \emph{divisor} on $\NL(X,0)$ a finite sum of points of $\NL(X,0)$ with integral coefficients, and we denote by $\Div\big(\NL(X,0)\big)=\bigoplus_{v\in\NL(X,0)}\Z[v]$ the abelian group of divisors on $\NL(X,0)$.
Then, if $\pi\colon X_\pi\to X$ is a good resolution of $(X,0)$, the retraction map $r_\pi\colon\NL(X,0)\to \Gamma_\pi$ induces a map of divisors $(r_\pi)_*\colon \Div\big(\NL(X,0)\big) \to \Div(\Gamma_\pi)$ by the same formula as above.


\section{The inner rate function} \label{sec:inner rates function}

Let $(X,0) \subset (\C^n,0)$ be a surface germ with an isolated singularity.
In this section we define the inner rate function on a non-archimedean link $\NL(X,0)$ and prove its basic properties.

We will use the big-Theta asymptotic notations of Bachmann--Landau in the following form: given two function germs $f,g\colon \big([0,\infty),0\big)\to \big([0,\infty),0\big)$ we say $f$ is   \emph{big-Theta} of $g$ and we write   $f(t) = \Theta \big(g(t)\big)$ if there exist real numbers $\eta>0$ and $K >0$ such that for all $t$, if $f(t)\leq \eta$ then ${K^{-1}}g(t) \leq f(t) \leq K g(t)$. 
   
Let $(\gamma,0)$ and $(\gamma',0)$ be two distinct germs of complex curves on the surface germ $(X,0)\subset(\C^n,0)$, and denote by $S_{\epsilon}$ the sphere in $\C^n$ having center $0$ and radius $\epsilon>0$.
Denote by $d_i$ the inner distance on $(X,0)$.
The \emph{inner contact} between $\gamma$ and $\gamma'$  is the rational number $q_i=q^X_i(\gamma, \gamma')$  defined by 
\[
d_i \big(\gamma \cap S_{\epsilon}, \gamma' \cap S_{\epsilon}\big) = \Theta(\epsilon^{q_i}).
\]

\begin{remark}
While the existence of the inner contact $q^X_i(\gamma, \gamma')$ and its rationality can be deduced from the work of \cite{KurdykaOrro1997}, in the case that interests us this can also be seen as a consequence of the next lemma.
\end{remark}

The following lemma is fundamental, as it will allow us to define the inner rate of a divisorial valuation of $\NL(X,0)$.

Recall that if $\pi \colon X_{\pi} \to X$ is a resolution of $X$ and if $E$ is an irreducible component of $\pi^{-1}(0)$, a {\it curvette} of $E$ is  a smooth curve germ $(\gamma,p)$ in $X_{\pi}$, where $p$ is a point of $E$ which is a smooth point of  $\pi^{-1}(0)$  and such that $\gamma$ and $E$ intersect transversely.

\begin{lemma} \label{lemma_definition_inner_rate}
Let $v \in\NL(X,0)$ be a divisorial valuation on $(X,0)$ and let $\pi \colon  X_\pi \to X$ be a good resolution of $(X,0)$ which factors through the blowup of the maximal ideal and through the Nash transform of $(X,0)$ and such that $v$ is the divisorial valuation associated with an irreducible component $E_{v}$ of $\pi^{-1}(0)$. Consider two curvettes $\gamma^*$ and $\widetilde{\gamma}^*$ of $E_{v}$ meeting it at distinct points, and write $\gamma = \pi(\gamma^*)$ and $\widetilde{\gamma} = \pi(\widetilde{\gamma}^*)$.
Then the inner contact $q^X_i(\gamma, \widetilde{\gamma})$ between $\gamma$ and $\widetilde{\gamma}$ only depends on $v$ and not on the choice of $\pi$, $\gamma^*$, and $\widetilde{\gamma}^*$. Moreover, $m_{v} q^X_i(\gamma, \widetilde{\gamma})$ is an integer,  $q^X_i(\gamma, \widetilde{\gamma}) \geq 1$, and if $\ell \colon (X,0) \to (\C^2,0)$ is a generic projection with respect to $\pi$ we have $q_i^X(\gamma, \widetilde{\gamma}) = q_i^{\C^2}\big(\ell(\gamma), \ell(\widetilde{\gamma})\big).$
\end{lemma}

\begin{definition} \label{definition:inner_rate}
	We denote by $q_{v}$ the rational number $q^X_i(\gamma, \widetilde{\gamma})$, and call it the {\it inner rate} of $v$.
\end{definition}

\begin{remark} \label{remark:maxmin} 
It is worth noticing that the knowledge of all the inner rates $q_v$ allows one to compute the inner contact between any two complex curve germs $(\gamma,0)$ and $(\gamma',0)$ on $(X,0)$. 
Indeed, assume that the good resolution $\pi\colon X_\pi\to X$ of $(X,0)$ also resolves the complex curve $\gamma\cup\gamma'$ and let $v$ and $v'$ be the vertices of $\Gamma_{\pi}$ such that $\gamma^* \cap E_{v} \neq \emptyset$ and ${\gamma'}^* \cap E_{v'} \neq \emptyset$ respectively.
Then $q^X_{i}(\gamma, \gamma')= q_{v,v'}$, where $q_{v,v'}$ is the maximum, taken over all injective paths $\gamma$ in $\Gamma_\pi$ between $v$ and $v'$, of the minimum of the inner rates of the vertices of $\Gamma_\pi$ contained in $\gamma$.
We refer to \cite[Proposition 15.3]{NeumannPedersenPichon2018} for details.
\end{remark}

\begin{proof}[Proof of Lemma \ref{lemma_definition_inner_rate}] In the course of the proof we will use the {outer contact}   between $\gamma$ and $\gamma'$, which  is the rational number $q_o=q_o(\gamma, \gamma')$  defined by 
$d_o \big(\gamma \cap S_{\epsilon}, \gamma' \cap S_{\epsilon}\big) = \Theta(\epsilon^{q_o})$, where  $d_o$ is the {outer distance} $d_o(x,y) = || x-y||_{\C^n}$. 
It is simple to see that $q_o$ can also be defined by $d_o \big(\gamma \cap \{ z = \epsilon\}, \gamma' \cap \{ z = \epsilon\}\big) = \Theta(\epsilon^{q_o})$, whenever $z\colon (X,0)\to(\C,0)$ is a generic linear form on $(X,0)$.

First observe that in the smooth case the result comes from classical theory of plane curves singularities. 
Indeed, in $\C^2$ the inner and outer metrics coincide, and if $C_i$ is an exceptional component of a composition of blowups of points $\sigma \colon Y \to \C^2$ starting with the blowup of the origin, then for every pair of  distinct curves  $(\delta,0)$ and $(\delta',0)$ whose strict transforms by $\sigma$ meet $C_i$ at distinct smooth points of $\sigma^{-1}(0)$, the contact $q^{\C^2}_i(\delta,\delta') = q_o(\delta,\delta') $ coincides with the contact exponent between their Puiseux series (see for example \cite[page~401]{GarciaBarrosoTeissier1999}) and does not depend on the choice of the pair $\delta, \delta'$.  
 
Let us now focus on the general case. 
Denote by $p$ the point of $E_{v}$ where $\gamma^*$ passes through and consider coordinates $(z_1,\ldots, z_n)$ of $\C^n$ such that $\ell|_X\colon (z_1,\dots,z_n)\to  (z_1, z_2)$ is a generic projection with respect to $\pi$ and such that the strict transform $\Pi^*$ of the polar curve $\Pi$ of $\ell$ by $\pi$ does not pass through $p$, nor do the strict transforms of the curves $X\cap \{z_i=0\}$ for all $i=1,\ldots,n$. 
Choose local coordinates $(u_1,u_2)$ centered at $p$ such that $E_{v}$ has local equation $u_1=0$, $\gamma^*$ has local equation $u_2=0$, and such that $(z_1 \circ \pi)(u_1,u_2) =  u_1^{m_v}$. 
One can then express the other coordinates in terms of $(u_1,u_2)$ as follows: 
\begin{equation}\label{eqn:local_coordinates_compute_inner_rate}
( z_2  \circ \pi)(u_1,u_2)  =  u_1^{m_v} f_{2,0}(u_1) + u_1^{m_vq_2} \sum_{j \geq 1 } u_2^j f_{2,j}(u_1),
\end{equation}
and, for $i=3,\ldots, n$,
\[
(z_i  \circ \pi)(u_1,u_2)  =  u_1^{m_v} f_{i,0}(u_1) + u_1^{m_vq_i} \sum_{j \geq 1 } u_2^j f_{i,j}(u_1),
\]
for suitable choices of $ f_{i,j}(u_1) \in \C\{u_1\}$ and $1\le q_i \in \Q$.
Without loss of generality, we can assume that the $q_i$ are the biggest rational numbers such that expressions as above exist, and replacing $z_2$ by a generic combination of the functions  $z_1,\ldots, z_n$ we can assume that $q_i \geq q_2 \geq 1$ for all $i \geq 3$.  
We will prove the following claim:
\vskip0,1cm\noindent
{\bf Claim 1.} For any curvette  $\widetilde{\gamma}^*$ of $E_v$ meeting $E_v$ at a point distinct from $p$, we have $q_i^X(\gamma, \widetilde{\gamma}) = q_2$. 
\vskip0,1cm
Denote by $\gamma^*_t$ the curvette of $E_{v}$ defined by the equation $u_2=t$ (so that in particular we have $\gamma_0 = \gamma$), set $\gamma_t=\pi(\gamma^*_t)$, and let $D$ be a  disc neighborhood of $p$ in $E_{v}$ which is contained in a neighborhood on which the local coordinates $(u_1,u_2)$ are defined. 
Then for every $t \in D \setminus\{0\}$ we have  
\[
d_o\big(\gamma \cap \{ z_1 = \epsilon\},  {\gamma}_t  \cap \{ z_1 = \epsilon\}\big) = \Theta (\epsilon^{q_2}).
\]
Now, for every $t \in D \setminus\{0\}$, replacing ${\gamma}_t$ by any  other curvette $\widetilde{\gamma}^*$ passing through $(u_1,u_2)=(0,t)$ still gives 
\[  
d_o\big(\gamma \cap \{ z_1 = \epsilon\}, \widetilde{\gamma} \cap \{ z_1 = \epsilon\}\big) = 
\Theta (\epsilon^{q_2}),
\]
since $ \widetilde{\gamma}^*$ is defined by a parametrization of the form $u_2=t + h.o.$, where \emph{h.o.} denotes a sum of higher order terms.
In particular, we have $q_o(\gamma, \widetilde{\gamma}) =  q_2$.

On the other hand, let $\ell \colon (X,0) \to (\C^2,0)$ be a generic projection for $(X,0)$.  
By \cite[Proposition 3.3]{BirbrairNeumannPichon2014}, the local bilipschitz constant of the cover  $\ell$ is bounded outside $\pi(W)$, where $W$ is any analytic neighborhood of $\Pi^*$ in $X_\pi$. 
Since $\Pi^*$ does not pass through $p$, we can take $D$ and $W$ small enough that $D \cap W = \emptyset$. 
Therefore, the local bilipschitz constant of $\ell$ being bounded on $\bigcup_{t \in D} \gamma_t$, we have $q^X_i(\gamma, \widetilde{\gamma}) = q_i\big( \ell(\gamma), \ell(\widetilde{\gamma})\big)$   as long as $\widetilde\gamma^*$ passes through a point of $E_{v} \cap D$. 

Since the coincidence exponent between the curves $\ell(\gamma)$ and $\ell(\widetilde{\gamma})$  in $(\C^2,0)$ equals $q_2$, we deduce that $q^X_i(\gamma, \widetilde{\gamma}) =q_o\big( \ell(\gamma), \ell(\widetilde{\gamma})\big) = q_2$. 
This proves that $q^X_i(\gamma, \widetilde{\gamma}) =q_2$ does not depend on the choice of the curvette $\widetilde{\gamma}^*$ providing $\widetilde{\gamma}^* \cap E_{v}$  is in the neighborhood $D$ of $\gamma^*$. 

We now have to prove that $q_2$ does not depend on the point $p$, and that $q^X_i(\gamma, \widetilde{\gamma}) =q_2$ for any pair of curvettes $\gamma^*$ and $\widetilde{\gamma}^*$ meeting $E_{v}$ at distinct points.  
 
\vskip0,1cm \noindent
 {\bf Claim 2.} The contact order $q_2$ does not depend on the point $p$. 
 \vskip0,1cm 
Let $p'$ be another  smooth point of $\pi^{-1}(0)$ on $E_{v}$,  let $\delta^*$ be a curvette of $E_{v}$ through $p'$ and let ${\widetilde{\delta}}^*$ be a neighbor curvette. 
Let $\delta $ and $\widetilde{\delta}$ be the images through $\pi$ of $\delta^*$ and $\widetilde{\delta}^*$ respectively. 
Let $q_2(p')$ be the rate $q_2$ obtained as above by taking local coordinates centered at $p'$ instead of $p$. 
We can assume without loss of generality that the projection $\ell$ is also generic for these coordinates in the sense above. 
Then we have $q_o\big(\ell(\delta),\ell( \widetilde{\delta})\big) = q_2(p')$.  
Consider the minimal sequence of blowups  $\sigma \colon Y \to \C^2$  such that one of the  irreducible components $C$ of $\sigma^{-1}(0)$ corresponds to the valuation $\widetilde{\ell}(v)$ (where $\widetilde{\ell}$ is defined in Section \ref{subsec:projection}).  
Then the strict transforms of the four curves  $\gamma$, $\widetilde{\gamma}$,  $\delta$ and $\widetilde{\delta}$ intersect $C$ at four distinct points of $C$ which are smooth points of $\sigma^{-1}(0)$. 
Therefore, we have $q_2 = q_o(\gamma, \widetilde{\gamma}) = q_o(\delta, \widetilde{\delta}) = q_2(p')$. 
This proves Claim 2.

Let us now take  any  pair of curvettes $\gamma^*$ and  $\widetilde{\gamma}^*$ meeting $E_{v}$ at distinct points $p$ and $\widetilde{p}$. 
Since the contact of $\gamma$ with the $\pi$-image of any neighbor curvette equals $q_2$, we then have $q_i^X(\gamma,\widetilde{\gamma}) \leq q_2$. 
Moreover,  by compactness of $E$, we can choose a finite sequence of smooth points $p_1, p_2, \ldots, p_s$  of $\pi^{-1}(0)$ on $E_{v}$, such that $p_1=p$, $p_s=\widetilde{p}$ and for all $i=1,\ldots, s-1$,  $q^X_i(\gamma_i, \gamma_{i+1}) =  q_2$ where $\gamma_i^*$ and $\gamma_{i+1}^*$ pass through $p_i$ and $p_{i+1}$ respectively.  
We then have
\[
q_i^X(\gamma,\widetilde{\gamma})  \geq \min_{i=1,\ldots,s-1} \big(q_i^X(\gamma_i,\gamma_{i+1}) \big) = q_2.
\]
Putting this all together, we deduce that $q_i^X(\gamma,\widetilde{\gamma}) = q_2$.

Finally, observe that $q^X_i(\gamma, \widetilde{\gamma})$ only depends on $v$ and not on $E_{v}$, since all the computations performed above are unchanged if we first blowup a closed point of $E_{v}$.
\end{proof}

\begin{remark} 
The second claim in the proof of Lemma~\ref{lemma_definition_inner_rate} can also be proved by a computation using local coordinates in the resolution as follows.
Let $p'$ be another smooth point of $\pi^{-1}(0)$ on $E_{v}$. 
By the genericity of $z_1$ and $z_2$, we can assume that the normal form of equation \eqref{eqn:local_coordinates_compute_inner_rate} is also valid at $p'$, that is, there exists a coordinate system $(u'_1,u'_2)$ centered at $p'$ such that:
\[
\begin{aligned}
( z_1  \circ \pi)(u_1',u_2')  &=  u_1'^{m_v},\\
( z_2  \circ \pi)(u_1',u_2')  &=  u_1'^{m_v} g_{2,0}(u_1') + u_1'^{m_vq_v(p')} \sum_{j \geq 1 } u_2'^j g_{2,j}(u_1').
\end{aligned}
\]
Observe  that:
\[
(dz_1\circ \pi) \wedge (dz_2 \circ \pi) = u_1'^{m_v(1+q_v(p'))-1}\Big( \sum\nolimits_{j \geq 1 } j u_2'^{j-1} g_{2,j}(u_1')\Big) du'_1\wedge du'_2.
\]
We can interpret the exponent $m_v\big(1+q_v(p')\big)-1$ as the maximal order of the exceptional divisor that factors through $(dz_1\circ \pi) \wedge (dz_2 \circ \pi)$. 
Since this order is independent of the point $p' \in E_v$, we conclude that $q_v(p)=q_v(p')$ for every point $p'\in E_v$.
\end{remark}

The following result is what allows us to compute in a simple way the inner rate of any divisorial valuation of $(\C^2,0)$, and more generally that of any divisorial valuation of a singular germ $(X,0)$ if we know the inner rates of the vertices of the dual graph $\Gamma_\pi$ of a suitable good resolution $\pi$ of $(X,0)$.

\begin{lemma}\label{lem:inner rate } 
Let $\pi \colon X_\pi \to X$ be a good resolution of $(X,0)$ that factors through the blowup of the maximal ideal and through the Nash transform of $(X,0)$. 
Let $p$ be a point of the exceptional divisor $\pi^{-1}(0)$ and let $E_w$ be the exceptional component created by the blowup of $X_\pi$ at $p$. 
Then:
\begin{enumerate}
\item \label{lem:inner rate  SMOOTH PT}
 If $p$ is a smooth point of $\pi^{-1}(0)$ and $E_v$ is the irreducible component of $\pi^{-1}(0)$ on which $p$ lies, then
 \[
 m_w=m_v  \;  \hbox{ and }  \;  q_w = q_v + \frac{1}{m_v}.
 \]
\item \label{lem:inner rate  SINGULAR PT}
 If $p$ lies on the intersection of two irreducible components $E_v$  and $E_{v'}$ of $\pi^{-1}(0)$, then
 \[
 m_w=m_v + m_{v'} \; \hbox{ and } \; q_w = \frac{q_v m_v + q_{v'}m_{v'}}{m_v + m_{v'}}.
 \]
\end{enumerate}
\end{lemma}

\begin{proof} 
Assume first that $p$ is a smooth point of $\pi^{-1}(0)$. We use again the notations of the proof of Lemma \ref{lemma_definition_inner_rate}: in local coordinates $(u_1,u_2)$ centered at $p$, we have $(z_1 \circ \pi)(u_1,u_2) =  u_1^{m_v}$ and 
\[
( z_2  \circ \pi)(u_1,u_2)  =  u_1^{m_v} f_{2,0}(u_1) + u_1^{m_vq_2} \sum_{j \geq 1 } u_2^j f_{2,j}(u_1).
\]
Let us prove that $f_{2,1}(0) \neq 0$. 
Let $\Pi$ be the polar curve of $\ell$. Its total transform   $\pi^{-1}(\Pi)$  by $\pi$ is the critical locus of $\ell \circ \pi$. 
The Jacobian matrix of $\ell \circ \pi$ is: 
\[
\left(
\begin{array}{cc}
m_v u_1^{m_v-1}  &0     \\
\ast &      u_1^{m_v q_2}\big( f_{2,1}(u_1) + 2u_2 f_{2,2}(u_1) + \cdots  \big)
\end{array}
\right)
\]
thus $\pi^{-1}(\Pi)$ has equation $m_v u_1^{m_v + m_v q_2-1}  \big( f_{2,1}(u_1) + 2u_2 f_{2,2}(u_1) + \cdots  \big)=0$ and the strict transform $\Pi^*$ has equation  $  f_{2,1}(u_1) + 2u_2 f_{2,2}(u_1) + \cdots =0$. 
Since $\Pi^*$ does not pass through $p$, this implies that $f_{2,1}(0) \neq 0$.   

Let $e_p$ be the blowup of $X_\pi$ at $p$. 
In the coordinates chart $(u'_1,u'_2) \mapsto (u'_1,u'_2u'_1)$, we have $(z_1 \circ \pi \circ e_p)(u'_1,u'_2) = {u'_1}^{m_v}$ and 
\[
( z_2  \circ \pi \circ e_p)(u'_1,u'_2)  =  {u'_1}^{m_v} f_{2,0}(u'_1) + {u'_1}^{m_vq_2 + 1} \sum_{j \geq 1 } {u'_2}^j {u'_1}^{j-1} f_{2,j}(u'_1).
\]
Therefore $m_w = m_v$. 
Fixing $u'_2 \neq 0$ and comparing with the equation~\eqref{eqn:local_coordinates_compute_inner_rate} of the proof of Lemma \ref{lemma_definition_inner_rate}, we obtain $m_w q_w = m_vq_2 + 1$ by using Claim 1 of the proof of Lemma \ref{lemma_definition_inner_rate}. 
This proves \ref{lem:inner rate  SMOOTH PT}. 

Assume now that $p$  is an intersecting point between two irreducible components $E_v$  and $E_{v'}$. 
In local coordinates $(u_1,u_2)$ centered at $p$, we can assume without loss of generality that $(z_1 \circ \pi)(u_1,u_2) =  u_1^{m_v}u_2^{m_{v'}}$ and $ ( z_i  \circ \pi)(u_1,u_2)    \in \C\{u_1, u_2\}$; in particular, we consider their Taylor expansion
\[
( z_i  \circ \pi)(u_1,u_2) = \sum_{\alpha \in \mathbb{N}^2} T_{i\alpha} u_1^{\alpha_1}u_2^{\alpha_2}.
\]
Now, let $\Lambda =\big\{ \alpha \in \mathbb{N}^2\,\big|\, \exists q\in \mathbb{Q} \text{ such that } q\cdot \alpha = (m_v,m_{v'}) \big\}$. Since $( z_i  \circ \pi)(u_1,u_2)$ are convergent power series, a sub-series is also convergent, and therefore
\[
\widetilde{g}_i(u_1,u_2) = \sum_{\alpha \in \Lambda} T_{i\alpha} u_1^{\alpha_1}u_2^{\alpha_2} 
\]
is an analytic function. 
Furthermore, we know that $u_1^{m_v}u_2^{m_{v'}}$ divides $( z_i  \circ \pi)(u_1,u_2)$, so $\widetilde{g}_i(u_1,u_2) = u_1^{m_v}u_2^{m_{v'}} g_i(u_1,u_2)$ and:
\[
( z_i  \circ \pi)(u_1,u_2) = u_1^{m_v}u_2^{m_{v'}} g_i(u_1,u_2) +  u_1^{b_v}u_2^{b_{v'}}h_i(u_1,u_2)
\]
where $b_v \geq m_v$, $b_{v'}\geq m_{v'}$ and, without loss of generality, $h_2(u_1,u_2)$ is not identically zero over $E_{v}$ and $E_{v'}$. 
Now consider a point $(0,a_2)$ (the same computation works for $(a_1,0)$) and let us compare this Taylor expansion with the normal form given in equation \eqref{eqn:local_coordinates_compute_inner_rate}. 
Consider the analytic change of coordinates:
\[
u = u_1(u_2 - a_2)^{\frac{m_{v'}}{m_{v}}}, \quad v= u_2
\]
which is centered at $(0,a_2)$ and note that $(z_1 \circ \pi)(u_1,u_2) =  u^{m_v}$. 
Furthermore, it follows from direct computation and the definition of $\Lambda$ that:
\[
u_1^{m_v}u_2^{m_{v'}} g_i(u_1,u_2) = u^{m_v}g_i(u)
\]
so, these terms contribute only to the terms $f_{i0}$ of the normal form. 
Furthermore, if $\alpha \in \mathbb{N} \setminus \Lambda$, then:
\[
u_1^{\alpha_1}u_2^{\alpha_2} = u^{\alpha_1} \left(v-a_2 \right)^{\frac{\alpha_2 m_{v'} - \alpha_1 m_v}{m_{v'}}} = u^{\alpha_1}U_{\alpha}(u,v)
\]
where $U_{\alpha}(u,v)$ is a non-constant unit whose derivative in respect to $v$ is non-zero. 
By comparing the normal forms, it follows, that $b_{v} = m_v q_{v}$ (and by the analogous argument, that $b_{v'}= m_{v'} q_{v'}$), which yields to:
\[
( z_i  \circ \pi)(u_1,u_2) = u_1^{m_v}u_2^{m_{v'}} g_i(u_1,u_2) +  u_1^{ m_v q_{v}}u_2^{m_{v'} q_{v'}}h_i(u_1,u_2)
\]
Now, since $\pi$ factors through the Nash transform of $(X,0)$, we can suppose without loss of generality that the polar curve in respect to $(z_1,z_2)$ does no pass through $p$. 
This implies that the following $2$-form has support in the exceptional divisor:
\[
d ( z_1 \circ \pi) \wedge d(z_2 \circ \pi) = d(u_1^{m_v}u_2^{m_{v'}})\wedge d\big(u_1^{q_vm_v}u_2^{q_{v'}m_{v'}}h_2(u_1,u_2)\big)
\]
Denoting by $\partial = m_{v'} u_1 \partial_{u_1} - m_{v}u_2\partial_{u_2}$, we get that:
\[
d ( z_1 \circ \pi) \wedge d(z_2 \circ \pi) =u_1^{m_v-1}u_2^{m_{v'}-1} \partial\big(u_1^{q_vm_v}u_2^{q_{v'}m_{v'}}h_2(u_1,u_2)\big)  du_1 \wedge du_2
\]
this implies that $h_2$ must be a unit at $p$. 
Then part \ref{lem:inner rate  SINGULAR PT} of the Lemma follows from a simple direct computation.
\end{proof}

\begin{example} \label{example:E8_inner_rate}  
	Let us consider again the singularity $X=E_8$ of example \ref{example:E8_graph}. 
	The inner rates on the dual embedded resolution graphs $\Gamma_\sigma$ and $\Gamma_{\sigma'}$ of the discriminant curve $\Delta$ can be easily computed as contact exponent between Puiseux series of neighbor curvettes, or simply using Lemma~\ref{lem:inner rate }.
	They are then lifted to the graph $\Gamma_\pi$ via $\widetilde\ell$ thanks to Lemma~\ref{lemma_definition_inner_rate}, since $\ell$ is generic with respect to $\pi$. 
	This proves that the inner rate function on the vertices of $\Gamma_{\pi}$ coincides with the function introduced in Example \ref{example:E8_Laplacien} and depicted in Figure~\ref{fig:inner rates E_8}.  
\begin{figure}[h] 
		\centering
		\begin{tikzpicture}  
		
		\draw[thin ](-2,0)--(4,0);
		\draw[thin ](0,0)--(0,1);
		\draw[fill ] (-2,0)circle(2pt);
		\draw[fill ] (-1,0)circle(2pt);
		\draw[fill ] (0,0)circle(2pt);
		\draw[fill ] (1,0)circle(2pt);
		\draw[fill ] (2,0)circle(2pt);
		
		\draw[fill ] (3,0)circle(2pt);
		
		\draw[fill ] (4,0)circle(2pt);
		
		\draw[fill ] (0,1)circle(2pt);

		\node(a)at(4,0.4){ $\mathbf 1$};
		\node(a)at(3,0.4){ $\frac{\mathbf 4}{\mathbf 3}$};
		\node(a)at(2,0.4){ $ \frac{\mathbf 3}{\mathbf 2}$};
		\node(a)at(1,0.4){ $ \frac{\mathbf 8}{\mathbf 5}$};
		\node(a)at(0.2,0.4){$   \frac{\mathbf 5}{\mathbf 3}$};
		\node(a)at(-1,0.4){$   \frac{\mathbf 7}{\mathbf 4}$};
		\node(a)at(-2,0.4){ $\mathbf 2$};
		\node(a)at(0.3,1){ $\mathbf 2$};
		
		\node(a)at(-2,-0.4){$v_6$};
		\node(a)at(-1,-0.4){$v_5$};
		\node(a)at(0,-0.4){$v_4$};
		\node(a)at(1,-0.4){$v_3$};
		\node(a)at(2,-0.4){$v_2$};
		\node(a)at(3,-0.4){$v_1$};
		\node(a)at(4,-0.4){$v_0$};
		\node(a)at(-0.4,1){$v_7$};

		\end{tikzpicture} 
		\caption{The inner rates of the vertices of the graph $\Gamma_\pi$}\label{fig:inner rates E_8}
	\end{figure}
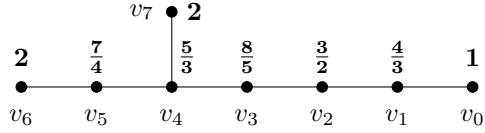
\end{example}

The starting point of our study of inner rates via potential theory on dual graphs is the following result, which states that the inner rates extend to a continuous function, and that it is piecewise-linear with respect to the metric defined in Section~\ref{subsection_preliminaries_laplacians}.

\begin{lemma} \label{lemma_inner_rate_function_NL}
There exists a unique continuous function
\[
\cal I_X \colon \NL(X,0) \to  \R_{\geq 1} \cup \{\infty\}
\]
such that $\cal I_X (v) = q_{v}$ for every divisorial point $v$ of $\NL(X,0)$.
If $\pi$ is a good resolution of $(X,0)$ that factors through the blowup of the maximal ideal and the Nash transform of $(X,0)$, then $\cal I_X $ is linear on the edges of $\Gamma_{\pi}$ with integral slopes.
\end{lemma}

\begin{proof}
Let $\pi$ be as in the statement.  
We only need to show that the inner rates extend uniquely to a continuous map on $\Gamma_\pi$ which is linear on its edges with integral slopes, as the first part of the statement will then follow immediately from the description of $\NL(X,0)$ as inverse limit of dual graphs.
The fact that the slopes are integer can be verified directly, as on an edge $e=[v,v']$ the slope is $(q_{v'}-q_v)/\length(e)=(q_{v'}-q_v)m_vm_{v'}$, which is an integer by Lemma~\ref{lemma_definition_inner_rate}.
To prove the linearity on the edges, since the subset of $\Gamma_\pi$ consisting of the divisorial points is dense in $\Gamma_\pi$, it is enough to show that the inner rates are linear on this set.
Let $e=[v,v']$ be an edge of $\Gamma_\pi$ corresponding to an intersection point $p$ of two components $E_v$ and $E_{v'}$ of $\pi^{-1}(0)$.
Since any divisorial point of $e$ is associated with a divisor appearing after a finite composition of point blowups centered over $p$, it is sufficient to prove that $\calI_X$ is linear on the set $\{v,v'',v'\}$, where $v''$ is the divisorial valuation associated with the exceptional divisor $E_{v''}$ of the blowup of $X_\pi$ at $p$.
Therefore, all we have to show is that 
\begin{equation*}\label{eq_proof_extension_innerrates}
\frac{\calI_X(v'')-\calI_X(v)}{\length([v,v''])} = \frac{\calI_X(v')-\calI_X(v)}{\length([v,v'])},
\end{equation*}
which follows from the definition of the lengths and from Lemma~\ref{lem:inner rate }.\ref{lem:inner rate  SINGULAR PT}.
\end{proof} 

Let $\pi \colon X_{\pi} \to X$ be a
good resolution of $(X,0)$ 
that factors through the blowup of the maximal ideal and through the Nash transform of $(X,0)$
and let $\ell \colon (X,0) \to (\C^2,0)$ be projection that is generic with respect to $\pi$.
Denote by $p_1,\ldots,p_k$ the points where the strict transform of the polar curve of $\ell$ intersects the exceptional divisor $\pi^{-1}(0)$ of $\pi$.
For each $j=1,\ldots,k$, let $E_{v_j}$ be the exceptional curve created by the blowup of $p_j$.
As an immediate consequence of the description of $\widetilde\ell$ given in Section \ref{subsec:projection} and of the computation of the inner rates of Lemma~\ref{lemma_definition_inner_rate}, we deduce that for every $v \in \NL(X,0)$ we have $\calI_X (v)=\calI_{\C^2}\big( \widetilde{\ell}(v)\big)$ if $v$ does not belong to one of the $k$ connected components of $\NL(X,0)\setminus \Gamma_\pi$ containing one of the $v_j$, and $\calI_X (v) < \calI_{\C^2}\big( \widetilde{\ell}(v)\big)$ otherwise.

In particular, this implies that $\cal I_X$ takes the value 1 precisely on the set of $\cal L$ nodes of $\NL(X,0)$, and is greater than one elsewhere.
Indeed, the set of $\cal L$ nodes of $\NL(X,0)$ coincides with $\widetilde\ell^{-1}(\ord_0)$, where $\ord_0$ is the divisorial valuation associated with the blowup of $\C^2$ at $0$ (that is, $\ord_0$ is the unique $\cal L$-node of $\NL(\C^2,0)$).
We have $\calI_{\C^2}(\ord_0)=1$ immediately from the definition and the fact that the inner rate function $\calI_{\C^2}$ on $\NL(\C^2,0)$ is greater than 1 elsewhere follows for example from Lemma~\ref{lem:inner rate }.
Moreover, the fact that $\calI_X$ can be computed on $\C^2$ and lifted via the finite cover $\widetilde\ell$ imposes strong conditions on the way it grows along paths in $\NL(X,0)$.
The following result, which illustrates this phenomenon, is not needed in the rest of this paper but can be helpful to the reader in order to acquire a better intuition for the behavior of inner rates.
It can be compared to the main result of \cite{MaugendreMichel2017}.

\begin{proposition} \label{proposition:growth_behavior}
Let $v$ be a point of $\NL(X,0)$.
Then there exist an $\cal L $-node $w$ of $\NL(X,0)$ and a path from $w$ to $v$ in $\NL(X,0)$ along which $\cal I_X$ is strictly increasing.
\end{proposition}

\begin{proof} It suffices to prove this when $v$ is a divisorial point of $\NL(X,0)$. 
Let $\pi \colon X_{\pi} \to X$ be a good resolution of $(X,0)$ that factors through the blowup of the maximal ideal and through the Nash transform of $(X,0)$ and such that $v$ is the divisorial point associated with an irreducible component $E_v$ of $\pi^{-1}(0)$.  
Consider a generic projection $\ell \colon (X,0) \to (\C^2)$ with respect to $\pi$.
Let $\sigma \colon Y \to \C^2$ be a sequence of blowups of points which resolves the discriminant curve $\Delta$ of $\ell$ and such that $\tilde{\ell}(w)$ corresponds to an exceptional curve of $\sigma^{-1}(0)$ for each vertex $w$ of $\Gamma_{\pi}$.
Let us perform the usual Hirzebruch--Jung construction for $\sigma$: taking the pull-back of $\sigma$ by $\ell$ and its normalization, we obtain morphisms $\hat{\ell} \colon Z \to Y$  and $\rho \colon Z \to X$ such that $\ell\circ\rho=\sigma\circ\hat{\ell}$.
Let $\pi' \colon X_{\pi'} \to X$ be the good resolution of $(X,0)$ obtained by composing $\rho$ with the minimal good resolution of $Z$. 
Since the singularities of $Z$ are quasi-ordinary singularities and the branching locus of the projection $\hat{\ell}$  is contained in the normal crossing divisor $\sigma^{-1}(\Delta)$, it follows that the inverse image by $\tilde{\ell}$ of an edge of $\Gamma_{\sigma}$ joining two vertices $\nu$ and $\nu'$ is a union of segments in  $\NL(X,0)$ each joining a vertex of $\tilde{\ell}^{-1}(\nu)$ to a vertex of $\tilde{\ell}^{-1}(\nu')$. 
This implies that the unique injective path $\widetilde{\gamma}$ from ${\ord_0}$ to $\tilde\ell({v})$ in $NL(\C^2,0)$ lifts via the ramified cover $\widetilde\ell\colon \NL(X,0) \to \NL(\C^2,0)$ to a union of injective paths joining the vertices of $\tilde\ell^{-1}\big(\tilde\ell({v})\big)$ to $\cal L$-nodes. 
We can choose one such injective lifting $\gamma$ of $\tilde{\gamma}$ from a vertex $w$ to $v$, and as $\widetilde{\ell}(w)=\ord_0$ it follows that $w$ is a $\cal L$-node of $\NL(X,0)$.
Since $\ell$ is generic with respect to $\pi$, we deduce that $\cal I_X$ and $\cal I_{\C^2} \circ \tilde{\ell}$ coincide on $\gamma$, and since $\cal I_{\C^2}$ is strictly increasing along $\tilde{\gamma}$ from ${\ord_0}$ to $\tilde{\ell}(v)$, it follows that $\cal I_X$ is strictly increasing along $\gamma$ from $w$ to $v$.
\end{proof}

\begin{remark}\label{remark:ultrametric}
With any continuous map $F\colon\NL(X,0)\to\R_+\cup\{+\infty\}$ one can naturally associate an ultrametric (that is, non-archimedean) distance on $\NL(X,0)$ via a standard minimax procedure: the distance between two points $v$ and $v'$ is set to be $e^{-F_{v,v'}}$, where $F_{v,v'}$ is the maximum, taken over all injective paths $\gamma$ in $\NL(X,0)$ between $v$ and $v'$, of the minimum of $F$ on $\gamma$.
The observation contained in Remark~\ref{remark:maxmin} allows us to give a natural geometric interpretation to the ultrametric distance associated with the inner rate function on $\NL(X,0)$.
\end{remark}

\begin{remark}\label{remark:discrepancies}
In the smooth case, by computing it thanks to Lemma~\ref{lem:inner rate }, one can show that the inner rate function $\calI_{\C^2}$ equals the normalized log discrepancy function on $\NL(\C^2,0)$, as studied for example in \cite{FavreJonsson2004} and in \cite{BarrosoPerezPopescu-Pampu2018}, minus one.
In the singular case, we expect that the inner rate of a divisorial valuation $v$ of $(X,0)$ can be approached using a tool commonly used in birational geometry, namely logarithmic Fitting ideals, and that it could thus be related to the Mather log discrepancy (or to the Jacobian discrepancy) of $v$.
We think that this birational point of view to the study of inner rates might lead to a deeper understanding of metric germs, including the case of dimension three and higher.
\end{remark}


\section{The Laplacian of the inner rate}
\label{section_main_theorem}

In this section we prove our main result, Theorem~\ref{theorem_main}, which computes the Laplacian of the restriction of the inner rate function to the dual graph of any good resolution of $(X,0)$ that factors through the blowup of the maximal ideal.  In order to state it precisely we need to collect a few more definitions.

Let $\pi\colon X_\pi\to X$ be a good resolution and consider the map $g\colon \Gamma_\pi \to\Z_{\geq 0}$ sending a vertex of $\Gamma_\pi$ to the  genus of the associated irreducible component of the exceptional divisor and everything else to zero.
We then define the \emph{canonical divisor} $K_{\Gamma_\pi}$ of the graph $\Gamma_\pi$ as the divisor $K_{\Gamma_\pi} = \sum_{v\in\Gamma_\pi}m_v\big(\val_{\Gamma_\pi}(v)+2g(v)-2\big)[v]$ of $\Gamma_\pi$, where $\val_{\Gamma_\pi}(v)$ denotes the valency of $v$ in $\Gamma_\pi$ (that is the number of edges of $\Gamma_\pi$ adjacent to $e$). 
Observe that this is indeed an element of $\Div(\Gamma_\pi)$ because for every point $v$ of the metric graph $\Gamma_\pi$ that is not a vertex we have $\val_{\Gamma_\pi}(v)=2$ and $g(v)=0$.

In particular, the canonical divisor $K_{\Gamma_\pi}$ will account for the fact that, since the inner rate grows linearly after blowing up a smooth point on an exceptional component of a good resolution, the Laplacian of $\cal I_X |_{\Gamma_\pi}$ at a point $v$ does depend on the choice of a resolution $\pi$ such that $v\in\Gamma_\pi$, and more precisely on the valency $\val_{\Gamma_\pi}(v)$.
For this reason, the Laplacians $\Delta_{\Gamma_\pi}(\cal I_X |_{\Gamma_\pi})$ will not define a divisor on $\NL(X,0)$ by a limit procedure.

\begin{remark}
	Thanks to the adjunction formula, the divisor $K_{\Gamma_\pi}$ is more closely related to the log-canonical divisor on $X_\pi$ than to the canonical divisor, since intersecting an exceptional component $E_v$ with the former yields $\val_{\Gamma_\pi}(v) + 2g(v)-2$, while intersecting it with the latter yields $E_v^2 + 2g(v)-2$.
	However, the terminology canonical divisor for $K_{\Gamma_\pi}$ seems to be quite ubiquitous in the literature (see for example the already cited \cite{BakerNicaise2016, corry2018divisors}), in part due to the fact that the graphs considered there are not always weighted by the self-intersections, and thus we decided to maintain it.
\end{remark}

Finally, we need to introduce two divisors on $\NL(X,0)$.
Define $L=\sum m_vl_v[v]$, where $v$ ranges over the set of divisorial valuations of $(X,0)$ associated with the irreducible components of the exceptional  divisor the the blowup of the maximal ideal of $(X,0)$, and $l_v$ is the number of  irreducible components of a generic hyperplane section of $(X,0)$ whose strict transforms by the blowup intersect the divisor $E_v$ associated with $v$.
Similarly, set $P=\sum m_v p_v[v]$, for $v$ ranging over the set of divisorial valuations of $(X,0)$ associated with the Nash transform of $(X,0)$, where $p_v$ is the number of components of the  polar curve of a generic projection of $(X,0)$ whose strict transforms by the Nash transform intersect $E_v$.
Observe that by definition the divisor $L$ is supported on the set of $\cal L$-nodes of $(X,0)$, while $P$ is supported on the set of its $\cal P$-nodes.

\begin{example}\label{example:E8_LPK}
Consider again the singularity $E_8$, whose resolution data has been described in Examples~\ref{example:E8_graph}, \ref{example:E8_metric_graph}, \ref{example:E8_Laplacien}, and \ref{example:E8_inner_rate}. 
Then the canonical divisor of $\Gamma_\pi$ is $K_{\Gamma_{\pi}} = -2[v_0] + 6[v_4] - 2 [v_6] - 3 [v_7]$.
Moreover, as can be seen from the discussion of Example~\ref{example:E8_metric_graph}, we have $L=2[v_0]$.
Observe that the divisor $P$ is not supported on $\Gamma_\pi$, since the resolution $\pi$, while a resolution of the polar curve of a generic projection, does not factor through the Nash transform. In fact, 
one can verify that $\pi$ factors through the Nash transform after blowing up a suitable smooth point of $E_{v_7}$ and then a suitable smooth point of the resulting exceptional divisor  (see \cite[Example 3.5]{BirbrairNeumannPichon2014} for details). 
In particular, we deduce that $(r_\pi)_*P=3[v_7]$.
\end{example}

\subsection{Statement of the main theorem}
\label{subsection_statement}
We have now collected all the ingredients needed to state our main theorem in full generality.
Since no risk of confusion will arise, we will make a small abuse of notation and simply write $\Delta_{\Gamma_\pi}\big(\calI_X\big)$ for the Laplacian $\Delta_{\Gamma_\pi}\big(\calI_X|_{\Gamma_\pi}\big)$ of the restriction $\calI_X|_{\Gamma_\pi}$ of the inner rate function on a dual graph $\Gamma_\pi$.

\begin{theorem}[Laplacian of the inner rate function]
	\label{theorem_main} 
	Let $(X,0)$ be the germ of a complex surface with an isolated singularity. 
	Let $\pi\colon X_\pi\to X$ be a good resolution that factors through the blowup of the maximal ideal of $(X,0)$. 
	Then the following equality
	\begin{equation*}\label{eq_laplacian}
	\Delta_{\Gamma_\pi}\big(\calI_X\big)=K_{\Gamma_\pi}+2L-(r_\pi)_*P
	\end{equation*}
	holds in $\Div(\Gamma_\pi)$.
	
\end{theorem}

\begin{remark}
	\label{remark:theorem_main_implies_B}
	In particular, if $\pi$ is a good resolution of $(X,0)$ which factors through the blowup of the maximal ideal of $X$ at $0$ and through the Nash transform of $(X,0)$, then we obtain the statement of Theorem~\ref{thm:B} as given in the introduction.
	Namely, the coefficient of the Laplacian of the inner rate function on the dual graph $\Gamma_\pi$ at a vertex $v$ is precisely $m_v\big(2l_v-p_v-\chi(\check E_{v})\big)$, where $l_v$ and $p_v$ are defined as above, since $K_{\Gamma_\pi}(v)=-m_v\chi(\check E_{v})$.
\end{remark}

\begin{example}\label{example:E8_Theorem}
	The formula of Theorem \ref{theorem_main} can be readily verified in the case of the singularity $E_8$ by combining Example~\ref{example:E8_Laplacien} and Example~\ref{example:E8_LPK}. For example, on the vertex $v_4$, we indeed get $\Delta_{\Gamma_\pi}(\calI_X)(v_4) = 6 = K_{\Gamma_\pi}(v_4) = K_{\Gamma_\pi}(v_4)+2L(v_4)-P(v_4)$.
\end{example}

We introduce another simple combinatorial definition.
Let $\Gamma_\pi$ be the dual graph of a good resolution of $(X,0)$ that factors through the blowup of the maximal ideal and through the Nash transform of $(X,0)$.
We call \emph{string} of $\Gamma_\pi$ a segment $\cal S$ in $\Gamma_\pi$ that starts and ends at  vertices $v$ and $v'$ of $\Gamma_\pi$ respectively and such that $\cal S\setminus\{v,v'\}$ contains no $\cal L$-nodes, no $\cal P$-nodes, and no points whose genus is strictly positive or whose valency in ${\Gamma_\pi}$ is at least $3$.
Observe that every edge of $\Gamma_\pi$ is a string.
The following result, while a straightforward consequence of Theorem~\ref{theorem_main}, is of independent interest.

\begin{corollary} \label{cor:linearity on strings_general}  Let $(X,0)$ be the germ of a complex surface with an isolated singularity. 
	Let $\pi\colon X_\pi\to X$ be a good resolution that factors through the blowup of the maximal ideal and through the Nash transform of $(X,0)$.
	Then, if $\cal S$ is a string of $\Gamma_\pi$, the inner rate function $\calI_{X}$ is linear on $\cal S$.
\end{corollary}

The remaining part of the section is devoted to the proof of Theorem~\ref{theorem_main}.
An outline of the proof's method can be found in the introduction.


\subsection{Reduction to dominant resolution and the smooth case}
\label{subsection_laplacian_C2}

We start by proving the formula of our main theorem in the case of a smooth surface germ.
This is the content of Proposition~\ref{proposition_laplacian_smooth}.

In order to do so, we will first establish a simple combinatorial lemma that will be used several times throughout the proof of Theorem~\ref{theorem_main}.

\begin{lemma}\label{lemma_reduction_bigger_modification}
	Let $\pi\colon X_\pi\to X$ and $\pi'\colon X_{\pi'}\to X$ be two good resolutions of $(X,0)$ and assume that $\pi'$ factors through $\pi$ which in turn factors through the blowup of the maximal ideal of $(X,0)$.
	If
	\begin{equation*}
	\Delta_{\Gamma_{\pi'}}(\calI_X)=K_{\Gamma_{\pi'}}+2L-(r_{\pi'})_*P
	\end{equation*}
	holds, then
	\begin{equation*}
	\Delta_{\Gamma_\pi}(\calI_X)=K_{\Gamma_\pi}+2L-(r_{\pi})_*P
	\end{equation*}
	holds as well.
	Moreover, the converse implication is also true if $\pi$ factors through the the Nash transform of $(X,0)$.
\end{lemma}

\begin{proof}
	To unburden the notation, we will write $\Gamma$ and $\Gamma'$ for $\Gamma_\pi$ and $\Gamma_{\pi'}$ respectively.
	Since $\pi'$ dominates $\pi$, the metric graph $\Gamma'$ contains $\Gamma$.
	Denote by $\tilde r \colon \Gamma'\to\Gamma$ the restriction of $r_{\pi}$ to $\Gamma'$, so that $r_{\pi}=\tilde r\circ r_{\pi'}$.
	By applying $\tilde r_*$ to the first equation of the statement we obtain
	\[
	\tilde r_*\Delta_{\Gamma'}(\calI_X) = 
	\tilde r_*K_{\Gamma'} + 
	\tilde r_*2L-\tilde r_*\circ(r_{\pi'})_*P.
	\]
	We will first prove that the left-hand side is equal to $\Delta_{\Gamma}(\calI_X)$ and that $r_*(K_{\Gamma'})=K_{\Gamma}$.	
	As $\pi'$ is obtained by composing the resolution $\pi$ with a finite sequence of point blowups, by induction we can assume without loss of generality that $\pi'$ is obtained by blowing up a single point $q$ of $\pi^{-1}(0)$.
	If $q$ lies on the intersection of two components of $\pi^{-1}(0)$ then topologically $\Gamma'=\Gamma$, the map $\tilde r$ is the identity, and $K_{\pi'}=K_{\pi}$, so in this case there is nothing to prove.
	We can therefore assume that $\Gamma'$ is obtained from $\Gamma$ by adding a single vertex $v'$ (corresponding to the exceptional component of the blowup of $q$) and precisely one edge $e$ connecting $v'$ to a vertex $v$ of $\Gamma$.
	Then we have 
	\[
	\Delta_{\Gamma'}(\calI_X)(v) = \Delta_{\Gamma}(\calI_X)(v) + s,
	\]
	where $s$ is the slope of $\calI_X$ on the edge $e$ going from $v$ to $v'$.  
	Since the Laplacian of the restriction of $\calI_X$ to the edge $e = \tilde r^{-1}(v)$ has degree $0$, we have 
	\[
	s=-\sum_{v''\in \tilde r^{-1}(v)\setminus\{v\}}\Delta_{\Gamma'}(\calI_X)(v'').
	\]
	This implies that 
	\[
	\Delta_{\Gamma}(\calI_X)(v) = \Delta_{\Gamma'}(\calI_X)(v) +  \Delta_{\Gamma'}(\calI_X)(v')
	= \sum_{v''\in \tilde r^{-1}(v)}\Delta_{\Gamma'}(\calI_X)(v''),
	\]
	which is what we wanted to prove.
	Moreover, observe that $\tilde r_*K_{\Gamma'}(v) = K_{\Gamma'}(v) + K_{\Gamma'}(v') = K_{\Gamma'}(v) - m_{v'} = K_{\Gamma'}(v) - m_{v}$, which is precisely $K_{\Gamma}(v)$ since the valency of $\Gamma$ in $v$ is equal to the valency of $v$ in $\Gamma'$ minus one.
	This shows that $\tilde r_*$ respects the canonical divisors.
	Since $\tilde r_*\circ(r_\pi)_* = (\tilde r \circ r_{\pi'})_* = (r_{\pi})_*$, this proves the first part of the lemma.
	To establish the second part of the lemma, observe that, under the hypothesis that $\pi$ factors through the blowup of the maximal ideal and through the Nash transform of $(X,0)$, both $L$ and $P$ are divisors on $\Gamma_\pi$ and therefore they are stable under $(r_{\pi})_*$ and $(r_{\pi'})_*$, and moreover the slope $s$ on the edge $[v,v']$ is equal to $m_v$ (and therefore to $m_{v'}$) thanks to Lemma~\ref{lem:inner rate }.\ref{lem:inner rate  SMOOTH PT}.
	This means that
	\begin{align*}
	& \Delta_{\Gamma'}(\calI_X)(v) - \Delta_{\Gamma}(\calI_X)(v) = m_v = K_{\pi'}(v)-K_\pi(v) \\ 
	= \,& \big(K_{\pi'} + 2L-(r_{\pi'})_*P\big)(v) - \big(K_\pi + 2L-(r_{\pi})_*P\big)(v),
	\end{align*}
	which concludes the proof.
\end{proof}

In the smooth case the non-archimedean link $\NL(\C^2,0)$ is a tree, it has divisors $L=[\ord_0]$, the only $\cal L$-node being the divisorial valuation $\ord_0$ associated with the blowup of $\C^2$ at $0$, and $P=0$, as a generic projection to $\C^2$ is unramified and thus there are no  $\cal P$-nodes.
The Laplacian of the inner rate function in this very special case is simple to compute thanks to the previous lemma.
 
\begin{proposition}\label{proposition_laplacian_smooth}
	Let $\pi\colon X_\pi\to \C^2$ be a sequence of point blowups of $(\mathbb{C}^2,0)$ that is not the identity.
	Then the Laplacian of the inner rate function $\calI_{\C^2}$ on the dual graph $\Gamma_\pi$ associated with $\pi$ is
	\[
	\Delta_{\Gamma_\pi}(\calI_{\C^2})=K_{\Gamma_\pi}+2[\ord_o].
	\]
\end{proposition}

\begin{proof}
We will argue by induction on the number of blowups in $\pi$.
If $\pi$ is obtained by blowing up the origin of $\C^2$ once, then the associated dual graph consists of the divisorial valuation $\mathrm{ord}_0$ and the formula is immediate, as $\Delta_{\{\mathrm{ord}_0\}}(\cal I_{\C^2})=0$ and $K_{\{\mathrm{ord}_0\}}(\mathrm{ord}_0)=-2$.
The inductive step is a direct consequence of the second part of Lemma~\ref{lemma_reduction_bigger_modification}.
\end{proof}

We state as a separate result the smooth case of Corollary~\ref{cor:linearity on strings_general}, which we obtain as an immediate consequence of Proposition~\ref{proposition_laplacian_smooth}, since we will need to refer to it in the course of our proof of Theorem~\ref{theorem_main}. 

\begin{corollary} \label{cor:linearity on strings}
Let $\pi\colon X_\pi\to\C^2$ be a sequence of point blowups starting with the blowup of the origin of  $\C^2$. 
Then, if $\cal S$ is a string of $\Gamma_\pi$, the inner rate function $\calI_{\C^2}$ is linear on $\cal S$.
\end{corollary}



\subsection{Dehn twists and screw numbers}
\label{subsection_monodromy_screw_numbers}

 Set $S^1=\{z \in \C \colon |z|=1\}$.  Let $f \colon A \to A$   be an orientation preserving diffeomorphism of  a disjoint union $A=\coprod_{i=1}^r A_i$ of annuli  $A_i \cong S^1\times[0,1] $ which cyclically exchanges the annuli  $A_i$ and which  is periodic on the union   $\partial A=\coprod_{i=1}^r \partial A_i$ of their boundaries. 
    Let $N$ be an integer such that $f^N$ is the identity on  $\partial A$. 
    Then, up to isotopy fixed on   $\partial A$, the map $f$ is characterized by a rational number $\tau$ defined as follows.
	Let us choose an annulus $A_{i_0}$ among the $A_i$ and fix an isomorphism $A_{i_0}\cong S^1\times[0,1]$.
	Observe that the restriction of $f^N$ to $A_{i_0}$ is, up to isotopy fixed on $\partial A_{i_0}$, a product of Dehn twists.  
	Consider the transverse oriented path $\delta = \{1\} \times [0,1]$ inside $S^1\times[0,1] $, 
	and let us orient the circle  $c = S^1 \times \{1\}$ in such a way that the intersection number $\delta.c$ of $\delta$ with $c$ on the oriented surface $S^1\times[0,1]$ is equal to $+1$. 
 
  \begin{definition} The {\it screw number} $\tau$  of $f$ is  the rational number defined by the following equality in $H_1(S^1\times[0,1], \Z)$: 
   \[
   N\tau  c  = f^N(\delta) - \delta.
   \] 
Observe that $\tau$ does not depend on the choice of the integer $N$ such that $f^N$ is the identity on $\partial A$, nor on the choice of the annulus $A_{i_0}$.
\end{definition}

In the sequel, we will use the two following simple lemmas. 
\begin{lemma}   \label{lem:twist1}  Set $A = S^1\times[0,1]$ and decompose  the annulus $A$ as the union of the two concentric annuli $A_1= S^1\times[0,\frac{1}{2}]$ and $A_2= S^1\times[\frac{1}{2}, 1]$.
Let $f \colon A \to A$ be an orientation preserving diffeomorphism such that $f(A_1) = A_1$,  $f(A_2) = A_2$, and there exists $N>0$ such that $f^N$ is the identity on $S^1 \times \{ 0,\frac12,1\}$.
Then, if $f|_{A_1}$ and $f|_{A_2}$ have screw numbers $\tau_1$ and $\tau_2$ respectively, $f$ has screw number $\tau_1+\tau_2$. 
\end{lemma}

\begin{proof} 
Set $\delta_1 = \{1\} \times  [0,\frac{1}{2}]$, $\delta_2 = \{1\} \times  [\frac{1}{2},1]$ and $\delta = \delta_1 \cup \delta_2 = \{1\} \times  [0,1] $. 
Set also $c_1 = S^1 \times \{\frac{1}{2}\}$ and  $c_2 = S^1 \times \{1\}$, oriented so that $\delta.c_i  = +1$ in $A$, so $c_1$ and $c_2$ have the same homology class on $A$. 
Then $ N \tau_1  c_1  = f^N(\delta_1) - \delta_1$ and $N  \tau_2  c_2  = f^N(\delta_2) - \delta_2$ hold as equalities of cycles in $A_1$ and $A_2$ respectively. 
Since $c_1 $ is homologous to $ c_2$ in $A$ and $\delta = \delta_1 \cup \delta_2$, we deduce that $ N(\tau_1 + \tau_2)  c_1 = f^N(\delta) - \delta$ in $H_1(A , \Z)$.
This proves that the screw number of $f$ equals $\tau_1+\tau_2$. 
\end{proof}

\begin{lemma}  \label{lem:twist2} Let $A =  S^1  \times  [0,1]$ be an annulus and let $\ell \colon A \to A$ be a cyclic cover of degree $\deg(\ell)$. 
Let $\phi \colon A \to A$  and $\phi' \colon A \to A$ be two orientation preserving diffeomorphisms such that $\ell \circ \phi = \phi' \circ \ell$ and such that both $\phi$ and $\phi'$ have a power which is the identity on the boundary of $A$, and let $\tau$ and $\tau'$ denote the screw numbers of $\phi$ and $\phi'$ respectively.
Then $\tau' = \deg(\ell) \tau$.
\end{lemma}   

\begin{proof}   Let $N$ be a positive integer such that both $\phi^N$ and  $(\phi')^N$ are the identity on $\partial A$. 
Then, using the same notations as before, the screw number $\tau$ of $\phi$ is defined by $\phi^N \delta - \delta = N\tau c$. Therefore  $\ell (\phi^N \delta) - \ell(\delta) = N\tau \ell(c)$. 
Observe that we have $\ell (\delta) - \delta=0$ and $\ell(c) - \deg (\ell) c =0$
in $H_1(A,\Z)$. 
Since $ \ell \circ \phi^N = ({\phi'})^N \circ \ell$, this implies that $({\phi'})^N \delta - \delta = N\tau \deg (\ell) c$, proving that the screw number $\tau'$ of $\phi'$ is equal to $\deg(\ell)\tau$. 
 \end{proof}

\subsection{Skeletal metric and the monodromy of the Milnor fibration} \label{subsection_monodromy_lengths}

Let $(X,0)$ be the germ of a complex surface with an isolated singularity, let $\pi \colon X_{\pi} \to X$ be a good resolution of $(X,0)$ which factors through the blowup of the maximal ideal, and let $h\colon (X,0) \to (\C,0)$ be a generic linear form on $(X,0)$ (that is, $h^{-1}(0)$ is a generic hyperplane section of $(X,0)$).
In this section we describe the lengths of the edges of $\Gamma_\pi$ as screw numbers of the monodromy of some pieces of the Milnor fibration of $h$.

Choose an embedding of $(X,0)$ into a smooth germ $(\C^n,0)$ and, given $\epsilon,\eta>0$, consider balls $B_{\epsilon} =\{z \in \C^n; \|z\| \leq \epsilon\}$ in $\C^n$ and $D_{\eta} = \{z \in \C \,|\, |z| \leq \eta\}$ and $D_{\eta}^*  = D_{\eta}  \setminus \{0\}$ in $\C$.
If $\eta<<\epsilon$ are small enough, then $B_\epsilon$ is a Milnor ball for $X$, that is $X$ intersects transversally the boundary of $B_{\epsilon'}$ for all $\epsilon'<\epsilon$, and the restriction $h|_{T_{\epsilon, \eta}} \colon T_{\epsilon, \eta} \to D_{\eta}^*$ of $h$ to ${T_{\epsilon, \eta}} = B_{\epsilon}    \cap h^{-1}(D_{\eta}^*)$ is a topologically trivial fibration called the  \emph{Milnor--L\^e fibration} of $h$.
We also call ${T_{\epsilon, \eta}}$ a \emph{Milnor tube} for $L$, and given $t$ in $D_{\eta}^*$ denote by $F_t = (h |_{T_{\epsilon,\eta}} )^{-1}(t)$ a fiber of this fibration. 

The monodromy of $h{|_{T_{\epsilon,\eta}}} $ admits a quasi-periodic representative $\phi$, which means that there exist a disjoint union of annuli $A$ embedded in the fiber of $h{|_{T_{\epsilon,\eta}}}$ and an integer $N\geq 1$ such that $\phi(A) = A$ and $\Phi^N$ is the identity map outside of the interior of $A$.  
Such a representative $\Phi$ can be described using a suitable resolution of $(X,0)$ in the following way.
Let $\pi \colon X_{\pi} \to X$ be a good resolution of $(X,0)$ which factors through the blowup of the maximal ideal. 
In what follows, we identify the fibers $F_t$ with their inverse image through $\pi$ (recall that $\pi$ is a diffeomorphism outside $0$) without changing the notation.

For each component $E_v$ of $E = \pi^{-1}(0)$, let $N(E_v)$ be a tubular neighborhood of $E_v$ in $X_{\pi}$, which is the total space of a normal disc-bundle to $E_v$ in $X_\pi$.  
Observe that $\pi^{-1}(B_{\epsilon})$ can be identified with a tubular neighborhood of $E$ obtained by plumbing the disc-bundles $N(E_v)$.  
For each vertex $v$ of $\Gamma_\pi$, we set
\[
\cal N (E_v) =  \overline{N(E_v)\setminus    \bigcup_{v' \neq v} N(E_{v'}) }.
\]
For each irreducible component $\gamma$ of $h^{-1}(0)$ in $(X,0)$, denote by $\gamma^*$ the strict transform of $\gamma$ in $X_\pi$, let $A_{\gamma^*}$ be a small disc in $E_v \setminus \bigcup_{v' \neq v}  E_{v'}$ centered at $\gamma^* \cap E_v$, and let $N(\gamma^*)$ be the restriction of $N(E_v)$ over $A_{\gamma^*}$, so that $N(\gamma^*)$ is a small polydisc neighborhood of $\gamma^*$ in $X_{\pi}$.
Set 
\[
\check{\cal N}(E_v) = \overline{\cal N(E_v)\setminus     \bigcup_{\gamma \subset \{h=0\}} N(\gamma^*) } \quad \text{ and } \quad F_{v,t} = F_t \cap  \check{\cal N}(E_v).
\]
We will often simply denote $F_{v,t}$ by $F_v$; since all fibers are diffeomorphic no risk of confusion will arise.

The intersection  $\partial T_{\epsilon,\eta} \cap \pi\big(\check{\cal N}(E_v)\big)$, where $\partial T_{\epsilon,\eta}=h^{-1}(\partial D_\eta)\cap B_\epsilon$, is fibered  by the oriented circles obtained by intersecting with $\partial T_{\epsilon,\eta}$ the image via $\pi$ of the disc-fibers of $N(E_v)$.
Then the monodromy $\phi \colon F_t \to F_t$ is defined on  each $F_{v,t}$ as the diffeomorphism of first return of these circles, so it is a periodic diffeomorphism. 

Now consider an edge $e=[v,v']$ in $\Gamma_{\pi}$ and let $q \in E_v \cap E_{v'}$ be the corresponding intersection point. 
Let $\cal N(e)$ be the component of $N(E_v) \cap N(E_{v'})$ containing $q$. 
The intersection $F_{e,t} = F_t \cap \cal N(e)$ (or simply $F_e$) is the disjoint union of $\gcd(m_v,m_{v'})$ annuli which are  cyclically exchanged by $\phi$.
Its monodromy is related to the length of the edge $e$ by the following proposition.

\begin{proposition}(\cite[Theorem 7.3.(iv)]{MatsumotoMontesinos2011})  \label{rk:dehn twist}  Denote by $\phi_e$ the restriction $\phi{|}_{F_e} \colon F_{e} \to F_{e}$ of the monodromy $\phi$ to $F_e$. 
Then $\phi_e^{m_v m_{v'}} = id_{F_e}$ and $\phi_e$ has screw number $ -\frac{1}{m_vm_{v'}}$.  
In other words, this screw number is the opposite of $\mathrm{length} (e)$.
\end{proposition}


\subsection{The resolution adapted to a projection} \label{subsec:adapted resolution}
    
Let $\ell \colon (X,0) \to (\C^2,0)$ be a generic projection as in Section \ref{subsec:projection} and let $\pi \colon X_{\pi} \to X$ be a good resolution of $(X,0)$. 
Consider the minimal sequence $\sigma_{\pi} \colon Y_{\sigma_\pi} \to \C^2$ of blowups of points starting with the blowup of the origin of $\C^2$  such that $\widetilde{\ell}\big(V(\Gamma_{\pi})\big) \subset  V(\Gamma_{\sigma})$.
Now let $\tilde{\pi} \colon X_{\tilde\pi} \to X$ be the good resolution of $(X,0)$ obtained by pulling back $\sigma_{\pi}$, through $\ell$, normalizing the resulting surface, and then resolving the remaining singularities. 
Denote by $\ell_{\pi} \colon  X_{\tilde\pi} \to Y_{\sigma_\pi}$ the  projection morphism, so that $\sigma_{\pi} \circ \ell_{\pi} =  \ell \circ \tilde{\pi}$.
The resolution  $\tilde{\pi}$ factors through $\pi$, and if we denote by $\alpha_{\pi} \colon  X_{\tilde\pi} \to X_{\pi}$ the resulting morphism we obtain the following commutative diagram: 
\[
\xymatrix{  X_{\tilde\pi}   \ar@/^1.2pc/[rr]^{\tilde{\pi}}   \ar[r]^{\alpha_{\pi}}  \ar[rd]^{{\ell}_{\pi}} & X_{\pi} \ar[r]^{\pi} & X \ar[d]^{{\ell}}   \\
       & Y_{\sigma_\pi} \ar[r]^{\sigma_{\pi}}  & \C^2}
\]
By construction, the following properties are satisfied: 
\begin{enumerate}
\item For every vertex $v$ of $\Gamma_{\pi}$, we have $\tilde \ell(v) \in V(\Gamma_{\sigma_\pi})$;
\item For every vertex $\nu$ of $\Gamma_{\sigma_\pi}$, we have $\tilde{\ell}^{-1}(\nu) \subset V(\Gamma_{\tilde{\pi}})$.
\end{enumerate}

\begin{definition} \label{def:adapted}
We call the map $\tilde{\pi} \colon X_{\tilde\pi} \to X$ defined as above the resolution of $(X,0)$ \emph{adapted to $\pi$ and $\ell$}.
\end{definition}

As part of the data of the adapted resolution we also keep the morphisms $\sigma_\pi\colon Y_{\sigma_\pi} \to \C^2$ and $\ell_\pi\colon X_{\tilde{\pi}}\to Y_{\sigma_\pi}$.

Observe that if the good resolution $\pi$ factors through the Nash transform of $(X,0)$ and if $\ell$ is generic with respect to $\pi$, then $\ell$ is also generic with respect to the adapted resolution $\tilde{\pi}$.

\begin{example}  \label{ex:A2}
	Consider the hypersurface $(X,0) \subset (\C^3,0)$ defined by the equation $x^2+y^2+z^3=0$. 
	This is the standard singularity $A_2$. 
	The exceptional divisor of its minimal resolution $\pi \colon X_{\pi} \to X$ consists of two $\mathbb P^1$-curves $E_{v_0}$ and $E_{v'_0}$ that intersect transversally at a point; its dual graph $\Gamma_\pi$ is depicted on the right of Figure~\ref{fig:A2}.
	To see this, we consider the generic projection  $\ell=(y,z) \colon (X,0) \to (\C^2,0)$ and we perform Laufer's algorithm from \cite{Laufer1972} (already used in Example \ref{example:E8_graph}.
	Denote by $\sigma \colon Y \to \C^2$ the minimal embedded resolution of the discriminant curve $\Delta$ of $\ell$, which is the cusp $y^2+z^3=0$; its dual graph $\Gamma_\sigma$ is the tree depicted on the left of Figure~\ref{fig:A2}, with vertices labeled $\nu_0, \nu_1, \nu_2$ in their order of appearance in the sequence of blowups.
	Then, performing Laufer's algorithm, we obtain the resolution $\pi' \colon X_{\pi'} \to X$ whose dual graph $\Gamma_\pi$ is at the middle of Figure~\ref{fig:A2}.
	Its vertices are labeled $v_0, v'_0, v_1$ and $v_2$, and we have  $\tilde{\ell}(v_0) = \tilde{\ell}(v'_0) = \nu_0$,  $\tilde{\ell}(v_1)  = \nu_1$ and  $\tilde{\ell}(v_2) = \nu_2$. 
	Denote by $\pi''$ the good resolution of $(X,0)$ obtained by blowing down the exceptional curve $E_{v_1}$.
	Then $\pi''$ is the minimal good resolution of $(X,0)$ that factors through its Nash transform of $(X,0)$.
	If we further blow down the exceptional curve $E_{v_2}$, which has self-intersection -1 in $X_{\pi''}$, we obtain the minimal resolution $\pi$ of $(X,0)$ that we described before. 
	\begin{figure}[h] 
		\centering
		\begin{tikzpicture}  
		
		\draw[thin ](-1,0)--(1,0);
		\draw[fill ] (-1,0)circle(2pt);
		\draw[fill ] (0,0)circle(2pt);
		\draw[fill ] (1,0)circle(2pt);
		
		\draw[thin,>-stealth,->](0,0)--+(0.7,0.7);
		\node(a)at(1,1){ $\Delta^*$};

		\node(a)at(1,-0.3){ $-3$};
		\node(a)at(-1,-0.3){ $-2$};
		\node(a)at(0,-0.3){ $-1$};
		
		\node(a)at(1,-0.6){ $\nu_0$};
		\node(a)at(-1,-0.6){ $\nu_1$};
		\node(a)at(0,-0.6){ $\nu_2$};
		\node(a)at(0,-1.2){$\Gamma_{\sigma}$};
		
		\begin{scope}[xshift=3.5cm]
		
		\draw[thin ](-1,0)--(1,0);
		\draw[thin ](0,0)--(0,1);
		
		\draw[fill ] (-1,0)circle(2pt);
		\draw[fill ] (0,0)circle(2pt);
		\draw[fill ] (1,0)circle(2pt);
		\draw[fill ] (0,1)circle(2pt);
		

		\node(a)at(1,-0.3){ $-3$};
		\node(a)at(-1,-0.3){ $-3$};
		\node(a)at(0,-0.3){ $-2$};
		\node(a)at(-0.4,1){ $-1$};
		
		\node(a)at(1,-0.6){ $v'_0$};
		\node(a)at(-1,-0.6){ $v_0$};
		\node(a)at(0,-0.6){ $v_2$};
		\node(a)at(0.3,1){ $v_1$};
		\node(a)at(0,-1.2){$\Gamma_{\pi'}$};
		
		\end{scope}
		
				\begin{scope}[xshift=7cm]
		
		\draw[thin ](-1,0)--(1,0);
		
		\draw[fill ] (-1,0)circle(2pt);
		\draw[fill ] (0,0)circle(2pt);
		\draw[fill ] (1,0)circle(2pt);
		

		\node(a)at(1,-0.3){ $-3$};
		\node(a)at(-1,-0.3){ $-3$};
		\node(a)at(0,-0.3){ $-1$};
		
		\node(a)at(1,-0.6){ $v'_0$};
		\node(a)at(-1,-0.6){ $v_0$};
		\node(a)at(0,-0.6){ $v_2$};
		\node(a)at(0,-1.2){$\Gamma_{\pi''}$};
		
		\end{scope}
		
		\begin{scope}[xshift=9.5cm]
		
		\draw[thin ](0,0)--(1,0);
		
		\draw[fill ] (0,0)circle(2pt);
		\draw[fill ] (1,0)circle(2pt);
		

		\node(a)at(1,-0.3){ $-2$};
		\node(a)at(0,-0.3){ $-2$};
		
		\node(a)at(1,-0.6){ $v'_0$};
		\node(a)at(0,-0.6){ $v_0$};
		\node(a)at(0.5,-1.2){$\Gamma_{\pi}$};
		
		\end{scope}
		\end{tikzpicture} 
		\caption{ }\label{fig:A2}
	\end{figure}

	\noindent	Observe that, since $\tilde{\ell}(v_0) = \tilde{\ell}(v'_0) = \nu_0$, the resolution $\pi$ is already adapted to itself and to $\ell$, and in this case $\sigma_\pi$ is the blowup of $\C^2$ at the origin, so that $\Gamma_{\sigma_\pi}=\{\nu_0\}$; observe that $\widetilde\ell(\Gamma_\pi)$ is not mapped into $\Gamma_{\sigma_\pi}$.
	However, $\sigma_{\pi''} = \sigma$, hence the resolution adapted to $\pi''$ and $\ell$ is $\pi'$. 
	This time we have $\widetilde\ell(\Gamma_{\pi''})\subset\Gamma_{\sigma_{\pi''}}$, which was to be expected since $\pi''$ factors through the Nash transform of $(X,0)$.
\end{example}

\subsection{The local degree formula} \label{subsec:degree}
 
Let $(X,0)$ be a complex surface germ.
For each divisorial valuation $v$ in $\NL(X,0)$, we are going to define the local degree $\deg(v) \in \N^*$ of $v$, which is an integer that measures the local topological degree at $v$ of a generic projection of $(X,0)$ onto $(\C^2,0)$.

Let $\pi \colon X_{\pi} \to X $ be a good resolution of $(X,0)$ that factors through its Nash transform and such that $v$ is associated with a component $E_v$ of its exceptional divisor $\pi^{-1}(0)$. 
Pick a projection $\ell \colon (X,0) \to (\C^2,0)$ which is generic with respect to $\pi$, let $\tilde{\pi} \colon X_{\tilde{\pi}} \to X$ be the resolution of $(X,0)$ adapted to $\pi$ and $\ell$, and let $ \ell_{\pi}  \colon X_{\tilde{\pi}} \to Y_{\sigma_\pi}$ be the natural morphism (see Definition~\ref{def:adapted}).
Set $\nu= \widetilde{\ell}(v)$ and let us denote by $C_\nu$ the component of $\sigma_{\pi}^{-1}(0)$ associated with $\nu$.
Note that ${\ell}_{\pi}(E_v) =  C_{\nu}$. 
For each component $E_i$ of $\tilde\pi^{-1}(0)$ (respectively $C_j$ of  $\sigma_{\pi}^{-1}(0)$), let us choose a tubular neighborhood disc bundle $N(E_i)$ (resp. $N(C_j)$), and consider the sets $\cal N(E_v)$ and $\cal N({C_\nu})$ introduced in Section~\ref{subsection_monodromy_lengths}.  
We can then adjust the disc bundles $N(E_i)$ and $N(C_j)$ in such a way that the cover $\ell$ restricts to a cover  
\[
\ell_v \colon \tilde\pi\big(\cal N(E_v)\big) \to   \sigma_{\pi}\big(\cal N(C_\nu)\big)
\]
branched precisely on the polar curve of $\ell$ (if $v$ is not a $\cal P$-node, the branching locus is just the origin).
Its degree only depends on $v$ and not on the choice of a generic projection $\ell$.

\begin{definition}  \label{def:degree} 
The \emph{local degree} $\deg(v)$ of $v$ is the degree $\deg(v) = \deg({\ell_v})$ of the cover $\ell_v$.
\end{definition}

We now prove two simple lemmas about the local degrees.
If $\pi\colon X_\pi\to X$ is a good resolution of $(X,0)$ that factors through the blowup of the maximal ideal and through the Nash transform of $(X,0)$, we denote by $V_N(\Gamma_\pi)$ the set of \emph{nodes} of $\Gamma_\pi$, that is the subset of $V(\Gamma_\pi)$ consisting of the $\cal P$-nodes, the $\cal L$-nodes, and of all the vertices of genus strictly grater than zero or whose valency in $\Gamma_\pi$ is at least three.
 
\begin{lemma}  \label{lemma:deg}
Let $\pi \colon X_\pi \to X$ be a good resolution of $(X,0)$ that factors through the blowup of the maximal ideal and through the Nash transform of $(X,0)$.
Then $\deg$ is constant on each connected component of $\Gamma_\pi \setminus V_N(\Gamma_{\pi})$.
\end{lemma}

\begin{proof} 
	Given a string $\cal S$ of $\Gamma_\pi$, we set $\cal N(\cal S) = N(e)$ if $\cal S$ is an edge $e$ of $\Gamma_\pi$, and $\cal N(\cal S) = \bigcup_v N(E_{v})$, where the union is taken over the set of vertices of $\Gamma_\pi$ contained in $\cal S$, otherwise.
	By definition, a connected component $\cal S$ of $  \Gamma_{\pi } \setminus V_N(\Gamma_{\pi})$ is a string of $\Gamma_\pi$.
	In particular, as $\cal S$ contains no $\cal P$-nodes, the polar curve of $\ell$ does not intersect $\cal N(\cal S) \setminus \{0\}$, and so the restriction $ \ell_{\cal S} = \ell|_{\cal N(\cal S)}$ of $\ell_{\cal S}$ to $\cal N(\cal S)$ is a regular cover outside $0$.
	Denote by $d$ the degree of this cover. 
	If $v$ is any divisorial point in the interior of $\cal S$, by further blowing up double points of the exceptional divisor we can assume without loss of generality that $v$ is a vertex of $\Gamma_\pi$.
	Then $\ell_v$ is the restriction of $\ell_{\cal S}$ to $\cal N(E_v)$, and therefore $\deg(v) = d$.  
\end{proof}

\begin{remark}
	Although this is not needed in the paper, it is worth noticing that Lemma~\ref{lemma:deg} allows us to extend the local degree to a map $\deg \colon \NL(X,0)\to\Z$ on the whole non-archimedean link $\NL(X,0)$.
	This map can be characterized as the unique upper semi-continuous map on $\NL(X,0)$ that takes the value $\deg(v)$ on any divisorial valuation $v$.
\end{remark}

As a consequence of Lemma~\ref{lemma:deg}, if $\pi$ is a resolution of $(X,0)$ that factors through the blowup of the maximal ideal  the Nash transform of $(X,0)$ then we can define $\deg(e)$ for any edge $e$ of the dual graph $\Gamma_{\pi}$ by setting $\deg(e) = \deg(v)$, where $v$ is any divisorial point in the interior of $e$. 
The following result gives an alternative way to compute local degrees.

\begin{lemma}\label{lem: degree formula 2} 
Let $\pi\colon X_\pi\to X$ be a good resolution of $(X,0)$  which factors through the blowup of the maximal ideal and through the Nash transform of $(X,0)$, let $\ell\colon (X,0)\to(\C^2,0)$ be a generic projection with respect to $\pi$, let $\tilde{\pi} \colon X_{\tilde\pi} \to X$ be the resolution of $(X,0)$ adapted to $\pi$ and $\ell$, let $v$ be a vertex of $\Gamma_{\pi}$, and let $e$ be an edge in $\Gamma_{\pi}$ adjacent to $v$. 
Denote by $W_{v,e}$ the set of the edges $e'$ of $\Gamma_{\tilde{\pi}}$ that are adjacent to $v$ and such that  $\widetilde \ell(e')  \subset \widetilde \ell(e)$.
Then we have
\[
\deg(v) = \sum_{e' \in W_{v,e}} \deg(e').
\]
\end{lemma}

\begin{proof} 
Consider an embedding $(X,0) \subset (\C^n,0)$ and choose coordinates $(z_1\dots,z_n)$ on $\C^n$ in such a way that that $z_1$ and $z_2$ are generic linear forms and
$\ell=(z_1,z_2)\colon X\to \C^2$ is our generic linear projection.  
We consider  the generic linear form $h = z_1|_X \colon (X,0) \to (\C,0)$ on $(X,0)$.   
As in \cite[Section 4]{BirbrairNeumannPichon2014}, instead of considering a standard Milnor ball defined via the standard round ball in $\C^n$, we will consider a Milnor tube $ T_{R\epsilon, \epsilon}$ of $h$.
More precisely, given some sufficiently small
$\epsilon_0$ and some $R>0$, for every $\epsilon\le\epsilon_0$ we define
\[
B_\epsilon=\big\{(z_1,\dots,z_n)\,\big|\,|z_1|\le \epsilon, |(z_1,\dots,z_n)|\le
R\epsilon\big\}\quad\text{and}\quad S_\epsilon=\partial B_\epsilon\,.
\]
By \cite[Proposition 4.1]{BirbrairNeumannPichon2014} we can choose $\epsilon_0$ and $R$ so that, as soon as $\epsilon\le \epsilon_0$, the fiber $h^{-1}(t)$ intersects the standard sphere $\big\{ (z_1, \ldots,z_n) \in \C^n \,\big|\, || (z_1,\ldots,z_n) || = R\epsilon\big\}$ transversely for $|t|\le \epsilon$, and the polar curve of $\ell$ meets $S_\epsilon$ in the part where $|z_1|=\epsilon$. 
We set $B'_{\epsilon} = \ell(B_{\epsilon})$. 
In what follows, the Milnor fibers of $z \colon (\C^2,0) \to (\C,0)$ are denoted by $F'$, while the Milnor fibers of $h$ are denoted by $F$. 
Note that by the choice of $B_{\epsilon}$ and $B'_{\epsilon}$, we have $\ell(F_t) = F'_t$ for every $t$ in $D_{\eta}^*$.
Write $e=[v,w]$ and set $\tilde{v} = \tilde{\ell}(v)$ and $\tilde{w} = \tilde{\ell}(w)$. 
Then $\tilde{\ell}(e)$ is a string joining $\tilde{v}$ to $ \tilde{w}$ in $\Gamma_{\sigma_{\pi}}$, and the intersection  $F'_{\tilde{\ell}(e)} =  F'_t  \cap \sigma_{\pi}\big( \cal N\big(\tilde{\ell}(e)\big)\big) $, where $\cal N\big(\tilde{\ell}(e)\big)$ for the string $\tilde{\ell}(e)$ is defined as at the beginning of the proof of Lemma \ref{lemma:deg}, is a disjoint union of annuli.
For each edge $e' =[v,w']$ of $W_{v,e}$, the intersection  $F_{e'} = F_t \cap \tilde{\pi} \big( \cal N(e')\big)$    is a disjoint union of annuli.
After adjusting the bundles $N(E_i)$ and $N(C_j)$ if necessary, we can assume that the cover $\ell$ restricts to a (possibly disconnected) regular cover
\[
\ell' \colon \bigcup_{e' \in W_{v,e}} F_{e'} \to F'_{\tilde{\ell}(e)}.
\] 
On the one hand, the maps $\ell'$ and $\ell_v$ coincide on the boundary of $ F'_{\tilde{\ell}(e)}$ that is inside $\cal N(C_{\tilde{v}})$, and by computing the degree of $\ell'$ over a point on this boundary we see that $\deg(\ell') = \deg({v})$.
On the other hand, computing the degree of $\ell'$ over a point on the boundary of $ F'_{\tilde{\ell}(e)}$ that is inside $\cal N(C_{\tilde{w}})$, we obtain
 \[
 \deg(\ell') = \sum_{e' \in W_{v,e}} \deg( \ell|_{F_{e'}}) =  \sum_{e' \in W_{v,e}}  \deg (w').
 \]
Finally, by blowing-up  all the intersections $E_v \cap E_{v'}$ where $v , v'$ belong to the set of nodes $V_N(\Gamma_\pi)$, we can assume that there are no adjacent nodes in $\Gamma_\pi$. 
Therefore, for each edge $e' =[v,w']$ of $W_{v,e}$, the vertex $w'$ has valency $2$, $\cal N(E_{w'}) \cap F_t$ is a disjoint union of annuli each of whom has common boundary with one of the annuli of $F_{e'}$, and so we have $\deg(e') = \deg(w')$.  
This proves the Lemma.
\end{proof}

The following proposition, which we call the local degree formula, will play a key role in our proof of Theorem~\ref{theorem_main}.

\begin{proposition}\label{proposition:degree_formula}
Let $\pi \colon X_{\pi} \to X$ be a good resolution of $(X,0)$ which factors through the blowup of the maximal ideal and through the Nash transform of $(X,0)$, let $\ell\colon (X,0)\to(\C^2,0)$ be a generic projection with respect to $\pi$, and let $e$ be an edge of $\Gamma_{\pi}$.   
Then 
\[
\mathrm{length} \big(\widetilde \ell (e)\big) = \deg(e)  \mathrm{length}(e).
\]
\end{proposition}

Observe that $\widetilde{\ell}(e)$ in general is not an edge of a dual graph of a sequence of blowups above the origin of $\C^2$, but only a string of edges.
Here we consider its length for the skeletal metric of $\NL(\C^2,0)$.

\begin{proof}
We use again the linear form $h = z_1 |_X$ and Milnor balls $B_{\epsilon}$ and $B'_{\epsilon} = \ell (B_{\epsilon})$ as at the beginning of the proof of Lemma \ref{lem: degree formula 2}.
Consider the fibers $F_t= h^{-1}(t) \cap B_{\epsilon}$ in $(X,0)$ and   $F'_t=\{z_1=t\} \cap B'_{\epsilon}$ in $(\C^2,0)$.
Denote by $v$ and $v'$ the two vertices of $\Gamma_\pi$ adjacent to $e$.
Let $\tilde{\pi} \colon X_{\tilde \pi} \to  X$ be the resolution of $(X,0)$ adapted to $\pi$ and $\ell$, and consider the commutative diagram $\ell \circ \tilde{\pi} = \ell_{\pi} \circ \sigma_{\pi}$ introduced in Section~\ref{subsec:adapted resolution}. 
Consider the string $\cal S= \widetilde{\ell}(e)$ of the dual tree  $\Gamma_{\sigma_{\pi}}$ of $\sigma_{\pi}$, and name its vertices $\nu_1,\ldots,\nu_n$ following the order of the string from $\nu_1=\widetilde \ell(v)$ to $\nu_n=\widetilde \ell(v')$.
Let us re-define  $\sigma_{\pi}\big(\cal N(\cal S)\big) = B'_{\epsilon} \cap \sigma_{\pi}\big(\cal N(\cal S)\big)$ and $\pi\big(\cal N(e)\big)  =B_{\epsilon} \cap  \pi\big(\cal N(e)\big)$ and note that that $\deg(e)$ is by definition the degree of the restriction  $\ell{|_{\pi{\textstyle(}\cal N(e){\textstyle)}}} \colon \pi\big(\cal N(e)\big) \to  \sigma_{\pi}\big(\cal N(\cal S)\big) $.
Next, set $F_{e,t} = F_{t} \cap \pi\big(\cal N(e)\big)$ and $F_{\cal S,t} = F'_t \cap \sigma_{\pi}\big(\cal N(\cal S)\big)$. 
Observe that the cover $\ell$ restricts to a regular cover $\ell_e \colon F_e \to F_{\cal S}$, because there are no polar curves passing by $e$ since $\pi$ factors through the Nash transform of $(X,0)$.
On the one hand, $F_e$ is a disjoint union of  $K = \gcd(m_v, m_{v'})$ annuli, and by Proposition \ref{rk:dehn twist}, the monodromy  $\phi \colon F_t \cap B_{\epsilon} \to F_t \cap B_{\epsilon} $ of the Milnor--L\^e fiber of $h \colon (X,0) \to (\C,0)$  restricts to a  map $\phi_e \colon F_e \to F_e$ with screw number 
\[
\tau = - \frac{1}{m_{v}m_{v'}} = -\mathrm{length}(e).
\]
On the other hand, for each $i=1,\ldots,n-1$, denote by $e_i=[\nu_i,\nu_{i+1}]$ the edge of $\Gamma_{\sigma_\pi}$ connecting $\nu_i$ to $\nu_{i+1}$.
The number $K'=\gcd (m_{\nu_i}, m_{\nu_{i+1}})$ does not depend on the choice of  $i$, as can be verified by computing the intersection number on $Y_{\sigma_\pi}$ of the divisor $\sum_{j=1}^n m_{\nu_j}E_{\nu_j}$ with each of the $E_{\nu_i}$, for $i=2,\ldots,n-1$.
By Proposition~\ref{rk:dehn twist},  $F_{e_i} =  F'_t \cap  \sigma_{\pi}\big(\cal N(e_i)\big)$ is a disjoint union of  $K'$ annuli which are cyclically exchanged by the monodromy $\phi'  \colon F'_t \cap B'_{\epsilon} \to F'_t \cap B'_{\epsilon}$ of the Milnor--L\^e fiber of $z_1 \colon (\C^2,0) \to (\C,0)$ and the restriction  of $\phi' |_{F_{e_i}} \colon F_{e_i} \to F_{e_i}$ has screw number $-\mathrm{length}(e_i)$. 
Now, $F_{\cal S} $   is a disjoint union of  $K'$ annuli and the intersections  $F_{e_i}=F'_t \cap \sigma_{\pi}\big(\cal N(e_i)\big)$, for $i=1,\ldots,n-1$, are concentric unions of $K'$ annuli inside $F_{\cal S}$. 
Therefore, applying  Lemma \ref{lem:twist1} we obtain that the restriction  $\phi'|_{F_{\cal S}} \colon F_{\cal S} \to F_{\cal S}$ has screw number
\[
\tau' =  - \sum_{i=1}^{n-1} \mathrm{length}(e_i) = -  \mathrm{length}(\cal S).
\]
A representative of the monodromy  $\phi' \colon F'_t \to F'_t$   is the  first return diffeomorphism on $F'_t$ of a  $1$-dimensional flow transverse to all the Milnor fibers, lifted from the standard unit tangent vector field  to  the circle $S^1_{|t|}$. 
Since $h  =  z_1 \circ \ell $, a representative of the monodromy $\phi  \colon F_t  \to F_t$  of the Milnor--L\^e fibration of $h$ is obtained by lifting this flow by $\ell$ and by taking its first return diffeomorphism on $F_t$. 
Therefore, we have $\ell \circ \phi = \phi' \circ \ell$. 
Moreover,  the restriction $\ell|_{{F_e}} \colon F_e \to F_{\cal S}$ is a cover with degree $\deg(e)$.  
Applying Lemma \ref{lem:twist2} to $\phi^{K'}$ and $({\phi'})^{K'}$ we then obtain $\tau' = \deg(e) \tau$.  
This concludes the proof.
\end{proof}

\subsection{Lifting $\Delta$ under $\widetilde\ell$.} 
We now explain how to use adapted resolutions and the local degree formula to relate the Laplacian of $\calI_X$ to that of $\calI_{\C^2}$ via the generic projection $\widetilde{\ell}$.

\begin{proposition}\label{proposition_lifting_Delta}
Let  $\pi \colon X_{\pi} \to X$ be a good resolution of $(X,0)$ which factors through the blowup of the maximal ideal and through the Nash transform of $(X,0)$ and let $\ell\colon (X,0)\to(\C^2,0)$ be a generic projection with respect to $\pi$.
Let $\tilde\pi\colon X_{\tilde\pi}\to X$ be the good resolution of $(X,0)$ adapted to $\pi$ and $\ell$.
For each vertex $v$ of $V(\Gamma_{\pi})$, we have   
\[
\Delta_{\Gamma_{\tilde\pi}} \big(\cal I_X \big)(v) =  \deg(v) \Delta_{\Gamma_{\sigma_\pi}}\big(\cal I_{\C^2}\big)\big(\widetilde{\ell}(v)\big).
\]
\end{proposition}

\begin{proof} 
Let $\mathrm{E}_v$ denote the set of those edges of $\Gamma_{\tilde{\pi}}$ that are adjacent to $v$, and for every edge $e$ in $\mathrm{E}_v$, denote by $v_e$ the second extremity of $e$.
On the other hand, set $\nu=\widetilde \ell(v)$ and let $\nu_1, \ldots, \nu_r$ be the vertices of $\Gamma_{\sigma_{\pi}}$ which are adjacent to the vertex $\nu$. 
Since $\pi$ factors through the Nash transform of $(X,0)$, the strict transform of the polar curve of $\ell$ on $X_\pi$ is a union of curvettes of the irreducible components $E_v$, for $v$ ranging among the vertices of $\Gamma_{\pi}$. 
Therefore, the strict transform of the discriminant curve $\Delta = \ell(\Pi)$ by $\sigma_{\pi}$ intersects the exceptional divisor $\sigma_{\pi}^{-1}(0)$ of $\sigma_{\pi}$ at smooth points, not passing through the intersecting points $C_{\nu_i} \cap C_{\nu}$. 
This means that 
the remaining singularities over the points   $p_j=C_{\nu_j} \cap C_{\nu}$ after the pull-back of $\sigma_{\pi}$ by $\ell$ in the construction of Subsection~\ref{subsec:adapted resolution} are quasi-ordinary singularities branched over the curve germ $(C_{\nu_j} \cup C_{\nu}, p_j)$. 
Therefore, for each edge $e$ in $\mathrm{E}(v)$ the image $\tilde{\ell}(e)$ is contained in the edge $[\nu, \nu_j]$ for some $j$ and conversely, for each $j=1,\ldots,r$, the inverse image $\tilde{\ell}^{-1}\big([\nu, \nu_j]\big)$ of the edge $[\nu, \nu_j]$ through $\tilde{\ell}$ is a union of segments in $\Gamma_{\tilde{\pi}}$ that depart from $v$, and each such segment contains exactly one of the edges of $\mathrm{E}_v$.
For each $j$, let $e_{j,1}, \ldots, e_{j,k_j}$ be these vertices. 
We have 
$
\St_v = \coprod_{j=1}^r  \{ e_{j,1}, \ldots, e_{j,k_j}\},
$
and so the Laplacian of $\calI_X$ on $\Gamma_{\tilde\pi}$ at $v$ is
\[
\Delta_{\Gamma_{\tilde\pi}} \big(\cal I_X\big)(v)
 = 
\sum_{e \in \St_v} \frac{\cal I_X(v_e) -  \cal I_X(v)}{\mathrm{length}(e)}
 =
\sum_{j=1}^r \sum_{i=1}^{k_j} \frac{\cal I_X(v_{e_{j,i}}) -  \cal I_X(v)}{\mathrm{length}(e_{j,i})}.
\]
By Lemma~\ref{lemma_definition_inner_rate}, we have $\cal I_X(w) =  \cal I_{\C^2}\big(\widetilde \ell (w)\big)$ for each vertex of $\Gamma_{\tilde{\pi}}$.
On the other hand, by Proposition~\ref{proposition:degree_formula} for every $j=1,\ldots,r$ and $i =1,\ldots,k_j$ we have 
\[
\deg(e_{j,i}) \mathrm{length} (e_{j,i}) =    \mathrm{length}([\nu,\tilde{\ell}(v_{e_{j,i}})]).
\]
Moreover, since $\cal I_{\C^2}$ is linear on the edge $[\nu, \nu_j]$ and $\tilde{\ell}(v_{e_{j,i}}) \in [\nu, \nu_j]$, we also have 
\[
\frac
{\cal I_{\C^2}\big( \tilde{\ell}(v_{e_{j,i}})\big) -  \cal I_{\C^2}(\nu)}
{ \mathrm{length}([\nu,\tilde{\ell}(v_{e_{j,i}})])  } 
=
\frac
{\cal I_{\C^2}(\nu_j) - \cal I_{\C^2}(\nu)}
{ \mathrm{length}([\nu,\nu_j])}\,.
\]
By putting this all together, we obtain 
\[
\Delta_{\Gamma_{\tilde\pi}} \big(\cal I_X\big)(v)
 = 
\sum_{j=1}^r  \Big(     \sum\nolimits_{i=1}^{k_j}   \deg(e_{j,i})    \Big)  
\frac{\cal I_{\C^2}(\nu_j) -  \cal I_{\C^2}(\nu)}
{\mathrm{length}([\nu,\nu_j ])}\,.
\]
Finally, by Lemma~\ref{lem: degree formula 2}, for every $j \in 1,\ldots,r$ we have $\sum_{i=1}^{k_j}    \deg(e_{j,i}) = \deg(v)$.  
We deduce that 
\[
\Delta_{\Gamma_{\tilde\pi}} \big(\cal I_X\big)(v)
 =
\deg(v)  \Delta_{\Gamma_{\sigma_\pi}}\big(\cal I_{\C^2}\big)(\nu),
\]
which is what we wanted to prove.
\end{proof}

\subsection{Lifting $K$ under $\widetilde\ell$.} 
The last ingredient we need for the proof of Theorem~\ref{theorem_main} is a tool to compute also the canonical divisor of a dual resolution graph of our singular surface via a generic projection.

\begin{proposition}\label{proposition_lifting_K} 
Let  $\pi \colon X_{\pi} \to X$ be a good resolution of $(X,0)$ which factors through the blowup of the maximal ideal and through the Nash transform of $(X,0)$ and let $\ell\colon (X,0)\to(\C^2,0)$ be a generic projection with respect to $\pi$.
Let $\tilde\pi\colon X_{\tilde\pi}\to X$ be the good resolution of $(X,0)$ adapted to $\pi$ and $\ell$, and let $v$ be a vertex of $\Gamma_{\pi}$. 
Then
\[
K_{\Gamma_{\tilde\pi}}(v) - m_vp_v    = \deg(v)  K_{\Gamma_{\sigma_{\pi}}}\big(\widetilde{\ell}(v)\big).
\] 
\end{proposition}

\begin{proof}
We use again the notations of Sections~\ref{subsection_monodromy_lengths} and \ref{subsec:degree}: we write $E_v$ for the divisor on $X_{\tilde\pi}$ associated with $v$, and $C_\nu$ for the divisor on $Y_{\sigma_\pi}$ associated with $\nu=\widetilde{\ell}(v)$, and again we consider the cover $\ell_v \colon  \tilde{\pi}\big( \cal N(E_v)\big) \to \cal \sigma_{\pi}\big(N(C_{\nu})\big)$ introduced in Section~\ref{subsec:degree}.
We choose again a linear form $z_1 \colon (\C^2,0) \to (\C,0)$ on $(\C^2,0)$ such that $h=\ell \circ z_1 \colon (X,0) \to (\C,0)$ is a generic linear form on $(X,0)$, as in the proof of the degree formula (Proposition~\ref{proposition:degree_formula}).  
For $t>0$ small enough, let   $F_t = h^{-1}(t) \cap B_{\epsilon}$ be the Milnor--L\^e fiber of $L$ and let $F'_t = \{z_1=t\}\cap B'_{\epsilon}$ be that of $h$, so that we have $\ell (F_t) = F'_t$. 
Set $F_v = F_t \cap  \tilde{\pi}\big(\cal N(E_v)\big)$ and $F'_{\nu} = F'_t \cap  \sigma_{\pi}\big(\cal N(E_{\nu})\big)$. 
Then $\ell_v$  restricts to a map $\ell_v|_{F_v}  \colon F_v \to F'_{\nu}$. 
 
Let $l_v$ (respectively $p_v$) be the number of irreducible components of $h=0$ (resp. of the polar curve of $\ell$) whose strict transforms by $\pi$ intersects $E_v$.
In particular, $l_v>0$ (resp. $p_v>0$) if and only if $v$ is an $\cal L$-node (resp. a $\cal P$-node). 
 
The cover $\ell_v \colon F_v \to F'_{\nu}$ is branched on $p_v m_v$ points of $F_v$. 
Since $\ell$ is generic, the images of these points are $p_v m_v$ distinct points of $F'_{\nu}$ and if $p\in F'_{\nu}$ is one of them,  $\ell_v^{-1}(p)$ consists of $\deg(v)-1$ points. 
Therefore, applying Hurwitz formula we obtain $\chi(F_v) - \big(\deg(v)-1\big)m_{v} p_v =\deg(v)\big( \chi(F'_{\nu}) -m_{v} p_v\big)$, that is
\begin{equation}\label{eq_chi_in_proof}
	\chi(F_v) + m_v  p_v = \deg(v) \chi(F'_{\nu}).
\end{equation}
Let us identify $F_v$ with its pull-back by $\tilde\pi$ and let us consider again the generic linear form $h \colon (X,0) \to (\C^2,0)$ on $(X,0)$. 
If $p$ is a smooth point of $E_v$ which does not intersect the strict transform $h^*$ of $h$ on $X_{\tilde\pi}$, the total transform of $h$ is   of the form $(h \circ \tilde{\pi})(u,v) = u^{m_v}$ in suitable local coordinates $(u,v)$ centered at $p$.  
Therefore, the  map $ \cal N(E_v) \to E_v \cap \cal N(E_v)$ which sends each fiber-disc of $\cal N(E_v)$ to its intersecting point with $E_v$ restricts to a regular cover $\rho_v \colon F_v \to E_v  \cap \cal N(E_v)$ of degree $m_v$. 
Then, applying Hurwitz formula again, we  obtain $\chi(F_v)=m_v \chi\big(E_v  \cap \cal N(E_v)\big)$. 
By the same argument, we also have  $\chi(F'_\nu) = m_{\nu}  \chi\big(C_{\nu} \cap \cal N(C_{\nu})\big)$. 

Assume first that $v$ is not an $\cal L$-node, that is $l_v=0$. 
Then $\chi\big(E_v  \cap \cal N(E_v)\big) = 2-2g(v)-{\val_{\Gamma_{\tilde{\pi}}}}(v)$. 
This implies that $\chi(F_v) = - K_{\Gamma_{\pi}}(v)$.
By the same argument, we have $\chi(F'_{\nu}) = - K_{\Gamma_{\sigma_{\pi}}}(\nu)$.
Using equation \eqref{eq_chi_in_proof}, we deduce that ${K_{\Gamma_{\tilde{\pi}}}}(v)  - m_vp_v    = \deg(v)  K_{\Gamma_{\sigma_{\pi}}}(\nu)$, which is what we wanted.

Assume  now that $v$ is  an $\cal L$-node. 
Then $\chi\big(E_v  \cap \cal N(E_v)\big) = 2-2g(v)- {\val_{\Gamma_{\tilde{\pi}}}}(v)-l_v $, which   implies $\chi(F_v) = - K_{\Gamma_{\pi}}(v)  -m_v l_v$, and  since $\nu$ is the root vertex of $\Gamma_{\sigma_{\pi}}$, we have $m_{\nu}=1$, and we get   $\chi(F'_\nu) =   \chi(C_{\nu})-1 =- K_{\Gamma_{\sigma_{\pi}}}(\nu)-1$. 
Using equation \eqref{eq_chi_in_proof}, we  then  obtain ${K_{\Gamma_{\tilde{\pi}}}}(v)  - m_vp_v +  m_vl_v    = \deg(v) \big( K_{\Gamma_{\sigma_{\pi}}}(\nu) +1\big)$, that is
\begin{equation}\label{eq_chi_in_proof2}
{K_{\Gamma_{\tilde{\pi}}}}(v)  - m_vp_v    = \deg(v)  K_{\Gamma_{\sigma_{\pi}}}(\nu)    + \big(\deg(v) -   m_vl_v\big).
\end{equation}  
Since $v$ is an $\cal L$-node, $\nu  =\widetilde{\ell} (v)$ is the root vertex of the tree $\Gamma_{\sigma_{\pi}}$.  
Let $(\gamma,0) \subset (\C^2,0)$ be a generic  complex line. 
Its  strict transform $\gamma^*$ by $\sigma_{\pi}$ is  a curvette of $E_{\nu}$ and since $m_{\nu}=1$, the intersection $\gamma \cap F'_v$ consists of a single point $p$. 
The cardinal of $\ell_v^{-1}(p)$ is equal to the degree $\deg(v)$ of the cover $\ell_v$  and also to the number of points in the intersection $\ell^{-1}(\gamma) \cap F_v$, which is exactly $m_v l_v$. 
This proves that $\deg(v) - m_vl_v=0$. 
Replacing this in equation \eqref{eq_chi_in_proof2}, we obtain ${K_{\Gamma_{\tilde{\pi}}}}(v)  - m_vp_v    = \deg(v)  K_{\Gamma_{\sigma_{\pi}}}(\nu)$, which is what we wanted to prove. 
\end{proof}


\subsection{Proof of Theorem \ref{theorem_main}} \label{subsec:main proof}
We are finally ready to put together all the pieces we prepared so far and finish the proof of our main theorem.

Thanks to Lemma \ref{lemma_reduction_bigger_modification}, it suffices to prove the theorem for a good resolution $\pi$ of $(X,0)$ that factors through the blowup of the maximal ideal and through the Nash transform of $(X,0)$.   
Again by the same lemma, it is enough to prove it for the resolution $\tilde\pi$ of $(X,0)$ adapted to $\pi$ and to a given projection $\ell$ generic with respect to $\pi$.
As the formula can be easily verified at all vertices of $\Gamma_{\tilde\pi}$ that are not vertices of $\Gamma_\pi$ thanks to Lemma~\ref{lem:inner rate } by the same argument as the one in the proof of Proposition~\ref{proposition_laplacian_smooth}, we just have to prove that the formula holds for $\Delta_{\Gamma_{\tilde\pi}}(\calI_X|_{\Gamma_{\tilde\pi}})$ at a vertex $v$ of $\Gamma_\pi$, that is we want to prove that
\[
\Delta_{\Gamma_{\tilde\pi}}(\cal I_X)(v) - K_{\Gamma_{\tilde\pi}}(v)  - 2 m_v l_v + m_v p_v  =0.
\]
Set $\nu = \widetilde{\ell}(v)$.
As we have already proved the theorem for the modification $\sigma_\pi$ of $(\C^2,0)$ in Proposition~\ref{proposition_laplacian_smooth}, we have $ \Delta_{\Gamma_{\sigma_\pi}}(\cal I_{\C^2})(\nu)  - K_{\Gamma_{\sigma_{\pi}}}(\nu)  - 2=0 $ if $\nu=\ord_0$ is the root of $\NL(\C^2,0)$, while $ \Delta_{\Gamma_{\sigma_\pi}}(\cal I_{\C^2})(\nu)  - K_{\Gamma_{\sigma_{\pi}}}(\nu) = 0$ otherwise.
By Proposition \ref{proposition_lifting_Delta}, we have  $\Delta_{\Gamma_{\tilde\pi}}(\cal I_X)(v)  = \deg(v) \Delta_{\Gamma_{\sigma_\pi}}(\cal I_{\C^2})(\nu)$, while by Proposition \ref{proposition_lifting_K} we have $K_{\Gamma_{\tilde\pi}}(v)  - m_vp_v    = \deg(v)  K_{\Gamma_{\sigma_{\pi}}}(\nu)$. 
Therefore, we have
\[
\Delta_{\Gamma_{\tilde\pi}}(\cal I_X)(v) - K_{\Gamma_{\tilde\pi}}(v)  - 2 m_v l_v + m_v p_v   = \deg(v)\big(\Delta_{\Gamma_{\sigma_\pi}}(\cal I_{\C^2})(\nu)  - K_{\Gamma_{\sigma_{\pi}}}(\nu)\big) -2m_vl_v.
\] 
If $v$ is not an $\cal L$-node, then $l_v = 0$ and we get 
\[
\Delta_{\Gamma_{\tilde\pi}}(\cal I_X)(v) - K_{\Gamma_{\tilde\pi}}(v)  - 2 m_v l_v + m_v p_v   =  \deg(v)\big(\Delta_{\Gamma_{\sigma_\pi}}(\cal I_{\C^2})(\nu)  - K_{\Gamma_{\sigma_{\pi}}}(\nu)\big)  = 0.
\]
If $v$ is an $\cal L$-node, then $\nu=\ord_0$ and, since as shows in the course of the proof of Proposition~\ref{proposition_lifting_K} we have $\deg(v)=m_vl_v$, we obtain:
\[
\Delta_{\Gamma_{\tilde\pi}}(\cal I_X)(v) - K_{\Gamma_{\tilde\pi}}(v) - 2 m_v l_v + m_v p_v   =  \deg(v)\big(\Delta_{\Gamma_{\sigma_\pi}}(\cal I_{\C^2})(\nu)  - K_{\Gamma_{\sigma_{\pi}}}(\nu) -2\big) = 0.
\]
This completes the proof of Theorem~\ref{theorem_main}.\hfill\qed


\section{Applications}
\label{section:applications}

In this section we will derive several consequences of Theorem~\ref{theorem_main}.

\subsection{L\^e--Greuel--Teissier formula} \label{subsection:le-greuel}
 
As a first application of Theorem~\ref{theorem_main}, we recover the L\^e--Teissier formula of \cite{LeTeissier1981} in the particular case of a generic linear form $h \colon (X,0) \to (\C,0)$. 
It is a singular version of the L\^e--Greuel  formula (see \cite{Greuel1975, Le1974} for the classical statement and \cite[Theorem 4.2(A)]{Massey1996} for a more general version).

\begin{proposition}
	\label{proposition:Le-Greuel}
	Let $(X,0)$ be a surface germ with isolated singularity and let  $h \colon (X,0) \to (\C,0)$ be a generic linear form. 
	Consider the Milnor--L\^e fiber $F_t = \{h=t\} \cap B_{\epsilon}$ of $h$, for given real numbers $\epsilon>0$ and $\eta >0$ such that $\eta << \epsilon < 1$. 
	Then we have 
	\[
	m(\Pi,0) = m(X,0)-\chi(F_t),
	\]
	where $m(X,0)$ denotes the multiplicity of $(X,0)$ and $m(\Pi,0)$ denotes the multiplicity of the polar curve $\Pi$ of a generic linear projection $\ell \colon (X,0) \to (\C^2,0)$. 
\end{proposition}

\begin{proof} 
	Let $\pi \colon X_{\pi} \to X$ be a good resolution of $(X,0)$ that factors through the blowup of the maximal ideal and through the Nash transform of $(X,0)$ and let $\Gamma_{\pi}$ be the dual graph of $\pi$. 
	As already observed in Subsection \ref{subsection_preliminaries_laplacians}, the degree of the Laplacian of a piecewise linear function on a metric graph is equal to zero. 
	Applying this on the inner rate function $\cal I_X$, we get
	\[
	\sum_{v \in V(\Gamma_{\pi})} \Delta_{\Gamma_\pi}\big(\calI_X\big)(v) = 0.
	\]
	If we replace the coefficients of this Laplacian by their expression given by Theorem~\ref{thm:B} (which as observed in Remark~\ref{remark:theorem_main_implies_B} is an immediate consequence of Theorem~\ref{theorem_main}), we obtain
	\[
	\sum_{v \in V(\Gamma_{\pi})} m_vl_v-  \sum_{v \in V(\Gamma_{\pi})}   m_v p_v-   \sum_{v \in V(\Gamma_{\pi})}  m_v \big( \chi(\check E_{v}) - l_v\big)  = 0.
	\]
	The first sum equals $m(X,0)$, while the second one equals $m(\Pi,0)$.
	Moreover, using again the notations of the proof of Proposition \ref{proposition_lifting_K}, for all $v \in V(\Gamma_{\pi})$ we have 
	\[
	\chi(F_v) = m_v \chi\big(E_v  \cap \cal N(E_v)\big) = m_v \big( \chi(\check E_{v}) - l_v\big),
	\]
	as $E_v  \cap \cal N(E_v)$ is obtained from $\check E_v$ by removing $l_v$ small discs, one for each point of $E_v$ through which a component of the pullback of a generic linear form passes.
	Since the connected components of $F_t \setminus \coprod_{v \in V(\Gamma_{\pi})}F_v$ are annuli and have therefore trivial Euler characteristic, the additivity of $\chi$ gives
	\[
	\chi (F_t) = \sum_{v \in V(\Gamma_{\pi})}  \chi(F_v) = \sum_{v \in V(\Gamma_{\pi})}  m_v \big( \chi(\check E_{v}) - l_v\big),
	\]
	which completes the proof.
\end{proof}

\subsection{A numeric version of Theorem \ref{theorem_main}} \label{prop:alternative}

We will now prove an alternative, more numeric version of  Theorem~\ref{theorem_main} which will be later used to show our Theorem~\ref{thm:A} from the introduction. 

First, we need to introduce some notation.   
Let $(X,0)$ be a complex surface germ with isolated singularities and let $\pi \colon X_{\pi} \to X$ be a good resolution of $(X,0)$. 
Let $h: (X,0) \to (\C,0)$ be a generic linear form on $(X,0)$ and denote by $h^*$ the strict transform via $\pi$ of the curve defined by $h=0$ in $X$.
For every vertex $v$ of $\Gamma_\pi$ denote by $l_\pi(v)$ the intersection multiplicity $h^*\cdot E_v$ of $h^*$ with the exceptional component $E_v$ of $\pi^{-1}(0)$ associated with $v$.
Observe that whenever $\pi$ factors through the blowup of the maximal ideal of $(X,0)$, then $l_\pi(v)$ coincides with the integer $l_v$ defined in the introduction. 
Similarly, let $\ell \colon (X,0) \to (\C^2,0)$ be a generic plane projection of $(X,0)$ with respect to $\pi$, and denote by $\pi_\pi(v)$ the intersection multiplicity $\Pi^*\cdot E_v$ of the strict transform $\Pi^*$ via $\pi$ of the polar curve $\Pi$ associated with $\ell$ with the exceptional curve $E_v$.
Whenever $\pi$ factors through the Nash transform of $(X,0)$, then $p_{\pi}(v) =p_{v}$, where, as in the introduction, $p_{v}$ denotes the number of connected components of $\Pi^*$ which meet $E_{v}$.

\begin{remark}  \label{rk:p2} 
Let  $h \colon (X,0) \to (\C,0)$ be a generic linear form and let $(\delta,0)$ be an irreducible curve germ on $(X,0)$ which is not a component of $h^{-1}(0)$ and whose strict transform $\delta^*$ by $\pi$ meets an exceptional component $E_v$ at a point $p$.
Then, by writing $h \circ \pi$ in local coordinates centered at $p$, we easily see that the intersection multiplicity $(E_v, \delta^*)_p$ of the curve germs $(E_v,p)$ and $(\delta^*,p)$ at $p$ equals $m_v^{-1}\mult_0(\delta)$, where $\mult_0(\delta)$ denotes the multiplicity of $\delta$ at $0$. 
Therefore, the integer $p_\pi(v)$ can also be computed as 
 	\[
 	p_{{\pi}}(v)
 	 = h^*\cdot E_v
 	 = \sum_{\delta \subset \Pi_{v}} \delta^*\cdot E_v
 	 = \sum_{\delta \subset \Pi_{v},\, p \in E_v} (E_v, \delta^*)_p,
 	\]
where the sums are indexed over the set $H_v$ of components of $h$ whose strict transform via $\pi$ meets $E_v$.
The analogous statement holds for $\Pi$ and the integers $p_\pi(v)$.
\end{remark}

To keep the notation short and consistent with the statements given in the introduction, we denote by $q_v=\calI_X(v)$ the inner rate of $v$.
We can now state the main result of the subsection.

\begin{proposition} \label{proposition:application_linear_algebra} 
	Let $(X,0)$ be the germ of a complex surface with isolated singularities, let $\pi \colon X_{\pi} \to X$ be a good resolution of $(X,0)$ and let $\Pi$ be the polar curve of a projection $\ell \colon (X,0) \to (\C^2,0)$ that is generic with respect to $\pi$.
	Let $v_1, \ldots, v_n$ be the vertices of the dual graph $\Gamma_{\pi}$ and denote by $E_{v_1},\ldots,E_{v_n}$ the corresponding irreducible components of $\pi^{-1}(0)$. 
	Consider the vectors $\underline{a} = {}^t(m_{v_1} q_{v_1}, \ldots, m_{v_n} q_{v_n})$, $\underline{L} = {}^t\big(l_{\pi}(v_1), \ldots, l_{\pi}(v_n)\big)$, $\underline{P} = {}^t\big(p_{\pi}(v_1) , \ldots, p_{\pi}(v_n)\big)$, and $\underline{K}=  {}^t\big(  k_\pi(v_1), \ldots, k_\pi(v_n)\big)$, where $k_\pi(v_i) = \val_{\Gamma_\pi}(v_i)+2g(v_i)-2=m_{v_i}^{-1}K_{\Gamma_{\pi}}(v_i)$, and denote by ${M=(E_{v_i}\cdot E_{v_j})_{1\leq i,j\leq n}}$ the intersection matrix of the exceptional divisor $\pi^{-1}(0)$ of $\pi$.	
	Then we have
	\[
	M \underline{a} = \underline{K}  +  \underline{L}    -  \underline{P}.
	\]	
\end{proposition}

\begin{proof}
Let $h \colon (X,0) \to (\C,0)$ be a generic linear form on $(X,0)$ such that the strict transform $h^*$of $h^{-1}(0)$ by $\pi$ does not intersect the strict transform of the polar curve $\Pi$.
Assume for the time being that $\pi \colon X_{\pi} \to X$ factors through the blowup of the maximal ideal and through the Nash transform of $(X,0)$. 
Since the total transform $(h)$ of $h$ by $\pi$, which is $(h) = \sum_{i=1}^n m_{v_i} E_{v_i} + h^*$, is a principal divisor (see \cite[Theorem 2.6]{Laufer1972}), for every vertex $v$ of $\Gamma_{\pi}$ we have $(h) \cdot E_v = 0$, that is 
\begin{equation}\label{eqn:proof_application}
\sum_{[v,v']} m_{v'}  + m_v E_v^2 + l_{v} = 0,
\end{equation}
where the sum runs over the edges of $\Gamma_\pi$ adjacent to $v$.
Consider the coefficient of the Laplacian at a vertex $v$, as given by the formula of Theorem~\ref{theorem_main}, and divide both sides by $m_v$.
This yields the equality
\[
\sum_{[v,v']} m_{v'} (q_{v'}-q_v) = \val_{\Gamma_\pi}(v) + 2g(E_v) - 2 + 2l_v -p_v,
\]
which combined with equation \eqref{eqn:proof_application} gives
\[
q_v m_v E_v^2 + \sum_{[v,v']} m_{v'}q_{v'} = \val_{\Gamma_\pi}(v) + 2g(E_v) - 2 + (2-q_v)l_v -p_v.
\]
Now, observe that $l_v$ vanishes unless $v$ is an $\cal L$-node of $(X,0)$, in which case $q_v=1$, so that for every $v$ we have $(2-q_v)l_v=l_v$. 
This proves that $M \underline{a} = \underline{K}  +  \underline{L}    -  \underline{P}$.
 
Let us now prove the formula without the assumption that  $\pi \colon X_{\pi} \to X$ factors through the blowup of the maximal ideal and through the Nash transform of $(X,0)$.
Since there exists a good resolution $\hat{\pi}$ of $(X,0)$ that dominates $\pi$ and factors through both, it is sufficient to work by induction and prove that if $\rho \colon X_{\pi'} \to  X_{\pi}$ is the blowup of a point $q \in \pi^{-1}(0)$ and if the proposition holds for  $\pi' = \pi \circ \rho$, then it also holds for $\pi$. 
	
Let $E_v$ be a component of $\pi^{-1}(0)$ passing through $q$ and denote by $E_{v''}$ the exceptional component created by the blowup $\rho$.
Observe that if $C$ is a curve germ on $X_\pi$ at $q$, since $E_{v''}^{2}=-1$ then by basic intersection theory we have
\begin{align*}\label{eq:intersection_number_blowup}
C\cdot E_v & = (C^* + d E_{v''})\cdot (E_v+E_{v''}) = C^* \cdot (E_v+E_{v''}) + d (E_v\cdot E_{v''}+ E_{v''}^{-1}) \\
& = C^* \cdot (E_v+E_{v''}) + d (1-1)= C^* \cdot E_v+ C^* \cdot E_{v''}\,,
\end{align*}
where $C^*$ denotes the strict transform of $C$ via $\rho$ and $d$ is the order of $C$ in $q$.
Note that the leftmost intersection number above is computed in $X_\pi$, while all the others are computed in $X_{\pi'}$.
In particular, we have
\begin{equation}\label{eq:intersection_number_blowup}
l_{{\pi}}(v) = l_{{\pi'}}(v) +l_{{\pi'}}(v'') \quad \text{ and } \quad p_{{\pi}}(v) = p_{{\pi'}}(v) +p_{{\pi'}}(v'').
\end{equation}
	
Assume now that $q$ is an intersection point of two irreducible components $E_v$ and $E_{v'}$ of $\pi^{-1}(0)$.
Since the formula that we want to prove is true for the vertex $v''$ of $\Gamma_{\pi'}$ by our inductive hypothesis, we have
\begin{equation} \label{eq1}
-{m_{v''}} q_{v''} + m_v q_v + m_{v'} q_{v'}= {l_{\pi'}(v'')} -p_{\pi'}(v''),
\end{equation}
where we have used the fact that $k_{\pi'}(v'')=0$.
On the other hand, the formula is also true for $v$ in $\Gamma_{\pi'}$, so we have: 
	\begin{equation} \label{eq2}
 (E_{v}^2-1)m_v q_v + m_{v''} q_{v''}  + \sum_{[v,w], w \neq v''} m_w q_w  = k_{\pi'}(v) +l_{\pi'}(v) - p_{{\pi'}}(v),
\end{equation}
where $E_{v}^2$ denotes the self-intersection of $E_v$ in $X_{\pi}$. 
Observe that the integer $\sum_{[v,w], w \neq v''} m_w q_w$, where the sum runs over edges of $\Gamma_{\pi'}$, is equal to the sum  $\sum_{[v,w], w \neq v'} m_w q_w$ running over edges of $\Gamma_{\pi}$, and that moreover we have $k_{\pi'}(v) = k_{\pi}(v)$. 
Therefore, by summing the equalities~\eqref{eq1} and \eqref{eq2} and using equation~\eqref{eq:intersection_number_blowup}, we obtain
\[
E_{v}^2m_v q_v + m_{v'} q_{v'} + \sum_{[v,w], w \neq v'} m_w q_w = k_{\pi}(v) + l_{\pi'}(v) - p_{{\pi}}(v),
\]
which is exactly what we wanted to prove for the vertex $v$. 
By symmetry, the formula also holds for $v'$ in $\Gamma_{\pi}$. 
Moreover, the formula remains unchanged for all the vertices different from $v$ and $v'$. 
	
Assume now that $q\in E_v$ is a smooth point of $\pi^{-1}(0)$.  
The formula for $v''$ in $\Gamma_{\pi'}$ is:
\begin{equation} \label{eq3}
-m_{v''} q_{v''} + m_v q_v = l_{\pi'}(v'') - p_{\pi'}(v''),
\end{equation}
while the formula for $v$ in $\Gamma_{\pi'}$ is equation \eqref{eq2} as before. 
Again, we obtain the formula we wanted for $v$ in $\Gamma_{\pi}$ by summing the two equalities~\eqref{eq2} and \eqref{eq3} and applying equation~\eqref{eq:intersection_number_blowup}.
\end{proof}

\subsection{From global geometric data to inner rates}
\label{subsection:global_to_local}

The following result, which is a generalization of Theorem~\ref{thm:A} given in the introduction, is a simple consequence of Proposition~\ref{proposition:application_linear_algebra}.
It explains in which sense the metric structure of the germ $(X,0)$, which is a very local datum, is determined by more global geometric data, namely by the topology of a good resolution of $(X,0)$ and the position of the components of a generic hyperplane section of $(X,0)$ and of the components of the polar curve of a generic projection of $(X,0)$ onto $(\C^2,0)$.

\begin{corollary}\label{corollary_global_to_inner}
	Let $(X,0)$ be an isolated complex surface singularity and let 
	$\pi\colon X_\pi\to X$ be a good resolution of $(X,0)$ that factors through the blowup of the maximal ideal of $(X,0)$. 
	Then the inner rates of the vertices of $\Gamma_\pi$ are determined by the following data:
\begin{enumerate}
	\item 
		\label{corollary_global_to_inner_topology}
		the topological data consisting of the dual graph $\Gamma_\pi$ decorated with the Euler classes and the genera of its vertices;
		
	\item 
		\label{corollary_global_to_inner_hypsections}
		the number $l_v$, for every vertex $v$ of $\Gamma_\pi$;
		
	\item \label{corollary_global_to_inner_polar}
		the number $p_\pi(v)$, for every vertex $v$ of $\Gamma_\pi$.
\end{enumerate}
Moreover, if $\pi$ also factors through the Nash transform of $(X,0)$, then the data above determines completely the inner rate function $\calI_X$ on the whole of $\NL(X,0)$, and hence the local inner metric structure of the germ $(X,0)$.
\end{corollary}

Observe that we do not need to require the metric on $\Gamma_\pi$ to be part of the initial data.
Indeed, the multiplicities $m_v$ which we use to define it can be easily deduced from the data given in \ref{corollary_global_to_inner_topology} and \ref{corollary_global_to_inner_hypsections}, as will be explained in the proof.

\begin{proof} Number the vertices of $\Gamma_\pi$ as $v_1,\ldots,v_n$, and denote by $M=(E_{v_i}\cdot E_{v_j})_{1\leq i,j\leq n}$ the intersection matrix of the exceptional divisor $\pi^{-1}(0)$ of $\pi$. 
Then, in the notation of Proposition~\ref{proposition:application_linear_algebra}, the data described in the statement determine the three vectors $\underline{K}$, $\underline{L}$ and $\underline{P}$.
By combining the equations \eqref{eqn:proof_application} for $v=v_1,\ldots,v_n$ we obtain the linear system $M\underline m = -L$, where $\underline m = {}^t(m_{v_1},\ldots,m_{v_n})$ is the vector of the multiplicities of $h$ along  the exceptional components. 
The matrix $M$, being negative definite, is invertible, and therefore we can retrieve the vector of multiplicities as $\underline m = - M^{-1}L$.
Set $\underline b = \underline{K} +\underline{L} - \underline{P}$.  
By Proposition \ref{proposition:application_linear_algebra}, we have $M \underline a=\underline{b}$. 
We then have $\underline a = M^{-1} \underline b $, and therefore we obtain the inner rates $q_v$ for every vertex $v$ of $\Gamma_\pi$ by dividing each entry of $\underline a $ by the corresponding multiplicity $m_v$.
Now, if $\pi$ factors through the Nash transform of $(X,0)$, as $\calI_X$ is linear on the edges of $\Gamma_\pi$ (Corollary~\ref{cor:linearity on strings_general}) this determines the inner rate function $\calI_X$ on the dual graph $\Gamma_\pi$.
It remains to show that $\cal I_X$ is completely determined by its restriction to $\Gamma_\pi$.
This is the content of the next lemma (Lemma~\ref{lemma_inner_rate_determined_on_skeleton}).
\end{proof}

The next result will require us to consider a different metric $d$ on $\NL(X,0)$.
This is defined on a dual graph $\Gamma_\pi$ in a similar way as the one introduced in Section~\ref{subsection_preliminaries_NL}, by declaring the length of an edge connecting two divisorial valuations $v$ and $v'$ to be equal to $\lcm\{m_v,m_{v'}\}^{-1}=\gcd\{m_v,m_{v'}\}/m_vm_{v'}$.
Being stable under further blowup by a computation analogous to the one performed in Section~\ref{subsection_preliminaries_NL}, these metrics on dual graphs induce a metric $d$ on $\NL(X,0)$.
The difference between the two metrics has also been discussed in \cite[Section 7.4.10]{Jonsson2015} and, in a more arithmetic setting, in \cite[Remark 2.3.4]{BakerNicaise2016}.
Recall that $r_\pi\colon \NL(X,0)\to\Gamma_\pi$ denotes the usual retraction map.

\begin{lemma}\label{lemma_inner_rate_determined_on_skeleton}
	Let $\pi$ be a good resolution of $(X,0)$ which factors through the the blowup of the maximal ideal and through the Nash transform of $(X,0)$. 
	Then, for every $v$ in $\NL(X,0)$, we have
	\[
	\calI_X(v)=\calI_X\big(r(v)\big) + d\big(v,r_\pi(v)\big).
	\]
\end{lemma}

Since $\calI_X$ is linear on the edges of $\Gamma_\pi$, this actually proves that $\cal I_X$ is completely determined by its restriction to the vertices of $\Gamma_\pi$.

\begin{proof}
	Let $\pi'$ be the minimal good resolution of $(X,0)$ that factors through $\pi$ and such that $v$ is contained in $\Gamma_{\pi'}$.
	Observe that, if $\pi''$ is a good resolution of $(X,0)$ that is sandwiched in between $\pi'$ and $\pi$, so that $\Gamma_\pi\subset\Gamma_{\pi''}\subset\Gamma_{\pi'}$, then we have $d\big(v,r_{\pi}(v)\big)=d\big(v,r_{\pi''}(v)\big)+d\big(r_{\pi''}(v),r_{\pi}(v)\big)$.
	This equality, together with the fact that $\pi'$ can be obtained from $\pi$ as a sequence of point blowups, allows us to reduce ourselves without loss of generality to the case where $\pi'$ is the composition of $\pi$ with the blowup of $X_\pi$ at a smooth point $p$ of the exceptional divisor $\pi^{-1}(0)$.
	Let $v_0$ denote the vertex of $\Gamma_\pi$ corresponding to the divisor $E_{v_0}$ containing $p$ and let $v_1$ denote the vertex of $\Gamma_{\pi'}$ corresponding to the exceptional divisor of the blowup at $p$, so that $v$ is a point of the edge $e=[v_0,v_1]$ of $\Gamma_{\pi'}$.
	Since $\calI_X$ is linear on $e$, to prove the theorem it is sufficient to prove that $\calI_X(v_1)=\calI_X(v_0)+d(v_0,v_1)$.
	This follows immediately from the definition of $d$ and from Lemma~\ref{lem:inner rate }.\ref{lem:inner rate  SMOOTH PT}.
\end{proof}

\begin{remark}\label{remark_laplacian_determines_function}
	Instead of relying on Proposition~\ref{proposition:application_linear_algebra}, part of the proof of Corollary~\ref{corollary_global_to_inner} 
	can also be replaced by a purely combinatorial argument.
	Indeed, one can prove that if $F$ is a piecewise linear map on a metric graph $\Gamma$, then the Laplacian $\Delta_\Gamma(F)$ of $F$ determines $F$ uniquely up to an additive constant as follows.
	Assume that $F$ and $F'$ are two piecewise linear functions on $\Gamma$ with the same Laplacian, so that $G = F-F'$ is a piecewise linear function such that $\Delta_\Gamma(G)=0$. 
	In order to show that $G$ is constant one can reason by contradiction as follows.
	Assume that there exists an oriented segment $[v,v']$ in $\Gamma$ along which $G$ is strictly increasing.
	Then, as $\Delta_\Gamma(G)(v')=0$, there exists a different segment $[v',v'']$ along which $G$ is strictly increasing as well.
	By iterating the same construction we obtain an infinite chain of segments along which $G$ grows.
	As $\Gamma$ is finite, this yields a closed path along which $G$ strictly increases, which is absurd.
\end{remark}

\begin{remark} 
	In earlier papers on the subject the inner rates where always computed by considering a generic projection $\ell \colon (X,0) \to (\C^2,0)$ and lifting the inner rates of the components of the exceptional divisor of a suitable resolution of the discriminant curve of $\ell$. 
	For this reason, only simple examples have been computed, since it is generally very hard to decide whether a projection is generic, and moreover computing discriminant curves is not simple. 
	Outside of the simplest examples, the calculations were usually done via computational tools such as Maple (see for example \cite[Example 15.2]{BirbrairNeumannPichon2014} or \cite[Appendix]{SampaioFernandes2018}).
	However, it is generally much simpler to compute the data appearing in Corollary~\ref{corollary_global_to_inner}, and particularly so for hypersurfaces in $(\C^3,0)$.
	This means that our result often allows for a much simpler computation of the inner rates, as is the case in the next example, where we compute the inner rates for the surface singularity of \cite[Example 15.2]{BirbrairNeumannPichon2014}. 
\end{remark}

\begin{example}\label{example:very_big}
	Consider the hypersurface singularity $(X,0)\subset(\C^3,0)$ defined by the equation $(zx^2+y^3)(x^3+zy^2)+z^7=0$. 
	A good resolution $\pi\colon X_\pi\to X$ of $(X,0)$ can be computed explicitly, we refer to \cite[Example 15.2]{BirbrairNeumannPichon2014} for the details.
	The exceptional divisor of $\pi$ is a configuration of copies $\mathbb P^1$ whose dual graph $\Gamma_\pi$ is drawn in Figure~\ref{fig:very big example} represents the  dual graph $\Gamma_{\pi}$. 
	The vertices of $\Gamma_\pi$ are decorated with their self-intersection, with solid arrows representing the components of the polar curve $\Pi$ of a generic projection of $(X,0)$, and with dashed arrows representing those of a generic hyperplane section of $(X,0)$.
 \begin{figure}[h]
  \begin{center}
\begin{tikzpicture}
  \draw[thin,dashed, >-stealth,->](-1.5,2.5)--+(-1.2,0.2);
    \draw[thin,dashed, >-stealth,->](-1.5,2.5)--+(-1.2,0.4);
      \draw[thin,dashed, >-stealth,->](-1.5,2.5)--+(-1.2,0.6);

   \draw[thin,>-stealth,->](-1.5,2.5)--+(-1.2,-0.2);
    \draw[thin,>-stealth,->](-1.5,2.5)--+(-1.2,-0.4);
     \draw[thin,>-stealth,->](-1.5,2.5)--+(-1.2,-0.6);

      \draw[thin,dashed, >-stealth,->](1.5,2.5)--+(1.2,0.2);
    \draw[thin,dashed, >-stealth,->](1.5,2.5)--+(1.2,0.4);
      \draw[thin,dashed, >-stealth,->](1.5,2.5)--+(1.2,0.6);

   \draw[thin,>-stealth,->](1.5,2.5)--+(1.2,-0.2);
    \draw[thin,>-stealth,->](1.5,2.5)--+(1.2,-0.4);
     \draw[thin,>-stealth,->](1.5,2.5)--+(1.2,-0.6);

 \draw[thin,>-stealth,->](-2,0)--+(-1.2,0.4);
  \draw[thin,>-stealth,->](2,0)--+(1.2,0.4);
    \draw[thin,>-stealth,->](0,0)--+(0,1.4);

    \draw[thin,>-stealth,->](0,3.25)--+(0.2,1.2);
     \draw[thin,>-stealth,->](0,2.86)--+(0.5,1.1);
                  \draw[thin,>-stealth,->](0,2.5)--+(0.6,1);
                   \draw[thin,>-stealth,->](0,2.12)--+(0.7,0.9);

              \draw[thin,>-stealth,->](0,1.75)--+(0.8,0.7);

 \draw[] (-2,0)circle(2.5pt);
  \draw[thin ](-2,0)--(-1,1);

  \draw[thin ](0:0)--(-1,1);
 
     \draw[thin ](0:0)--(1,1);

        \draw[thin ](1,1)--(2,0);
            \draw[] (2,0)circle(2.5pt);
   \draw[thin ](1,1)--(1.5,2.5);
    \draw[thin ](-1,1)--(-1.5,2.5);
    \draw[ fill] (1.5,2.5)circle(2.5pt);
     \draw[ fill] (-1.5,2.5)circle(2.5pt);
     \draw[thin ](-1.5,2.5)--(1.5,2.5);

\draw[fill] (-1,1)circle(2.5pt);
\draw[fill] (1,1)circle(2.5pt);
\draw[fill] (0,0)circle(2.5pt);
\draw[fill ] (2,0)circle(2.5pt);
\draw[fill ] (-2,0)circle(2.5pt);

\draw[thin] (-1.5,2.5)..controls (-0.5,3) and (0.5,3)..(1.5,2.5);
\draw[thin] (-1.5,2.5)..controls (-0.5,2) and (0.5,2)..(1.5,2.5);
\draw[thin] (-1.5,2.5)..controls (-0.5,3.5) and (0.5,3.5)..(1.5,2.5);
\draw[thin] (-1.5,2.5)..controls (-0.5,1.5) and (0.5,1.5)..(1.5,2.5);
\node(a)at(2,-0.35){-2};
\node(a)at(1,0.65){-1};
\node(a)at(0,-0.35){-5};
\node(a)at(-1,0.65){-1};
\node(a)at(-2,-0.35){-2};

\node(a)at(-0.4,3.5){-1};
\node(a)at(-0.4,1.6){-1};
\node(a)at(-1.7,2.2){-23};
\node(a)at(1.7,2.2){-23};
 
     \draw[ fill] (0,2.5)circle(2.5pt);
           \draw[ fill] (0,2.86)circle(2.5pt);
             \draw[ fill] (0,2.12)circle(2.5pt);
               \draw[ fill] (0,1.75)circle(2.5pt);
                 \draw[ fill] (0,3.25)circle(2.5pt);

  \end{tikzpicture} 
  \end{center}
\caption{}\label{fig:very big example}\end{figure}

\noindent We deduce from the data contained in Figure~\ref{fig:very big example} the vectors $\underline{m} = -M^{-1} \underline{L}$ and $\underline{b}=\underline{K}+ \underline{L} - \underline{P}$. 
We thus obtain the graph on the left of Figure~\ref{fig:very big example2}, whose vertices are decorated by the pairs $(m_v, b_v)$. Finally, by computing $ \underline a={}^t(q_{v_1}m_{v_1},\ldots,q_{v_n}m_{v_n}) =M^{-1} \underline b$ and dividing each entry by the corresponding  multiplicity  $m_{v_i}$, we deduce the inner rates $q_{v_i}$, which are inscribed on the graph on the right of Figure~\ref{fig:very big example2}. 

 \begin{figure}[h]
  \begin{center}
\begin{tikzpicture}
  
 \draw[] (-2,0)circle(2.5pt);
  \draw[thin ](-2,0)--(-1,1);
  \draw[thin ](0:0)--(-1,1);
     \draw[thin ](0:0)--(1,1);
     
        \draw[thin ](1,1)--(2,0);
            \draw[] (2,0)circle(2.5pt);
   \draw[thin ](1,1)--(1.5,2.5);
    \draw[thin ](-1,1)--(-1.5,2.5);
    \draw[ fill] (1.5,2.5)circle(2.5pt);
     \draw[ fill] (-1.5,2.5)circle(2.5pt);
     \draw[thin ](-1.5,2.5)--(1.5,2.5);

\draw[fill] (-1,1)circle(2.5pt);
\draw[fill] (1,1)circle(2.5pt);
\draw[fill] (0,0)circle(2.5pt);
\draw[fill ] (2,0)circle(2.5pt);
\draw[fill ] (-2,0)circle(2.5pt);

\draw[thin] (-1.5,2.5)..controls (-0.5,3) and (0.5,3)..(1.5,2.5);
\draw[thin] (-1.5,2.5)..controls (-0.5,2) and (0.5,2)..(1.5,2.5);
\draw[thin] (-1.5,2.5)..controls (-0.5,3.5) and (0.5,3.5)..(1.5,2.5);
\draw[thin] (-1.5,2.5)..controls (-0.5,1.5) and (0.5,1.5)..(1.5,2.5);
\node(a)at(2,-0.35){$(5, -2)$};
\node(a)at(1.8,0.9){$(10, 1 )$};
\node(a)at(0,-0.35){$(4,-1)$};
\node(a)at(-1.8,1.1){$(10, 1)$ };
\node(a)at(-2,-0.35){$(5,-2 )$};

\node(a)at(0,3.6){$(2,-1)$};
\node(a)at(0,1.3){$(2,-1)$};
\node(a)at(-2,2.5){$(1,4)$};
\node(a)at(2,2.5){$(1,4)$};
 
     \draw[ fill] (0,2.5)circle(2.5pt);
           \draw[ fill] (0,2.86)circle(2.5pt);
             \draw[ fill] (0,2.12)circle(2.5pt);
               \draw[ fill] (0,1.75)circle(2.5pt);
                 \draw[ fill] (0,3.25)circle(2.5pt);

                 
                 \begin{scope}[xshift=6cm]
                 
                  \draw[] (-2,0)circle(2.5pt);
  \draw[thin ](-2,0)--(-1,1);
  \draw[thin ](0:0)--(-1,1);
     \draw[thin ](0:0)--(1,1);
     
        \draw[thin ](1,1)--(2,0);
            \draw[] (2,0)circle(2.5pt);
   \draw[thin ](1,1)--(1.5,2.5);
    \draw[thin ](-1,1)--(-1.5,2.5);
    \draw[ fill] (1.5,2.5)circle(2.5pt);
     \draw[ fill] (-1.5,2.5)circle(2.5pt);
     \draw[thin ](-1.5,2.5)--(1.5,2.5);

\draw[fill] (-1,1)circle(2.5pt);
\draw[fill] (1,1)circle(2.5pt);
\draw[fill] (0,0)circle(2.5pt);
\draw[fill ] (2,0)circle(2.5pt);
\draw[fill ] (-2,0)circle(2.5pt);

\draw[thin] (-1.5,2.5)..controls (-0.5,3) and (0.5,3)..(1.5,2.5);
\draw[thin] (-1.5,2.5)..controls (-0.5,2) and (0.5,2)..(1.5,2.5);
\draw[thin] (-1.5,2.5)..controls (-0.5,3.5) and (0.5,3.5)..(1.5,2.5);
\draw[thin] (-1.5,2.5)..controls (-0.5,1.5) and (0.5,1.5)..(1.5,2.5);
\node(a)at(2,-0.35){$  \frac{  7}{  5}$};
\node(a)at(1.4,1){$  \frac{  6}{  5}$};
\node(a)at(0,-0.35){$  \frac{  5}{  4}$}; 
\node(a)at(-1.4,1){$  \frac{  6}{  5}$};
\node(a)at(-2,-0.35){$  \frac{  7}{ 5}$};

\node(a)at(0,3.6){$ \frac{ 3}{  2}$};
\node(a)at(0,1.3){$ \frac{  3}{  2}$};
\node(a)at(-1.8,2.5){$ 1$};
\node(a)at(1.8,2.5){$ 1$};

\node(a)at(2.5,1.5){};
 
     \draw[ fill] (0,2.5)circle(2.5pt);
           \draw[ fill] (0,2.86)circle(2.5pt);
             \draw[ fill] (0,2.12)circle(2.5pt);
               \draw[ fill] (0,1.75)circle(2.5pt);
                 \draw[ fill] (0,3.25)circle(2.5pt);
                 
                 \end{scope}

  \end{tikzpicture} 
  \end{center}
\caption{On the left: pairs $(m_v, b_v)$. On the right: inner rates $q_v$.}\label{fig:very big example2}\end{figure}
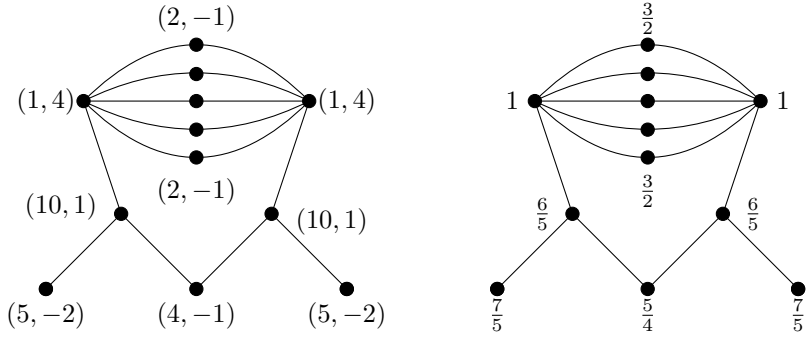
\end{example}

\subsection{Restrictions on generic hyperplane sections and generic polar curves.}
\label{subsection:restrictions_L_P}
Proposition \ref{proposition:application_linear_algebra}  gives fairly strong restrictions on the possible values of $\underline{L}  -  \underline{P}$ and so on the possible relative positions of the $\cal L$ and $\cal P$-nodes, as a consequence of the fact that the coefficients of the vectors $\underline L$ and $\underline P$ must be positive and that $q_v$ must be always at least 1, and precisely 1 at the $\cal L$-nodes of $\NL(X,0)$. 
This can be interpreted as a first step towards the famous question of D. T. L\^e (see \cite[Section 4.3]{Le2000} or \cite[Section 8]{BondilLe2002}) inquiring about the existence of a duality between the two main algorithms of resolution of a complex surface, via normalized blowups of points (\cite{Zariski1939}) or via normalized Nash transforms (\cite{Spivakovsky1990}).

\begin{example} \label{example:BS}
To illustrate this, let us consider the graph of Figure \ref{fig:BS1}.
This graph is the minimal dual resolution graph of any member of the classical Brian\c{c}on--Speder family introduced in \cite{BrianconSpeder1975}, the family of hypersurfaces in $(\C^3,0)$ defined by the equations $x^5+z^{15}+y^7z+txy^6=0$, which depend on the parameter $t$ (it is a $\mu$-constant family which is not Whitney equisingular). 
The number {$12$} between brackets means that the corresponding exceptional component has genus {$12$}; we warn the reader that this genus was mistakenly claimed to be 8 in \cite[Example 15.7]{BirbrairNeumannPichon2014} (note that this has no impact on the validity of the results of that paper).
\begin{figure}[h]
	\begin{center}
		\begin{tikzpicture}
		\draw[thin ](0,0)--(1.5,0);
		
		\draw[fill] (0,0)circle(2pt);
		\draw[fill] (1.5,0)circle(2pt);
		
		\node(a)at(0,0.3){ $-3$};
		\node(a)at(0,-0.3){{$[12]$}};
		\node(a)at(1.5,0.3){ $-2$};
		
		\end{tikzpicture} 
		\caption{ }\label{fig:BS1}
	\end{center}
\end{figure}

\noindent As shown in \emph{loc.\ cit.}, the configurations of $\cal L$-nodes differ for $t\neq 0$ and $t=0$.  
The dual graphs of the minimal resolutions factoring through the blowup of the maximal ideal are depicted in Figure \ref{fig:BS1-L nodes}.  
The numbers between parenthesis are the multiplicities $m_v$ and the arrows represent the components of a generic hyperplane section of $(X,0)$. 
We label the vertices of both configurations by $v_1,\ldots,v_4$.

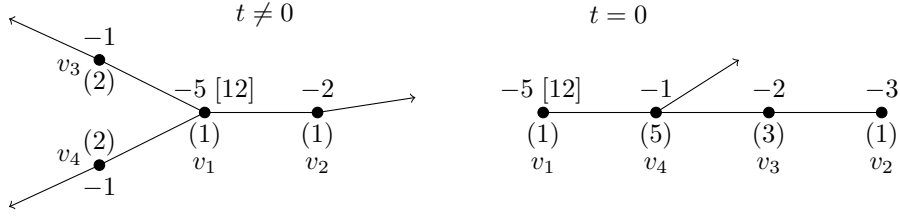
\begin{figure}[h]
	\begin{center}
		\begin{tikzpicture}
		
		\draw[thin ](0,0)--(1.5,0);
		\draw[thin ](0,0)--+(-1.4,0.7);
		\draw[thin ](0,0)--+(-1.4,-0.7);

		\draw[thin,>-stealth,->](1.5,0)--+(1.3,0.2);
		\draw[thin,>-stealth,->](-1.4,0.7)--+(-1.2,0.55);
		\draw[thin,>-stealth,->](-1.4,-0.7)--+(-1.2,-0.555);

		\draw[fill    ] (0,0)circle(2pt);
		\draw[fill   ] (1.5,0)circle(2pt);
		
		\draw[fill    ] (-1.4,0.7)circle(2pt);
		\draw[fill   ] (-1.4,-0.7)circle(2pt);
		
		\node(a)at(-0.2,0.3){ $-5$};
		\node(a)at(0.4,0.3){$[12]$};
		\node(a)at(0,-0.3){ $(1)$};
		\node(a)at(0,-0.7){ $v_1$};
		
		\node(a)at(1.5,-0.3){ (1)};
		\node(a)at(1.5,-0.7){$v_2$};
		\node(a)at(1.5,0.3){ $-2$};
		
		\node(a)at(-1.4,1){ $-1$};
		\node(a)at(-1.4,-1){ $-1$};
		
		\node(a)at(-1.4,0.4){ $(2)$};
		\node(a)at(-1.8,0.6){ $v_3$};
		
		\node(a)at(-1.4,-0.4){$(2)$};
		\node(a)at(-1.8,-0.6){ $v_4$};
		
		\node(a)at(0.8,1.3){ $t\neq0$};

		\begin{scope}[xshift=4.5cm]
		\node(a)at(1.0,1.3){ $t=0$};
		\draw[thin ](0,0)--(4.5,0);
		
		\draw[thin,>-stealth,->](1.5,0)--+(1.1,0.7);
		
		\draw[fill   ] (0,0)circle(2pt);
		\draw[fill   ] (1.5,0)circle(2pt);
		\draw[fill  ] (3,0)circle(2pt);
		\draw[fill   ] (4.5,0)circle(2pt);
		
		\node(a)at(-0.35,0.3){ $-5$};
		\node(a)at(0.25,0.3){{$[12]$}};
		\node(a)at(0,-0.3){ $(1)$};
		\node(a)at(0,-0.7){ $v_1$};

		\node(a)at(1.5,0.3){ $-1$};
		\node(a)at(1.5,-0.3){$(5)$};
		\node(a)at(1.5,-0.7){$v_4$};
		
		\node(a)at(3,0.3){ $-2$};
		\node(a)at(3.0,-0.3){$(3)$};
		\node(a)at(3.0,-0.7){$v_3$};
		
		\node(a)at(04.5,0.3){ $-3$};
		\node(a)at(4.5,-0.3){$(1)$};
		\node(a)at(4.5,-0.7){$v_2$};
		
		\end{scope}

		\end{tikzpicture} 
		\caption{Dual graphs of the minimal resolutions factoring through the blowup of the maximal ideal for the generic element ($t\neq0$) and for the central fiber ($t=0$) of the Brian\c con--Speder family.}\label{fig:BS1-L nodes}
	\end{center}
\end{figure}
For each of the two graphs of Figure \ref{fig:BS1-L nodes} we will now examine the possible values for $\underline{P}_{\pi}$ and for the inner rates, by looking at the linear equations coming from the equality $M \underline{a} = \underline{K} +  \underline{L}   -  \underline{P}$. 

Let us treat the case of the graph on the left of Figure~\ref{fig:BS1-L nodes} first. 
At $v_3$ the equation is $-2q_{v_3}+ q_{v_1} =-1+1-p_{\pi}(v_3)$. 
Since $v_3$ is an $\cal L$-node, we have $q_{v_3} = 1$, and so $q_{v_1} =2-p_{\pi}(v_3)$. 
Since $q_{v_1}$ is strictly greater than $1$ and $p_\pi(v_3)$ is a positive integer, this implies that $p_{\pi}(v_3)=0$, and therefore $q_{v_1}=2$. 
By symmetry, we also have $p_{\pi}(v_4)=0$ (and of course $q_{v_4} =1$).  
The equation corresponding to $v_1$ gives $-5 q_{v_1}+2q_{v_3} + 2q_{v_4}  +q_{v_2} =  {25}-p_{\pi}(v_1)$. 
Since $q_{v_3}=q_{v_4}=q_{v_2}=1$  and  $q_{v_1}=2$, we obtain $p_{\pi}(v_1) =30$.  
Finally, the equation for $v_2$ is: $-2q_{v_2} + q_{v_1}=-1+1-p_{\pi}(v_2)$.
Replacing $q_{v_2}=1$ and $q_{v_1}=2$, we then obtain $p_{\pi}(v_2)=0$. 
Therefore, this data of  $\underline{L}$ determines a unique possible value for  $\underline{P}$ and for the inner rates on $\Gamma_{\pi}$ (which corresponds then to that of the singularity   $x^5+z^{15}+y^7z+txy^6=0$ with $t \neq 0$). 
We have obtained  $(q_{v_1}, q_{v_2},q_{v_3} , q_{v_4})= (2,1,1,1)$ and $ \underline{P} = {}^t({30},0,0,0)$

Let us  now treat the graph on the right of Figure~\ref{fig:BS1-L nodes}. 
The equation for $v_1$ is $-5q_{v_1}+5 q_{v_4} = {23} -p_{\pi}(v_1)$. 
Since $q_{v_4}=1$, this gives $ p_{\pi}(v_1)= {18} +5 q_{v_1}$. 
For $v_4$, we get:  $-5 q_{v_4} + q_{v_1} +3q_{v_3} = 1-p_{\pi}(v_4)$, that is $q_{v_1} + 3q_{v_3} =  6- p_{\pi}(v_4)$. 
Since $m_{v_3}  =3$, we have $q_{v_3} \in \frac13 \N$. 
Since $q_{v_3} >1$, we obtain  $q_{v_3}  \geq \frac43$ and since $m_{v_1}  =1$, we have $q_{v_1} \in  \N_{\geq 2}$. 
We then get $  q_{v_1} +3q_{v_3} \geq 6$. 
Since $q_{v_1} + 3q_{v_3} =  6- p_{\pi}(v_4)$, this implies  $q_{v_3}  = \frac43$,  $q_{v_1} =2$ and $ p_{\pi}(v_4)=0$. 
Moreover, $ p_{\pi}(v_1)= {18} +5 q_{v_1} = {28}$. 
Now, the equation for $v_3$ is: $-2.3.q_{v_3} + 5 q_{v_4} + q_{v_2}=- p_{\pi}(v_3)$,   therefore $q_{v_2}= 3- p_{\pi}(v_3)$. 
In the other hand, the last equation is $-3 q_{v_2} +3q_{v_3}  = -1- p_{\pi}(v_2)$, which gives $ p_{\pi}(v_2) = 3 q_{v_2}-5$. 
Since $q_{v_2} \in \N_{>1}$, we obtain two cases: either $p_{\pi}(v_3)=1$, $q_{v_2} =2$ and $p_{\pi}(v_2)=1$, or $ p_{\pi}(v_3)=0$, $q_{v_2} =3$ and $p_{\pi}(v_2)=4$. 
So there are exactly two possible configurations for the inner rates on $\Gamma_{\pi}$ and for the vector $\underline{P}$: 

\noindent
{\bf Case 1.} $(q_{v_1}, q_{v_2},q_{v_3} , q_{v_4})= (2,2,4/3, 1)$ and $ \underline{P} = {}^t({28},1,1,0)$

\noindent
{\bf Case 2.} $(q_{v_1}, q_{v_2},q_{v_3} , q_{v_4})= (2,3,4/3, 1)$ and $ \underline{P} = {}^t({28},4,0,0)$

\noindent
Computing explicitly the polar curve from the equations from the total transform of the three functions $x, y$ and $z$ computed in \cite[Example 15.2]{BirbrairNeumannPichon2014} (as is done in \cite[Example 3.5]{BirbrairNeumannPichon2014}), it is easy to see that the singularity $x^5+z^{15}+y^7z =0$ corresponds to the first case.
We do not know whether the second case is realized by a surface singularity. 
\end{example}

\subsection{Localization of the polar curve of a morphism} 
\label{subsec:jacobian quotients}

Proposition~\ref{proposition:application_linear_algebra} can also be related  to the approach of \cite{Michel2008}, where the following natural question is considered. 
Given a complex surface germ $(X,0)$ and a finite morphism $(f,g) \colon (X,0) \to (\C^2,0)$, let $\pi \colon X_{\pi} \to X$ be the minimal good resolution of the pair $\big(X,(fg)^{-1}(0)\big)$ (that is, $\pi$ is the minimal good resolution of $(X,0)$ such that the total transform $(fg)^{-1}(0)$ is a normal crossing divisor) and let $\Gamma_{\pi, f, g}$ be the dual resolution graph of $\pi$ decorated with two sets of arrows corresponding to the  strict transforms of $f$ and $g$.
Is it then possible to localize the strict transform of the polar curve $\Pi$ of the morphism $(f,g)$ in the decorated graph $\Gamma_{\pi, f, g}$ ?

A partial answer to this question is given in \cite{Michel2008} and goes as follows.
For each vertex $v$ of $\Gamma_{\pi, f, g}$, consider the corresponding Hironaka quotient ${m_v(f)}/{m_v(g)} \in \Q_{>0}$, where $m_v(f)$ and $m_v(g)$ denote the multiplicities of $f$ and $g$ along the corresponding irreducible component $E_v$ of the exceptional divisor of $\pi$ (that is, the valuations $\ord_{E_v}(f)$ and $\ord_{E_v}(g)$). 
For each Hironaka quotient $q\in\Q$, let $E(q)$ be the union of those irreducible components $E_v$ of $\pi^{-1}(0)$ that have Hironaka quotient $q$, and let $E^1(q)\ldots,E^{n_q}(q)$ be the connected components of $E(q)$.
Finally, for every $k=1,\ldots,n_q$, denote by $\Gamma^k(q)$ the union of those components of $\Pi$ whose strict transform via $\pi$ intersects the curve $E^k(q)$.
The curve $\Gamma^k(q)$ is called a {\it bunch} in \cite{Michel2008}.

\begin{theorem} (\cite[Theorem 4.9]{Michel2008})  
The bunch $\Gamma^k(q)$ is not empty if and only if $E^k(q)$ contains a component $E_v$ such that $\chi(E'_v)<0$, where $E'_v$ is the the subset of $E_v$ consisting of the points that are smooth in the total transform $\big((fg) \circ \pi\big)^{-1}(0)$.
Moreover, 
\[
V_f\big(\Gamma^k(q)\big) = - \sum_{E_v \subset E^k(q)} m_f(v) \chi(E'_v),
\]
where $V_f(\gamma)$ denotes the multiplicity of $f$ relatively to a curve germ $(\gamma,0) \subset (X,0)$ which has no components in common with $f=0$ (see \cite[Definition 7.1]{Michel2008}).
\end{theorem}

	Now, let us consider the case where $f,g$ is a pair of generic linear forms, so that $\ell=(f,g) \colon (X,0) \to (\C^2,0)$ is a generic projection of $(X,0)$. 
	Then $\pi$ is the minimal good resolution of $(X,0)$ which factors through the blowup of the maximal ideal and the data of the graph $\Gamma_{\pi, f, g}$ is equivalent to that consisting of the dual graph $\Gamma_{\pi}$ and of the vector $\underline{L}$.
	In this case, the Hironaka quotient ${m_v(f)}/{m_v(g)}$ is equal to $1$ at each vertex $v$ of $\Gamma_\pi$ and therefore there is only one bunch, the whole polar curve $\Pi$ of the projection $\ell$. 
	In this case, the result \cite[Theorem 4.9]{Michel2008} says nothing more than the L\^e--Greuel--Teissier formula stated in Proposition~\ref{proposition:Le-Greuel}. 
	However, our Proposition~\ref{proposition:application_linear_algebra} gives strong restrictions on the possible localization of the polar curve $\Pi$.
	More precisely, it decomposes the polar curve $\Pi$ into smaller bunches $\Pi_v$ than in \cite{Michel2008}, one for each vertex $v$ of $\Gamma_{\pi}$, and it  gives strong restrictions on the possible values of $p_{\pi}(v)$, as illustrated in Example~\ref{example:BS}.  
	Observe that in this case we have $V_f(\Pi_v) = \sum_{\delta \subset \Pi_v} \mult_0(\delta)$, and therefore $p_{\pi}(v)$ is nothing else than $V_f(\Pi_v)/m_v$.

\bibliographystyle{alpha}                              
\bibliography{biblio}

\vfill

\end{document}